\documentclass[10pt]{article}

\usepackage[a4paper, top=2cm, bottom=2cm, left=2cm, right=2cm]{geometry}
\usepackage{stmaryrd}
\usepackage{lipsum}
\usepackage{amsfonts, amssymb, amsmath, mathtools}
\usepackage{physics}
\usepackage{graphicx, epstopdf, float, booktabs, array, siunitx}
\usepackage{xcolor, caption, subcaption}
\usepackage{tikz, tikz-3dplot}
\usetikzlibrary{patterns}
\usepackage[utf8]{inputenc}
\usepackage[T1]{fontenc}
\usepackage[colorlinks=true,    
            linkcolor=blue,     
            citecolor=cyan,    
            urlcolor=red]       
           {hyperref}
\usepackage{systeme, enumerate, xfrac}

\usepackage{amsopn}
\usepackage{algorithm}
\usepackage{algorithmic}  

\newcommand{\gdata} {g} 
\newcommand{\fmeas} {f} 
\newcommand{\biform} {b} 
\newcommand{\bzed} {z}
\newcommand{\bx} {x}
\newcommand{\by} {y}
\newcommand{\nn} {n}
\newcommand{\EE} {\mathbb{E}}
\newcommand{\YY}{{Y}}
\newcommand{\Proj}{{\Pi}}
\newcommand{\ru} {u_{r}}
\newcommand{\iu} {u_{i}}
\newcommand{\rv} {v_{r}}
\newcommand{\iv} {v_{i}}

\newcommand{\linv} {\psi} 
\newcommand{\noise}{\epsilon}                      
\newcommand{\noisei} {{{\noise}_{i}}}
\newcommand{\run}{u_{r}^{\noise}}
\newcommand{\iun}{u_{i}^{\noise}}
\newcommand{\rvn}{v_{r}^{\noise}}
\newcommand{\ivn}{v_{i}^{\noise}}

\newcommand{\one}[1]{#1^{1}} 
\newcommand{\two}[1]{#1^{2}} 

\newcommand{\dx} {\mathrm{d}x}

\newcommand{\dss} {\mathrm{d}s}

\newcommand{\dn}{\partial_\nu}
\newcommand{\omegastar}{\omega}
\newcommand{\omegastarplus}{\omega_{+}}
\newcommand{\omegastarminus}{\omega_{-}}
\newcommand{\vect}[1]{#1} 
\newcommand{\uprime}{u^{\prime}}
\newcommand{\zhat}{\tilde{z}}
\newcommand{\smallo}{\mathrm{o}}

\newcommand{\adU}{\mathcal{U}}

\newcommand{\HH}{H^{1}{(\varOmega)}}
\newcommand{\JJ}{{J}}
\newcommand{\sfTheta}{{\Theta}}
\newcommand{\VV}{{\theta}}
\newcommand{\Vn}{{\theta}_{\nu}}
\newcommand{\JJone}{{J}_{1}}
\newcommand{\JJtwo}{{J}_{2}}
\newcommand{\Ball}{{B}}
\newcommand{\muzeps}{{\mu}_{\xi}^{\varepsilon}}
\newcommand{\varmuzeps}{\tilde{\mu}^{\varepsilon}} 
\newcommand{\intO}[1]{\int_{\varOmega} #1\, \dx}     
\newcommand{\inS}[1]{\langle{#1}\rangle}

\def\ii{\boldsymbol{i}}
\def\uu{{u}}

\def\v{{v}}
\def\w{{w}}

\usepackage{amsthm}
\newtheorem{theorem}{Theorem}[section]

\newtheorem{problem}[theorem]{Problem}
\newtheorem{remark}[theorem]{Remark}

\newtheorem{lemma}[theorem]{Lemma}
\newtheorem{proposition}[theorem]{Proposition}
\newtheorem{corollary}[theorem]{Corollary}

\title{Statistical Topological Gradient and Shape Optimization for Robust Metal--Semiconductor Contact Reconstruction
}

\author{
L.~Afraites\thanks{EMI, FST, Universit\'{e} Sultan Moulay Slimane, B\'{e}ni-Mellal, Morocco (\texttt{l.afraites@usms.ma})}
\and A.~Hadri\thanks{Faculté Polydisciplinaire, Université Ibn Zohr, Agadir, Morocco (\texttt{aissamhadri20@gmail.com})}
\and M.~Hrizi\thanks{Monastir University, Department of Mathematics, Monastir, Tunisia (\texttt{mourad-hrizi@hotmail.fr})}
\and J.~F.~T.~Rabago\thanks{Faculty of Mathematics and Physics, Kanazawa University, Japan (\texttt{jfrabago@gmail.com})}
}

\date{} 

\begin{document}

\maketitle

\begin{abstract}
We develop a statistically robust framework for reconstructing metal--semiconductor contact regions using topological gradients. 
The inverse problem is formulated as the identification of an unknown contact region from boundary measurements governed by an elliptic model with piecewise coefficients. 
Deterministic stability of the topological gradient with respect to measurement noise is established, and the analysis is extended to a statistical setting with multiple independent observations. 
A central limit theorem in a separable Hilbert space is proved for the empirical topological gradient, yielding optimal $n^{-1/2}$ convergence and enabling the construction of confidence intervals and hypothesis tests for contact detection. 
To further refine the reconstruction, a shape optimization procedure is employed, where the free parameter $\beta$ in the CCBM formulation plays a crucial role in controlling interface sensitivity. 
While $\beta$ affects both topological and shape reconstructions, its influence is particularly pronounced in the shape optimization stage, allowing more accurate estimation of the size and geometry of the contact subregion. 
The proposed approach provides a rigorous criterion for distinguishing true structural features from noise-induced artifacts, and numerical experiments demonstrate the robustness, precision, and enhanced performance of the combined statistical, topological, and $\beta$-informed shape-based reconstruction.
\end{abstract}



\tableofcontents


\section{Introduction}
\subsection{The contact resistivity problem}
In very-large-scale integration (VLSI) circuits (Figure~\ref{fig:VLSI1}), electrical contacts form the critical interfaces between semiconductor devices and metallic interconnects. 
As device dimensions shrink into the deep-submicron regime, parasitic contact resistance increasingly limits overall electrical performance (Figure~\ref{fig:VLSI2}). In its simplest approximation, the contact resistance is given by
\[
R_{\mu} =  \frac{\rho_{\mu}}{A},
\]
where $A$ is the contact area and $\rho_{\mu}$ is the specific contact resistivity. 
Since $A$ typically scales with the square of the minimum feature size $\lambda$, this model predicts the ideal scaling law $R_{\mu} \propto \lambda^{-2}$ (see, e.g., \cite{loh1985analysis,loh1987modeling}).

In practice, this ideal scaling fails due to non-uniform current injection near the contact edges, a phenomenon known as \emph{current crowding}. 
The spatial extent of current spreading is characterized by the \emph{transfer length} $l_t$. When the lateral dimensions of the contact become comparable to or smaller than $l_t$, current crowding intensifies, causing the effective contact resistance to increase more rapidly than predicted by geometric scaling. 
This effect ultimately sets a fundamental physical limit on contact miniaturization in advanced integrated technologies. 
Figure~\ref{fig:VLSI2} schematically illustrates a simplified metal--semiconductor contact structure.
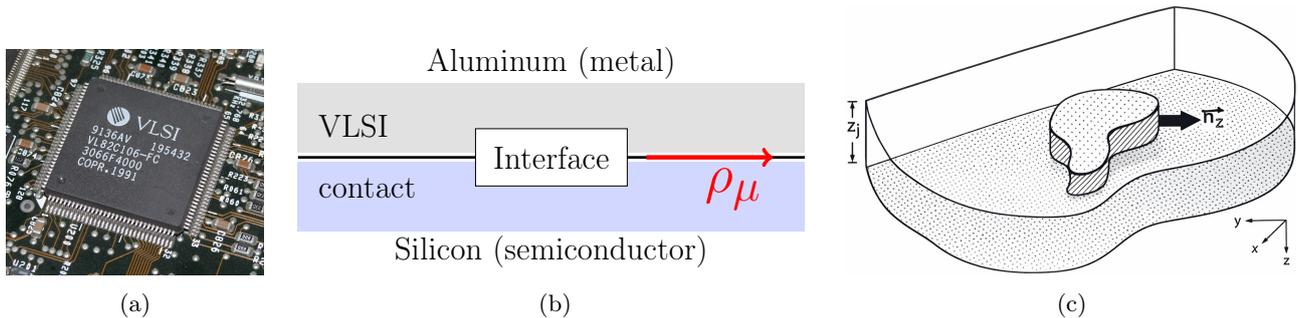
\begin{figure}[htbp!]
  \centering
  \begin{subfigure}{0.2\textwidth}
    \resizebox{\textwidth}{!}{%
	    \includegraphics{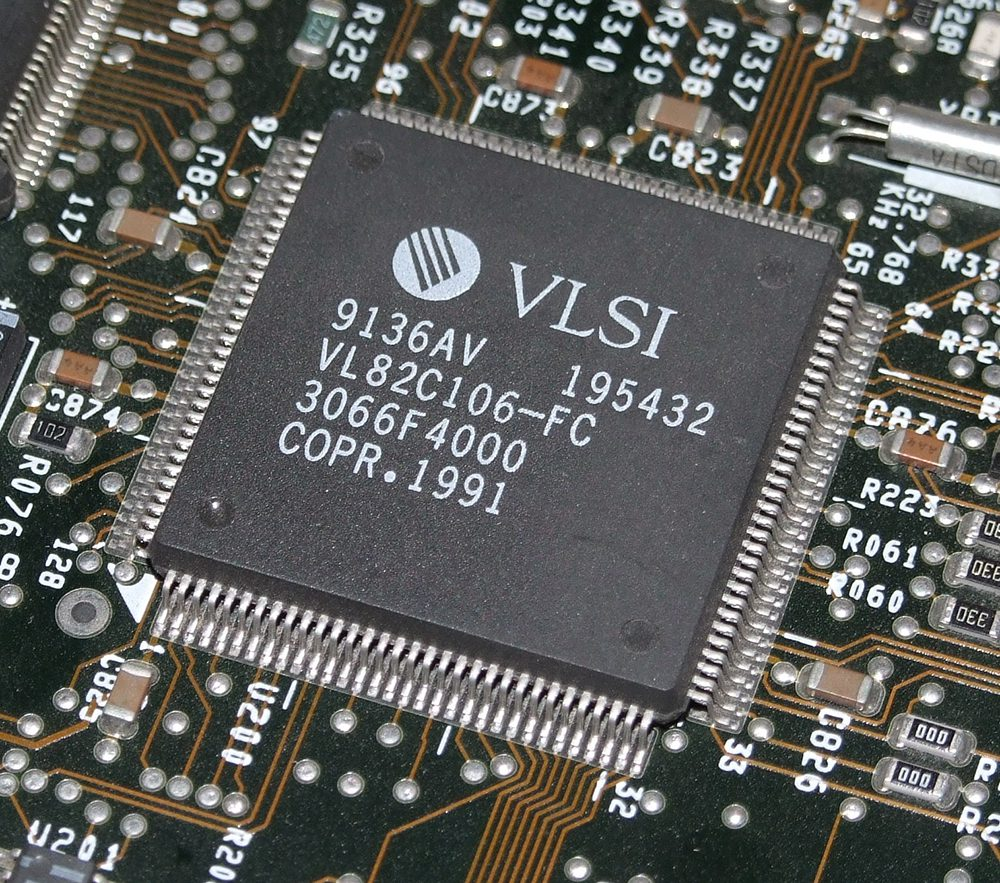}
    }
    \caption{}\label{fig:VLSI1}
  \end{subfigure}
	\hfill
  \begin{subfigure}{0.4\textwidth}
        	\resizebox{\textwidth}{!}{%
        \begin{tikzpicture}[x=1cm,y=1cm]
        \definecolor{metalgray}{RGB}{225,225,225}
        \definecolor{siliconblue}{RGB}{210,215,255}
        
        \def\W{10}
        \def\Hmetal{1.1}
        \def\Hsi{1.1}
        \def\gap{0.15} 
        
        \fill[metalgray] (2,\gap/2) rectangle (\W,\gap/2+\Hmetal);
        
        \fill[siliconblue] (2,-\Hsi-\gap/2) rectangle (\W,-\gap/2);
        
        \draw[line width=1.2pt] (2,0) -- (\W,0);
        
        \node[
          draw,
          fill=white,
          inner sep=8pt
        ] at (6,0) {\Large Interface};
        
        \node at (6,1.5) {\Large Aluminum (metal)};
        \node at (6,-1.5) {\Large Silicon (semiconductor)};
        
        \node[anchor=south west] at (2.2,0.2) {\Large VLSI};
        \node[anchor=north west] at (2.2,-0.2) {\Large contact};
        
        \draw[->,red,line width=2pt] (7.5,0) -- (9.5,0);
        \node[red,anchor=west] at (8.3,-0.5) {{\Huge $\rho_{\mu}$}};
        
        \end{tikzpicture}
        }%
    \caption{}\label{fig:VLSI2}
  \end{subfigure}
	\hfill
  \begin{subfigure}{0.35\textwidth}
    \resizebox{\textwidth}{!}{%
	    \includegraphics{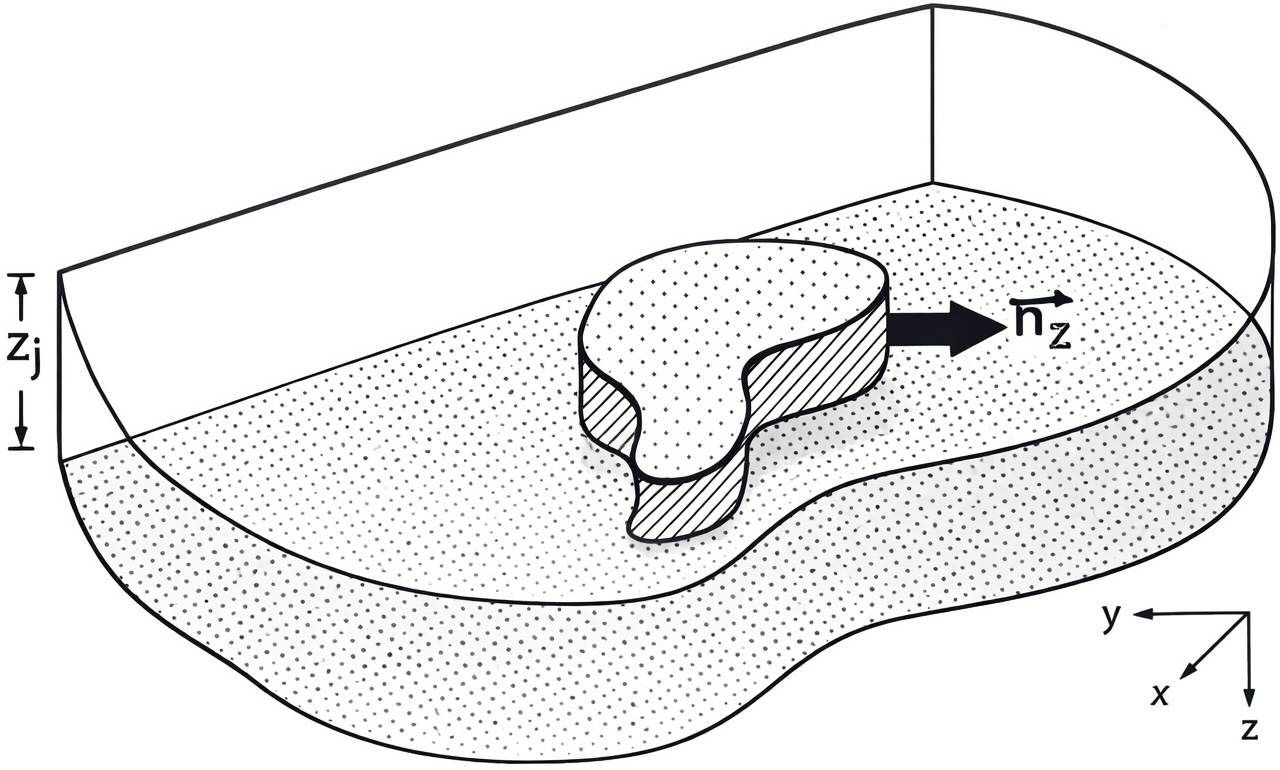}
    }
    \caption{}  \label{fig:model_reduction}
  \end{subfigure}
  \caption{(a) A VLSI circuit; (b) simplified metal/semiconductor contact; (c) a 2D slab contact model with interface at $\bzed=0$ and diffusion-layer thickness $\bzed_j$. Adapted from \cite{loh1987modeling}.}
  \label{fig:VLSI}
\end{figure}
In deeply scaled integrated circuits, contact resistivity governs device performance, making its accurate characterization essential. Yet direct measurement is fundamentally constrained by the buried nature of the contact interface is buried within the device structure, precluding physical access for local probing. Consequently, researchers employ indirect methods: prescribed boundary currents and measured potentials provide the data for inferring internal contact properties. This inverse problem, recovering internal parameters from boundary observations, has been extensively studied, with seminal works establishing the mathematical foundation for contact resistivity determination \cite{loh1985analysis,loh1987modeling}.
Related inverse problems for semiconductor devices can be found in \cite{BurgerEnglMarkowichPietra2001,BusenbergFang1991,FangItoRedfern2002} and the references therein.

\subsection{Mathematical formulation}
Although the physical metal--semiconductor contact system is inherently three-dimensional (3D), the extremely shallow junction depths encountered in modern VLSI processes justify an effective two-dimensional (2D) reduction. 
As illustrated in Figure~\ref{fig:model_reduction}, the contact interface is idealized as a planar surface orthogonal to the $\bzed$-axis, while the semiconductor diffusion region of thickness $\bzed_j$ exhibits a depth-dependent conductivity profile $\sigma=\sigma(\bzed)$ that is laterally uniform in the $(\bx,\by)$-plane. 
This assumption is well supported by contemporary fabrication techniques, where implantation and diffusion produce nearly planar doping distributions across contact windows.

Following the thin-film reduction of Loh et al.~\cite{loh1987modeling}, we consider the asymptotic regime in which the vertical-to-lateral aspect ratio tends to zero. 
Vertical integration of the 3D current continuity equation yields the effective \emph{sheet resistance}
\[
R_s = \left( \int_{0}^{\bzed_j} \sigma(\bzed)\, \mathrm{d}\bzed \right)^{-1}  \ > 0,
\]
which characterizes lateral current transport. This averaging maps the 3D potential $v(\bx,\by,\bzed)$ to an effective 2D potential
\[
u(\bx,\by)
=
R_s
\int_{0}^{\bzed_j}
\sigma(\bzed)\, v(\bx,\by,\bzed)\, \mathrm{d}\bzed .
\]
The reduced model governs the potential $u=u(\bx,\by)$ in an open, bounded
semiconductor domain $\varOmega\subset\mathbb R^2$, with
$\omegastar \Subset \varOmega$ an open, bounded subdomain representing the
metal--semiconductor contact.
Outside the contact region,
$\varOmega\setminus\overline{\omegastar}$, lateral current spreads without
vertical transfer and $u$ satisfies $\Delta u = 0$ in
$\varOmega\setminus\overline{\omegastar}$, with purely two-dimensional current
density $\mathbf J = -R_s^{-1}\nabla u$. Inside the contact region $\omegastar$,
vertical current injection introduces a reaction term, and $u$ solves
$\Delta u - l_t^{-2} u = 0$ in $\omegastar$, where the transfer length
$l_t=\sqrt{\rho_{\mu}/R_s}$ characterizes the depth of current penetration.

At the contact boundary $\partial\omegastar$, continuity of the potential and
conservation of the normal current flux impose
$u_{|_{\partial\omegastarminus}} = u_{|_{\partial\omegastarplus}}$ and
$R_{sk}^{-1}\,\nn_\bzed\!\cdot\nabla u_{|_{\partial\omegastarminus}}
=
R_s^{-1}\,\nn_\bzed\!\cdot\nabla u_{|_{\partial\omegastarplus}}$,
where $\nn_\bzed$ denotes the unit vector orthogonal to the contact boundary and
tangential to the $(\bx,\by)$-plane (cf. Figure~\ref{fig:model_reduction}), and $R_{sk}$ is the sheet resistance
directly beneath the contact.

Collecting the governing relations, the fully reduced boundary value problem can be written compactly as
\[
-\Delta u + {\mu_{0}}\,\chi_{\omegastar}\,u = 0 \quad \text{in }\varOmega,
\qquad
\partial_\nu u = \gdata \ge 0 \quad \text{on }\partial\varOmega, \qquad (\gdata \not\equiv 0),
\]
where $ {\mu_{0}} \coloneqq 1/l_t^{\,2} = R_s / \rho_{\mu} $, $\chi_{\omegastar}$ is the characteristic function of the contact region,  $\gdata$ denotes the prescribed boundary current density, and $\partial_\nu u = \nabla u \cdot \nu$ represents the outward normal flux, with $\nu$ the unit outward normal to $\partial\varOmega$.
The term $(1/l_t^2)\chi_{\omegastar}u$ accounts for vertical current exchange at the metal-semiconductor interface. The transfer length $l_t$ controls current crowding: during the transition from vertical injection beneath the contact to lateral transport in the semiconductor, current concentrates near the contact edges, leading to strongly non-uniform current densities that dominate the effective contact resistance in modern device geometries.

\subsection{The geometric inverse problem}
The general inverse problem aims to recover both the contact region 
$\omegastar \Subset \varOmega$ and the parameter ${\mu_{0}}>0$ from Cauchy boundary 
data $(\gdata,\fmeas)$. This problem is severely ill-posed, as different pairs 
$(\chi_{\omegastar},{\mu_{0}})$ can produce identical measurements \cite{isakov2006inverse}.  
In this work, we adopt the standard assumption that ${\mu_{0}}$ is known a priori, 
reducing the problem to the geometric identification of the contact region 
$\omegastar$ (or equivalently $\chi_{\omegastar}$) 
from the boundary data. We implicitly assume the existence of an exact contact 
region $\omega^{\star}$ generating the measurements.

Throughout the paper, unless stated otherwise, $\varOmega \subset \mathbb{R}^2$ is a bounded domain with Lipschitz boundary, and $\omegastar \Subset \varOmega$ is a bounded open subdomain with Lipschitz boundary.

\begin{problem}\label{eq:overtermined_problem}
Let ${\mu_{0}}>0$ be a given constant.
Given nontrivial Cauchy boundary data
$(\gdata,\fmeas)\in H^{-\frac{1}{2}}(\partial\varOmega)\times H^{\frac{1}{2}}(\partial\varOmega)$,
find the subdomain $\omegastar$ and a function $u\in \HH$ such that
\begin{equation}\label{eq:overdetermined}
-\Delta u + {\mu_{0}} \chi_{\omegastar} u = 0 \ \text{in } \varOmega,\qquad
\partial_\nu u = \gdata,\quad u = \fmeas \ \text{on } \partial\varOmega.
\end{equation}
\end{problem}

The analysis of this reduced inverse problem naturally involves three closely
related aspects:
\begin{enumerate}
\item[(i)] \textit{Identifiability.}  
A fundamental question is whether the boundary data 
$(\gdata,\fmeas)\in H^{-\frac{1}{2}}(\partial\varOmega)\times H^{\frac{1}{2}}(\partial\varOmega)$ 
uniquely determine the contact region $\omegastar$. Uniqueness has been established for certain classes of admissible subdomains \cite{fang1992inverse,kang2001recovery,kim2002unique}. Notably, Kim~\cite{kim2002unique} proved a global uniqueness result for convex regions.

\item[(ii)] \textit{Stability.}  
Another key issue is the continuous dependence of $\omegastar$ on boundary data perturbations. Local stability estimates exist \cite{choulli2006determination,fang1992inverse}. In particular, in \cite{fang1992inverse}, Fang and Cumberbatch established a stability result with the help of a single-point boundary measurement. On the other hand, Choulli~\cite{choulli2006determination} provides quantitative bounds under small domain perturbations using additional Neumann measurements, illustrating the conditional nature of stability.

\item[(iii)] \textit{Identification.}  
Reconstructing $\omegastar$ numerically is ill-posed, despite uniqueness and conditional stability. A common approach reformulates the problem as
\[
\min_{\omegastar \in \adU} J(\omegastar),
\] 
where $J$ is a data-misfit functional and $\adU$ denotes admissible regions. Since $\adU$ is not a vector space, shape and topology optimization techniques are typically used \cite{fang2003identification,fang2012recovery,hrizi2021reconstruction,marin2004boundary}. In these works, $J$ is either a classical least-squares functional, which uses only boundary measurements, or a Kohn--Vogelius-type functional, which quantifies the discrepancy between two auxiliary boundary value problems: one driven by the measured boundary data and the other by prescribed boundary excitations. 
\end{enumerate}

In this work, we focus on numerically reconstructing $\omegastar$ from boundary data, building on established results on identifiability and stability.  
Our main contribution is a $\beta$-weighted coupled-complex-boundary framework (see~\eqref{eq:CCBM}) that unifies shape and topology optimization and incorporates a statistical topological gradient with $n^{-1/2}$ convergence (Theorem~\ref{thm:clt_topograd}). 
This provides robust initialization and confidence-based detection under noisy measurements (Subsection~\ref{subsec:statistical_detection}), while $\beta$ enhances reconstruction accuracy, capturing the location, size, and geometry of $\omegastar$ and filtering out noise (Subsection~\ref{subsec:effect_of_beta}).

In the context of detection inverse problems \cite{abda2009topological,amstutz2005crack,beretta2017inverse,carpio2021processing,caubet2012localization,chaabane2013topological,fernandez2018noniterative,gangl2022topological,hintermuller2011optimal}, the topological derivative method has been widely  recognized as an efficient and accurate non-iterative optimization tool,  typically requiring no additional regularization to stabilize the  reconstruction process. Although numerical studies in these works demonstrate  robustness with respect to noisy data, a complete theoretical justification  remains limited.  In contrast, within the imaging literature, Ammari et al.~\cite{ammari2013localization,ammari2012stability} provided a rigorous analytical explanation of the effectiveness of the topological derivative for detecting small acoustic anomalies, including stability and resolution analyses in the presence of medium and measurement noise. Motivated by this line of research, the present paper establishes a rigorous statistical stability analysis of the topological gradient based on the central limit theorem in separable Hilbert spaces, thereby providing a theoretical foundation for confidence-based reconstruction under random noise.

\subsection{Notations}
We use the Lebesgue spaces $L^2(\varOmega)$ and $L^2(\partial\varOmega)$, and $H^{\frac{1}{2}}(\partial\varOmega)$ denotes the trace space of $\HH$. Unless stated otherwise, the same notation applies to real- and complex-valued functions; occasionally, $X(\varOmega;\mathbb{C})$ is used to emphasize complex-valued functions.
For complex functions, the $\HH$ inner product is 
$(u,v)_{1,\varOmega} \coloneqq (\nabla u, \nabla v)_{0,\varOmega} + (u,v)_{0,\varOmega}$ with 
$(u,v)_{0,\varOmega} \coloneqq \int_\varOmega u \overline{v} \, \,\dx$, and the norm is $\norm{u}_{1,\varOmega}^2 \coloneqq (u,u)_{1,\varOmega} = \norm{\nabla{u}}_{0,\varOmega}^2 + \norm{u}_{0,\varOmega}^2$. 
Here, it is clear that $\norm{\nabla{u}}_{0,\varOmega}^2 = \norm{\nabla{u}}_{L^{2}(\varOmega)^{d}}^2$ ($d=2$).
The bar $\overline{(\cdot)}$, denoting complex conjugation, is omitted when clear from context.
The duality pairing between $H^{-\frac{1}{2}}(\partial\varOmega)$ and $H^{\frac{1}{2}}(\partial\varOmega)$ is denoted by $\langle \cdot, \cdot \rangle\coloneqq\langle \cdot, \cdot \rangle_{\partial\varOmega}$, coinciding with the $L^2(\partial\varOmega)$ inner product $(\cdot, \cdot)_{0,\partial\varOmega}$ if both arguments lie in $L^2(\partial\varOmega)$. 
We adopt a unified notation for norms of scalar- and vector-valued functions.
For a complex-valued function $\uu$, we write $\uu = \ru +  \ii \, \iu \coloneqq \Re{\uu} + \ii \Im{\uu}$, where $\ii = \sqrt{-1}$, and $\Re{\uu}$ and $\Im{\uu}$ denote the real and imaginary parts, respectively.
Throughout the paper, $C>0$ denotes a generic constant that may change from line to line.
Finally, whenever it is necessary to emphasize the dependence of a variable or expression on certain parameters, such as the coefficient $\mu$ or the boundary data $\fmeas$, we use bracket notation, e.g., $\uu[\mu,\fmeas]$, $\uu[\mu]$, or $\uu[\fmeas]$, depending on the context; this dependence is omitted when no ambiguity arises.

In this paper, the terms \emph{contact region} and \emph{inclusion} will be used interchangeably for convenience in the mathematical and numerical discussion, although caution is advised since their physical meanings differ.
\subsection{The optimization framework}
To address Problem~\ref{eq:overtermined_problem}, we adopt an optimization-based formulation
based on the coupled complex boundary method, or simply CCBM.
Originally introduced in~\cite{Chengetal2014} for inverse source problems, CCBM has
since been developed into a flexible framework for a broad class of inverse and
geometric identification problems.
Its effectiveness has been demonstrated in conductivity
reconstruction~\cite{GongChenHan2017}, parameter identification for elliptic
systems~\cite{ZhengChengGong2020}, and shape reconstruction
problems~\cite{Afraites2022}, as well as in geometric inverse source
formulations~\cite{AfraitesMasnaouiNachaoui2022}. More recent works have extended
CCBM-based formulations to increasingly complex settings, including exterior
Bernoulli-type problems~\cite{Rabago2023b}, free-surface
flows~\cite{RabagoNotsu2024}, inverse Cauchy problems for Stokes
systems~\cite{Ouiassaetal2022}, obstacle identification in viscous
flows~\cite{RabagoAfraitesNotsu2025}, tumor
localization~\cite{Rabago2025}, source approximation in time-fractional partial
differential equations~\cite{HriziPrakashNovotny2025}, and bioluminescence
tomography~\cite{WuGongGongZhangZhu2025}.

Geometric inverse problems arising in semiconductor device modeling have been
extensively studied in \cite{fang1992inverse,fang2003identification,kang2001recovery,kim2002unique}.
These works focus on the identification of metal--semiconductor contact regions
and contact resistivity parameters from boundary voltage and current
measurements, motivated by applications to transistor modeling. Fundamental
results on identifiability and conditional stability were established under
appropriate geometric and regularity assumptions on the contact region.

The present work revisits the semiconductor contact identification problem
within the CCBM framework. The objective is to exploit the favorable stability
and accuracy properties reported in earlier CCBM-based studies, particularly in
free-boundary and free-surface problems~\cite{Rabago2023b,RabagoNotsu2024}, while
maintaining computational costs comparable to those of Kohn--Vogelius-type
formulations~\cite{Rabago2023b,RabagoNotsu2024}. 

Following the CCBM philosophy, the overdetermined boundary conditions in Problem~\ref{eq:overtermined_problem} are replaced by a single complex-valued Robin condition. 
For a fixed tuning parameter $\beta \in \mathbb{R}^{+}$, we consider the complex-valued problem
\begin{equation}\label{eq:CCBM}
-\Delta \uu + {\mu_{0}} \chi_{\omegastar} \, \uu = 0 \quad \text{in } \varOmega, \qquad
\partial_\nu \uu + \ii\, \beta \, \uu = \gdata + \ii\, \beta \, \fmeas \quad \text{on } \partial\varOmega.
\end{equation}
Writing $\uu = \ru + \ii \, \iu$, one observes that Problem~\ref{eq:overtermined_problem} is satisfied if and only if $\iu = 0$ in $\varOmega$. 
This motivates the reformulation of the inverse problem as the identification of a coefficient ${\mu} =  {\mu_{0}} \chi_\omega$ such that the corresponding solution satisfies $ \iu[{\mu}] = \Im\{\uu[{\mu}]\} =  0$ in $\varOmega$ (Problem~\ref{eq:main_problem}).

\begin{remark}
The introduction of the free parameter $\beta$ in \eqref{eq:CCBM} represents a novel extension of the conventional CCBM formulation. 
While a similar parameter has appeared in previous work \cite{GongChengHan2016} for convergence analysis under a source condition, here $\beta$ is employed for a different purpose: it provides additional control over interface sensitivity and significantly enhances shape reconstruction, particularly under noisy measurements, as illustrated in Section~\ref{subsec:effect_of_beta}.
\end{remark}

Accordingly, we introduce the admissible set of coefficients $\mu$, denoted by $\adU$.
Let $\omega$ be a Lebesgue measurable set.
We define the admissible set by
\[
\adU \coloneqq \Bigl\{ {\mu} \in L^\infty(\varOmega) \ \Big| \
{\mu} =  {\mu_{0}} \chi_\omega, \ \omega \Subset \varOmega, \ \abs{\omega}>0 \Bigr\}
\subset L^\infty_{+}(\varOmega),
\qquad \mu_0>0,
\]
where $\abs{\omega}=\operatorname{meas}(\omega)$ denotes the Lebesgue measure of $\omega$. 

%
\begin{problem}\label{prob:CCBM_weak}
Given $\mu \in \adU$, the weak formulation of \eqref{eq:CCBM} reads: find $\uu \in \HH$ such that
\begin{equation}\label{eq:CCBM_weak}
a(\uu, \v) = F(\v), \quad \forall \v \in \HH,
\end{equation}
where the sesquilinear form $a(\cdot,\cdot)$ and the linear form $F(\cdot)$ are given by
\[
a(\w, \v) = (\nabla \w, \nabla \v)_{0,\varOmega} + (\mu \, \w, \v)_{0,\varOmega} + \ii \, \beta \, \inS{\w , \v}, \qquad
F(\v) = \inS{\gdata , \v} + \ii \, \beta \, \inS{\fmeas , \v}.
\]
\end{problem}

The following lemma is crucial for establishing the coercivity of $a(\cdot,\cdot)$, which in turn ensures the existence and uniqueness of the solution to Problem~\ref{prob:CCBM_weak}, as stated in Proposition~\ref{prop:CCBM_wellposedness}.
The proof relies on the complex Lax--Milgram lemma \cite[p.~376]{DautrayLionsv21998} (see also \cite[Lem.~2.1.51, p.~40]{SauterSchwab2011}).

\begin{lemma}\label{lem:norm_estimate}
Let $\varOmega \subset \mathbb{R}^{d}$, $d \ge 2$, be a bounded Lipschitz domain, and let $\omega \Subset \varOmega$ be measurable with $\abs{\omega}>0$. Then there exists a constant $C(\varOmega) > 0$ such that for all $\varphi \in \HH$,
\[
\norm{\varphi}_{1,\varOmega}^2 \le C(\varOmega) \norm{\varphi}_{\ast}^{2},
\qquad
\norm{\varphi}_{\ast} \coloneqq \Big( \norm{\nabla \varphi}_{0, \varOmega}^2 + \norm{\varphi}_{0, \omega}^2 \Big)^{\frac{1}{2}}.
\]
\end{lemma}
The proof is provided in Appendix~\ref{appx:norm_estimate}.
\begin{proposition}\label{prop:CCBM_wellposedness}
Let $\mu \in \adU$ and $(\gdata, \fmeas) \in H^{-\frac{1}{2}}(\partial\varOmega) \times H^{\frac{1}{2}}(\partial\varOmega)$ be given nontrivial Cauchy data. 
Then Problem~\ref{prob:CCBM_weak} admits a unique weak solution $\uu=\uu[\mu] \in \HH$ satisfying
\begin{equation}\label{eq:stability_estimate}
\norm{\uu}_{1,\varOmega} \le C \norm{(g,f)}_{\frac{1}{2},\partial\varOmega}, \quad \norm{(g,f)}_{\frac{1}{2},\partial\varOmega}\coloneqq\norm{\gdata}_{H^{-\frac{1}{2}}(\partial\varOmega)} + \norm{\fmeas}_{H^{\frac{1}{2}}(\partial\varOmega)},
\end{equation}
where $C > 0$ depends only on $\varOmega$.
\end{proposition}

\begin{proof}
The continuity of the sesquilinear form $a(\cdot,\cdot)$ and the linear form $F(\cdot)$ is standard and omitted. We focus on the crucial coercivity of $a(\cdot,\cdot)$.

Let $\uu = \ru + \ii\, \iu \in \HH$. By definition of $a$ and because $\mu = \mu_{0} \, \chi_{\omegastar} \in \adU$, we have
\begin{align*}
\Re\{a(\uu,\uu)\} 
&= (\nabla \ru, \nabla \ru)_{0,\varOmega} + (\mu \, \ru, \ru)_{0,\varOmega} 
 + (\nabla \iu, \nabla \iu)_{0,\varOmega} + (\mu \, \iu, \iu)_{0,\varOmega} \\
&= \big( \norm{\nabla \ru}_{0,\varOmega}^2 + \mu_{0} \, \norm{\ru}_{0,\omegastar}^2 \big)
 + \big( \norm{\nabla \iu}_{0,\varOmega}^2 + \mu_{0} \, \norm{\iu}_{0,\omegastar}^2 \big).
\end{align*}

Applying Lemma~\ref{lem:norm_estimate}, there exists a constant $C(\varOmega) > 0$ such that
\begin{equation}\label{eq:coercivity_estimate}
\Re\{a(\uu,\uu)\} \ge \frac{\min\{1,\mu_{0}\}}{C(\varOmega)} \Big( \norm{\ru}_{1,\varOmega}^2 + \norm{\iu}_{1,\varOmega}^2 \Big)
\equiv \frac{\min\{1,\mu_{0}\}}{C(\varOmega)} \, \norm{\uu}_{1,\varOmega}^2.
\end{equation}

Hence, the sesquilinear form $a(\cdot,\cdot)$ is $\HH$-elliptic. Together with the standard continuity of $a$ and $F$, this guarantees the unique solvability of Problem~\ref{prob:CCBM_weak} and the stability estimate \eqref{eq:stability_estimate}.
\end{proof}

The well-posedness of Problem~\ref{prob:CCBM_weak} ensures that the mapping 
$\mu \mapsto \uu[\mu]$ from $\adU$ to $\HH$ is well-defined. 
This allows us to formulate the following identification problem.
\begin{problem}\label{eq:main_problem}
Let $(\gdata,\fmeas) \in H^{-\frac{1}{2}}(\partial\varOmega) \times H^{\frac{1}{2}}(\partial\varOmega)$ be given nontrivial Cauchy boundary data.
Find ${\mu} \in \adU$ such that the unique weak solution $\uu=\uu[{\mu}] \in H^{1}(\varOmega;\mathbb{C})$ of Problem~\ref{prob:CCBM_weak}
satisfies
\[
	\iu[{\mu}] = 0 \quad \text{in } \varOmega.
\]

Equivalently, this can be cast as the shape optimization problem
\begin{equation}\label{eq:shape_optimization_problem}
	\min_{{\mu} \in \adU} \JJ(\mu) \coloneqq \min_{{\mu} \in \adU}\, (\iu[{\mu}] , \iu[{\mu}] )_{0,\varOmega}
\qquad \text{subject to \quad \eqref{eq:CCBM_weak}}.
\end{equation}
\end{problem} 

To solve this optimization problem \eqref{eq:shape_optimization_problem}, we shall use a
topological sensitivity analysis method \cite{ammari2004reconstruction,baumann2024complete,garreau2001topological,novotny2012topological,sokolowski1999topological}. This method studies how the cost functional $\JJ$ changes when a small topological perturbation of the resistivity is introduced.

\section{Topological Sensitivity Analysis}\label{sec:TSA}
In this section, we derive the topological sensitivity of the functional $\JJ$ (Theorem~\ref{thm:topological_gradient}) with respect to infinitesimal perturbations of the resistivity domain. 
For a given $\mu \in \adU$, we introduce a small inclusion, leading to the perturbed resistivity
\begin{equation}\label{eq:mu_perturbation}
\muzeps(x) =
\begin{cases}
\mu_{0}, & x \in \Ball_{\xi}^{\varepsilon},\\
0, & x \in \varOmega \setminus \overline{\Ball_{\xi}^{\varepsilon}},
\end{cases}
\end{equation}
where $\Ball_{\xi}^{\varepsilon} = \xi + \varepsilon \, \Ball$, $ \xi \in \varOmega$, $0 < \varepsilon \ll 1$, and $\Ball \subset \mathbb{R}^2$ is a bounded open reference domain containing the origin.  
Hereafter, we assume $\mu \in \adU$ without further notice.  

Our goal is to obtain an asymptotic expansion of the perturbed functional:
\begin{equation*}
\JJ(\mu + \muzeps) = \JJ(\mu) + \phi(\varepsilon)\, \delta\JJ[\mu](\xi) + \smallo(\phi(\varepsilon)), \quad \forall \xi \in \varOmega,
\end{equation*}
where $\phi:\mathbb{R}^+ \to \mathbb{R}^+$ is a scaling function vanishing as $\varepsilon \to 0$, and $\delta\JJ[\mu]$ denotes the topological gradient. For brevity, we occasionally write $\varmuzeps = \mu + \muzeps$ and omit the explicit dependence on $\xi$ or $\mu$ when the context is clear.

From \eqref{eq:shape_optimization_problem}, the perturbed functional reads
\begin{equation*}
\JJ(\varmuzeps) = (\iu[\varmuzeps] , \iu[\varmuzeps] )_{0,\varOmega},
\end{equation*}
with $\uu[\varmuzeps] = \ru[\varmuzeps] +  \ii \, \iu[\varmuzeps] \in \HH$ solving
\begin{equation}\label{eq:perturbed_problem}
-\Delta \uu[\varmuzeps] + \varmuzeps\, \uu[\varmuzeps] = 0\ \text{ in } \varOmega,\qquad
\partial_\nu \uu[\varmuzeps] +  \ii \,\beta \uu[\varmuzeps] = \gdata +  \ii \,\beta \fmeas \ \text{ on } \partial\varOmega.
\end{equation}

\begin{remark}
To minimize the cost functional $\JJ$, the optimal location of the resistivity perturbation $\mu^\varepsilon_\xi$ in $\Omega$ is characterized by the sign of the topological gradient $\delta\JJ[\mu]$. More precisely, regions where $\delta\JJ[\mu](\xi) < 0$ indicate descent directions of the functional. Indeed, if $\delta\JJ[\mu](\xi) < 0$, then for sufficiently small $\varepsilon > 0$,
\[
\JJ(\mu + \mu^\varepsilon_\xi) < \JJ(\mu),
\]
which implies that introducing a small inclusion at $\xi$ decreases the cost functional.
\end{remark}

The main result of this section is stated in the following theorem, which provides the expression of the topological gradient.
\begin{theorem}\label{thm:topological_gradient}
Let $\uu{} \in \HH$ solve Problem~\ref{prob:CCBM_weak}, and let $\v{} \in \HH$ solve the adjoint problem
\begin{equation}\label{eq:adjoint}
-\Delta \v{} + \mu\,\v{} = -2\,\iu{} \quad \text{in } \varOmega, \qquad
\partial_\nu \v{}-\ii\,\beta\,\v{} = 0 \quad \text{on } \partial\varOmega.
\end{equation}
For the perturbed resistivity \eqref{eq:mu_perturbation}, the functional admits the asymptotic expansion
\begin{equation}\label{eq:topological_gradient_expansion} 
\JJ(\varmuzeps) = \JJ(\mu) + \varepsilon^2 \mu_0 |\Ball| \, \delta \JJ[\mu](\xi) + o(\varepsilon^2), \quad \varepsilon \to 0,
\end{equation}
where $\xi \in \varOmega$ is the center of the small inclusion, and
\begin{equation}\label{eq:topological_gradient}
\delta \JJ[\mu](x) \coloneqq \iu[\mu](x) \, \rv[\mu](x) - \ru[\mu](x) \, \iv[\mu](x), \quad x \in \varOmega,
\end{equation}
defines the \emph{topological gradient} of $\JJ$.
\end{theorem}

To identify the scaling $\phi(\varepsilon)$ and the topological gradient $\delta\JJ[\mu]$, we decompose the variation as
\begin{equation}\label{eq:variation_decomposition}
\JJ(\varmuzeps)-\JJ(\mu)
=
\underbrace{\int_\varOmega \big( \iu[\varmuzeps]-\iu[\mu] \big)^2}_{:=\JJone(\varepsilon)}
+
\underbrace{2\int_\varOmega \iu[\mu]\big( \iu[\varmuzeps]-\iu[\mu] \big)}_{:=\JJtwo(\varepsilon)}.
\end{equation}
Estimating $\JJone(\varepsilon)$ and $\JJtwo(\varepsilon)$ separately allows us to extract the leading-order term and the topological gradient. 
The conclusion of Theorem~\ref{thm:topological_gradient} then follows from Lemmas~\ref{lem:Jone_estimate} and~\ref{lem:Jtwo_estimate}.

\subsection{Estimate of $\JJone(\varepsilon)$}

The next lemma gives an estimate for the first term of \eqref{eq:variation_decomposition}.
\begin{lemma}\label{lem:Jone_estimate}
Let $\uu[\varmuzeps] \in \HH$ and $\uu{}=\uu[\mu] \in \HH$ solve \eqref{eq:perturbed_problem} and \eqref{eq:CCBM_weak}, respectively.  
There exists a constant $C>0$, independent of $\varepsilon$, such that
\[
\norm{\uu[\varmuzeps] - \uu{}}_{1, \varOmega} \le C \, \varepsilon^{4/3}.
\]
Consequently, $\JJone(\varepsilon)=\smallo(\varepsilon^2)$ as $\varepsilon\to0$.
\end{lemma}

\begin{proof}
Set $\w \coloneqq \uu[\varmuzeps] - \uu{}$. Then $\w$ satisfies
\[
-\Delta \w + \varmuzeps \, \w = -\muzeps \, \uu{}
\quad \text{in } \varOmega,
\qquad
\dn \w + \ii \,\beta \w = 0
\quad \text{on } \partial\varOmega.
\]
Testing with $\w\in\HH$ and applying \eqref{eq:coercivity_estimate}, we obtain
\begin{equation}\label{eq:w_estimate}
\norm{\w}_{1, \varOmega}^2
\le
C \, \Big| \int_\varOmega \muzeps \, \uu{} \, \overline{\w} \, \dx \Big|.
\end{equation}
By H\"{o}lder's inequality, the Sobolev embedding $\HH\subset L^6(\varOmega)$ (valid in 2D; see \cite[Thm.~4.12, p.~85]{AdamsFournier2003}),
and the \textit{a priori estimate} \eqref{eq:stability_estimate}, we have
\begin{align}
\Big| \int_\varOmega \muzeps \, \uu \, \overline{\w} \, \dx \Big|
&\le
\norm{\muzeps}_{L^{\frac{3}{2}}(\varOmega)}
\norm{\uu}_{L^6(\varOmega)}
\norm{\w}_{L^6(\varOmega)} \nonumber\\
&\le
C \, \norm{\muzeps}_{L^{\frac{3}{2}}(\varOmega)}
\norm{\uu}_{1,\varOmega}
\norm{\w}_{1,\varOmega} \nonumber\\
&\le
C \, \norm{\muzeps}_{L^{\frac{3}{2}}(\varOmega)}
\norm{(g,f)}_{\frac12,\partial\varOmega}
\norm{\w}_{1,\varOmega}.\label{eq:Jone_first_estimate}
\end{align}
From \eqref{eq:mu_perturbation} and the identity $|\Ball_\xi^\varepsilon|=\varepsilon^2|\Ball|$, it follows that
$\norm{\mu_\varepsilon}_{L^{\frac{3}{2}}(\varOmega)} \le C\,\varepsilon^{\frac{4}{3}}$.
Substituting this bound into \eqref{eq:Jone_first_estimate}, together with \eqref{eq:w_estimate}, yields the desired estimate,
with a constant $C>0$ depending on $\varOmega$, $\mu_{0}$, $\Ball$, and
$\norm{(g,f)}_{\frac12,\partial\varOmega}$, but independent of $\varepsilon$.
As a consequence, we also have $\JJone(\varepsilon)=\smallo(\varepsilon^2)$ as $\varepsilon\to0$.
\end{proof}
\subsection{Estimate of $\JJtwo(\varepsilon)$}
In this subsection, we prove the following lemma.
\begin{lemma}\label{lem:Jtwo_estimate}
We have 
\begin{equation}\label{eq:Jtwo_estimate}
\JJtwo(\varepsilon)
= \varepsilon^2 \mu_0 \abs{\Ball} 
\Big(
\iu{}(\xi) \rv{}(\xi) 
- \ru{}(\xi) \iv{}(\xi)
\Big)
+ \smallo(\varepsilon^2).
\end{equation}
\end{lemma}
\begin{proof} 
Let ${q}\coloneqq\uu[\varmuzeps]-\uu{}$.
Testing the variational formulation of \eqref{eq:adjoint} with $q$ gives
\begin{equation}\label{eq:equation_for_vq}
(\nabla \v{}, \nabla q)_{0,\varOmega}
+ (\mu\,\v{}, q)_{0,\varOmega}
- \ii (\beta\v{}, q)_{0,\partial\varOmega}
= -2 (\iu{}, q)_{0,\varOmega}.
\end{equation}
The function $q \in \HH$ satisfies
\[
-\Delta q + \mu\,q = -\muzeps\,\uu[\varmuzeps] \quad \text{in } \varOmega, 
\qquad
\partial_\nu q + \ii\,\beta\,q = 0 \quad \text{on } \partial\varOmega,
\]
and testing its variational formulation with $\v{}$ gives
\begin{equation}\label{eq:equation_for_qv}
\underbrace{(\nabla q, \nabla \v{})_{0,\varOmega}
+ (\mu\,q, \v{})_{0,\varOmega}
+ \ii (\beta\,q, \v{})_{0,\partial\varOmega}}_{:=\zhat}
= - (\muzeps\,\uu[\varmuzeps], \v{})_{0,\varOmega}.
\end{equation}

Combining \eqref{eq:equation_for_vq} and \eqref{eq:equation_for_qv} and using 
$\zhat + \overline{\zhat} = 2 \Re\{\zhat\}$, we obtain
\[
2 \Re\{\zhat\} 
= - (\muzeps\,\uu[\varmuzeps], \v{})_{0,\varOmega}
- 2 (\iu{}, q)_{0,\varOmega}.
\]

Writing $\uu=\ru+\ii\,\iu$ and $\v=\rv+\ii\,\iv$, we identify imaginary parts to obtain
\begin{align*}
\JJtwo(\varepsilon)
&= \intO{\muzeps \big( \iu[\varmuzeps]\,\rv{} - \ru[\varmuzeps]\,\iv{} \big)}\\
&= \varepsilon^2 \mu_0 \abs{\Ball} \big( \iu{}(\xi)\,\rv{}(\xi) - \ru{}(\xi)\,\iv{}(\xi) \big) + \sum_{j=1}^4 J_{2,j}(\varepsilon),
\end{align*}
with 
\begin{align*}
J_{2,1}(\varepsilon) &= \intO{\muzeps\, (\iu[\varmuzeps]-\iu{}) \, \rv{}}, &\quad
J_{2,2}(\varepsilon) &= -\intO{\muzeps\, (\ru[\varmuzeps]-\ru{}) \, \iv{}}, \\
J_{2,3}(\varepsilon) &= \intO{\muzeps\, (\iu{}-\iu{}(\xi)) \, \rv{}}, &\quad
J_{2,4}(\varepsilon) &= -\intO{\muzeps\, (\ru{}-\ru{}(\xi)) \, \iv{}}.
\end{align*}
Arguing as in the proof of Lemma~\ref{lem:Jone_estimate} and using the regularity of $\uu$ near $\xi$, we obtain
$J_{2,j}(\varepsilon) = o(\varepsilon^2)$ for all $j=1,\dots,4$.
Consequently, the term $\JJtwo(\varepsilon)$ satisfies \eqref{eq:Jtwo_estimate}.
\end{proof}
The conclusion of Theorem~\ref{thm:topological_gradient} is thus established by Lemmas~\ref{lem:Jone_estimate} and~\ref{lem:Jtwo_estimate}.

\section{Stability of the topological gradient under noisy measurements}
\label{sec:stability}

We study the stability of the topological gradient $\delta \JJ$ under noisy Dirichlet boundary measurements. 
The goal is to quantify how perturbations in the measured Dirichlet data affect the reconstructed topological gradient, 
and in particular, to provide rigorous bounds that guarantee robustness of the topological method under realistic measurement errors.
Throughout this section, we omit the explicit dependence on $\mu$, since $\mu$ is fixed and we are concerned only with perturbations of the boundary data $f$.

\subsection{Noisy boundary measurements and perturbed systems} \label{sec:noise-def}
We recall that $\gdata \in H^{-\frac{1}{2}}(\partial\varOmega)$ denotes the exact Neumann boundary data and
$\fmeas \in H^{\frac{1}{2}}(\partial\varOmega)$ the corresponding Dirichlet measurement.
Measurement errors are modeled on a complete probability space
$(\Omega_{\mathrm{prob}},\mathcal{F},\mathbb{P})$
by a centered Gaussian random variable
\[
\noise : \Omega_{\mathrm{prob}} \to H^{\frac{1}{2}}(\partial\varOmega),
\qquad
\noise \in L^2(\Omega_{\mathrm{prob}};H^{\frac{1}{2}}(\partial\varOmega)),
\]
with covariance
\[
\EE[\noise(x)\noise(y)]
=
\sigma_{\noise}^2\,\mathcal{K}(x,y),
\quad x,y\in\partial\varOmega,
\]
where $\EE[\cdot]$ denotes expectation,
$\sigma_{\noise}^2>0$ is the noise intensity,
and $\mathcal{K}$ is a symmetric positive semidefinite kernel.
The noisy measurement is defined by
\[
\fmeas^{\noise}(\eta)
=
\fmeas + \noise(\eta),
\qquad
\mathbb{P}\text{-a.e. } \eta\in\Omega_{\mathrm{prob}}.
\]

For each event $\eta\in\Omega_{\mathrm{prob}},$ the noisy state $\uu^{\noise}:=\uu^{\noise}(\cdot;\eta) \in \HH$ and adjoint $\v^{\noise}=\v^\noise(\cdot;\eta) \in \HH$ are the solutions of
\begin{equation}\label{eq:noisy_state_system}
-\Delta \uu^{\noise} + \mu \, \uu^{\noise} = 0 \quad \text{in } \varOmega, \qquad
\partial_\nu \uu^{\noise} + \ii \, \beta\,\uu^{\noise} = \gdata + \ii \, \beta\fmeas^{\noise} \quad \text{on } \partial\varOmega,
\end{equation}
\begin{equation}\label{eq:noisy_adjoint_system}
-\Delta \v^{\noise} + \mu \, \v^{\noise} = -2\, \iu^{\noise} \quad \text{in } \varOmega, \qquad
\partial_\nu \v^{\noise} - \ii \, \beta \, \v^{\noise} = 0 \quad \text{on } \partial\varOmega.
\end{equation}

We denote the real and imaginary components as $\run \coloneqq \Re(\uu^{\noise})$, $\iun \coloneqq \Im(\uu^{\noise})$, $\rvn \coloneqq \Re(\v^{\noise})$, and $\ivn \coloneqq \Im(\v^{\noise})$,
and the corresponding topological gradient under noise is defined, for
$\mathbb{P}$-a.e. $\eta$, by
\[
\delta \JJ^{\noise}(x) = \iun(x)\, \rvn(x) - \run(x)\, \ivn(x), \quad x \in \varOmega,
\]
so that
\[
\delta\JJ^{\noise}
\in
L^2\!\left(\Omega_{\mathrm{prob}};L^2(\varOmega)\right).
\]

\subsection{Lipschitz stability of the topological gradient}
To establish stability, we first consider the general topological gradient associated with any Dirichlet boundary data ${\fmeas} \in H^{\frac{1}{2}}(\partial\varOmega)$. 

Let $\uu{}\coloneqq\uu[{\fmeas}] \in \HH$ and $\v{}\coloneqq\v[{\fmeas}] \in \HH$ be the solutions of the state and adjoint systems
\begin{align}
-\Delta \uu{} + \mu \, \uu{} 
    &= 0, & 
\partial_\nu \uu{} + \ii \,\beta \uu{} 
    &= \gdata + \ii \,\beta {\fmeas} 
        && \text{on } \partial\varOmega, \label{eq:state_psi_input} \\[0.5em]
-\Delta \v{} + \mu \, \v{} 
    &= -2 \, \iu{}, & 
\partial_\nu \v{} - \ii \, \beta \, \v{} 
    &= 0 
        && \text{on } \partial\varOmega. \label{eq:adjoint_psi_input}
\end{align}
with real and imaginary components $\ru{}, \iu{}, \rv{}, \iv{}$.
The topological gradient is then
\begin{equation*}
\delta \JJ{}(x) = \iu{}(x) \, \rv{}(x) - \ru{}(x) \, \iv{}(x), \quad x \in \varOmega.
\end{equation*}

The next theorem shows that the topological gradient depends Lipschitz continuously on the Dirichlet boundary data, providing a rigorous bound that ensures robustness under noisy measurements (i.e., stability under noisy measurements).

For brevity, given two boundary data ${\fmeas}_1$ and ${\fmeas}_2$, we introduce the notation
\[
\one{X} := X[{\fmeas}_1],
\qquad
\two{X} := X[{\fmeas}_2],
\]
for any quantity $X$ depending on the boundary data.
\begin{theorem}
\label{thm:stab_topograd}
There exists a constant $C = C(\varOmega, \mu, \beta) > 0$ such that, for any ${\fmeas}_1, {\fmeas}_2 \in H^{\frac{1}{2}}(\partial\varOmega)$,
\begin{equation*}
\norm{\delta \one{\JJ} - \delta \two{\JJ}}_{L^2(\varOmega)}
\le C \, \norm{{\fmeas}_1 - {\fmeas}_2}_{H^{\frac{1}{2}}(\partial\varOmega)}.
\end{equation*}
In particular, for noisy measurements,
\begin{equation*}
\EE\Big[\norm{\delta \JJ[\fmeas^{\noise}] - \delta \JJ[\fmeas]}_{L^2(\varOmega)}\Big]
\le C \, \EE\Big[\norm{\noise}_{H^{\frac{1}{2}}(\partial\varOmega)}\Big].
\end{equation*}
\end{theorem}

\begin{remark}[Stability of the topological gradient]
Theorem~\ref{thm:stab_topograd} shows that the mapping $\fmeas \longmapsto \delta{\JJ}[\fmeas]$ is Lipschitz continuous from $H^{\frac{1}{2}}(\partial\varOmega)$ into $L^2(\varOmega)$, i.e.,
\[
\norm{
\delta{\JJ}[\fmeas_1]
-
\delta{\JJ}[\fmeas_2]
}_{L^2(\varOmega)}
\le
C
\norm{
\fmeas_1-\fmeas_2
}_{H^{\frac{1}{2}}(\partial\varOmega)} .
\]
In other words, small perturbations in the Dirichlet boundary measurements affect $\delta \JJ$ at most by the constant factor $C$, guaranteeing both reliable detection of resistive or inclusion domains and numerical stability under noisy measurements.
\end{remark}

To prove Theorem~\ref{thm:stab_topograd}, we first state the Lipschitz continuity of both the forward and adjoint solutions with respect to the Dirichlet boundary data. 
This property follows directly from the continuous dependence of the state on the boundary data (see Proposition~\ref{prop:CCBM_wellposedness}) 
and by a similar argument for the adjoint.
\begin{lemma}
\label{lem:stab_solutions}
There exists a constant $C = C(\varOmega, \mu, \beta) > 0$, independent of the boundary data, such that for any ${\fmeas}_1, {\fmeas}_2 \in H^{\frac{1}{2}}(\partial\varOmega)$, the solutions of \eqref{eq:state_psi_input} and \eqref{eq:adjoint_psi_input} satisfy
\begin{equation*}
\norm{\one{\uu} - \two{\uu}}_{1, \varOmega}
+
\norm{\one{\v} - \two{\v}}_{1, \varOmega}
\le C \, \norm{{\fmeas}_1 - {\fmeas}_2}_{H^{\frac{1}{2}}(\partial\varOmega)}.
\end{equation*}
\end{lemma}

\begin{proof}[Proof of Theorem~\ref{thm:stab_topograd}]
Let $\fmeas^1, \fmeas^2 \in H^{\frac{1}{2}}(\partial\varOmega)$ and denote by
$\uu^m, \v^m \in \HH$, for $m=1,2$, the corresponding state and adjoint solutions, with real and imaginary parts
$\ru^m, \iu^m, \rv^m, \iv^m$. Then
\[
\delta \one{\JJ} - \delta \two{\JJ}
= (\one{\iu}-\two{\iu}) \, \one{\rv}
+ \two{\iu} \, (\one{\rv}-\two{\rv})
- (\one{\ru}-\two{\ru}) \, \one{\iv}
- \two{\ru} \, (\one{\iv}-\two{\iv}).
\]

Applying the H\"{o}lder's inequality with exponent $4$ in two dimensions, the first product term on the right-hand side above can be estimated as
\[
\norm{ (\one{\iu}-\two{\iu}) \, \one{\rv} }_{L^2(\varOmega)}
\le
\norm{\one{\iu}-\two{\iu}}_{L^4(\varOmega)}
\,
\norm{\one{\rv}}_{L^4(\varOmega)}.
\]
Using the Sobolev embedding $\HH \hookrightarrow L^4(\varOmega)$ (valid in 2D; see \cite[Thm.~4.12, p.~85]{AdamsFournier2003}), we obtain
\[ 
\norm{\one{\iu}-\two{\iu}}_{L^4(\varOmega)}
\le C \, \norm{\one{\iu}-\two{\iu}}_{\HH},
\qquad
\norm{\one{\rv}}_{L^4(\varOmega)}
\le C \, \norm{\one{\rv}}_{\HH}.
\]
Hence,
\[
\norm{ (\one{\iu}-\two{\iu}) \, \one{\rv} }_{L^2(\varOmega)}
\le
C
\norm{\one{\iu}-\two{\iu}}_{\HH}
\,
\norm{\one{\rv}}_{\HH}.
\]
The remaining terms in $\delta \one{\JJ}-\delta \two{\JJ}$ can be estimated analogously.

By Lemma~\ref{lem:stab_solutions}, the solutions depend Lipschitz continuously on the boundary data in $\HH$, and remain uniformly bounded. 
Combining these estimates yields
\[
\norm{ \delta \one{\JJ} - \delta \two{\JJ} }_{L^2(\varOmega)}
\le
C \, \norm{ {\fmeas}_1 - {\fmeas}_2 }_{H^{\frac{1}{2}}(\partial\varOmega)},
\]
where $C>0$ depends only on $\varOmega$, $\mu$, $\beta$, and the uniform $\HH$ bounds of the solutions. 
This proves the Lipschitz continuity of the topological gradient.
\end{proof}

\subsection{Statistical stability under random noise}
The deterministic stability estimate (Theorem~\ref{thm:stab_topograd}) ensures that a single noisy measurement produces a perturbed topological gradient with an error proportional to the noise magnitude. 
We now extend this result to a statistical framework to quantify convergence and uncertainty when multiple independent measurements are available; see Theorem~\ref{thm:clt_topograd}.

Let $\{\fmeas_{i}^{\noise}\}_{i=1}^{n} \subset H^{\frac{1}{2}}(\partial\varOmega)$ denote $n$ independent boundary measurements of the Dirichlet data, modeled as
\[
\fmeas_{i}^{\noise} = \fmeas + \noisei,
\]
where $\fmeas$ is the exact (noise--free) datum and 
$\{\noisei\}_{i=1}^{n}$ (as defined in Subsection \ref{sec:noise-def}) are independent and identically distributed (i.i.d.) random fields in $L^2(\Omega_{\mathrm{prob}};H^{\frac{1}{2}}(\partial\varOmega))$ satisfying
\[
\EE\!\left[\noisei\right]=0,
\qquad
\EE\!\left[\norm{\noisei}_{H^{\frac{1}{2}}(\partial\varOmega)}^2\right]<\infty .
\]
For each realization $\fmeas_{i}^{\noise}$, we compute the corresponding topological gradient
$\delta{\JJ}[\fmeas_{i}^{\noise}] \in L^2(\varOmega)$.
The empirical (sample--averaged) gradient is defined by
\[
\overline{\delta\JJ_n{}}
\coloneqq
\mathcal{S}_{i}^{n} \{ \delta{\JJ}[\fmeas_{i}^{\noise}] \}, \qquad \mathcal{S}_{i}^{n}\{\cdot\} :=\frac{1}{n}\sum_{i=1}^{n} \cdot.
\]

Our analysis, in particular the proof of Theorem~\ref{thm:clt_topograd}, relies on the Fr\'echet central limit theorem (CLT) in separable Hilbert spaces.
Let $H$ be a separable Hilbert space and let $\{X_i\}_{i=1}^{n} \subset H$ be i.i.d. centered random elements satisfying
$\EE\!\left[\norm{X_i}_H^2\right] < \infty$.
Then
\[
\mathcal{S}_{i}^{n} \{  X_i  \} \xlongrightarrow{d} Z
\quad \text{in } H,
\]
where $Z$ is a centered Gaussian random element in $H$ whose covariance operator
is determined by the common distribution of the $X_i$
(see, e.g., \cite{ledoux2013probability}).

We apply this result with $H = L^2(\varOmega)$ to the fluctuation variables
\[
\YY_i
\coloneqq
\delta{\JJ}[\fmeas_{i}^{\noise}]
-
\delta{\JJ}[\fmeas],
\]
which represent the deviations of the noisy measurements from their mean response.

The following theorem establishes statistical stability of the topological gradient.

\begin{theorem}
\label{thm:clt_topograd}
Assume the conditions of Theorem~\ref{thm:stab_topograd} hold and that 
$\{\noisei\}_{i=1}^{n}$ are i.i.d. with zero mean and finite variance.
Suppose furthermore that the noise intensity satisfies
\begin{equation}\label{bias_condition}
\EE\!\left[
\norm{\noisei}_{H^{\frac{1}{2}}(\partial\varOmega)}^2
\right]
= o(1/\sqrt{n}).
\end{equation}
Then, as $n\to\infty$,
\begin{equation}\label{CV-CLT1}
\sqrt{n}
\Big(
\overline{\delta\JJ_n{}}
-
\delta{\JJ}[\fmeas]
\Big)
\xlongrightarrow{d} Z
\quad\text{in } L^2(\varOmega),
\end{equation}
where $Z$ is a centered Gaussian random field with covariance
\[
\operatorname{Cov}(Z(x),Z(y))
= \EE\!\Big[ \YY_{1}(x) \YY_{1}(y) \Big].
\]
For any $\Psi \in L^2(\varOmega)$,
\[
\sqrt{n}
\left\langle
\overline{\delta\JJ_n{}}
-
\delta{\JJ}[\fmeas],
\Psi
\right\rangle_{L^2(\varOmega)}
\xlongrightarrow{d}
\mathrm{N}(0, \sigma_\Psi^2),
\]
with variance satisfying
\[
\sigma_\Psi^2
\le
C^2
\sigma_{\noise}^2
\abs{\partial\varOmega}
\norm{\Psi}_{L^2(\varOmega)}^2 .
\]
\end{theorem}

\begin{corollary}\label{coro:clt_topograd}
Under the assumptions of Theorem~\ref{thm:clt_topograd}, define
\[
S_n
\coloneqq
\left\langle
\overline{\delta\JJ_n{}}
-
\delta{\JJ}[\fmeas], \ \Psi
\right\rangle_{L^2(\varOmega)} .
\]
Then
\[
\sqrt{n}\,S_n
\xlongrightarrow{d}
\mathrm{N}(0, \sigma_\Psi^2).
\]
Consequently, for any confidence level $1-\alpha\in(0,1)$,
an asymptotic $(1-\alpha)$ confidence interval for
$\langle \delta{\JJ}[\fmeas],\Psi\rangle_{L^2(\varOmega)}$
is
\[
\left[
\left\langle
\overline{\delta\JJ_n{}}, \ \Psi
\right\rangle_{L^2(\varOmega)}
-
z_{1-\alpha/2}\frac{\sigma_\Psi}{\sqrt{n}},
\;
\left\langle
\overline{\delta\JJ_n{}}, \ \Psi
\right\rangle_{L^2(\varOmega)}
+
z_{1-\alpha/2}\frac{\sigma_\Psi}{\sqrt{n}}
\right],
\]
where $z_{1-\alpha/2}$ denotes the $(1-\alpha/2)$ quantile of the standard normal distribution.
\end{corollary}

\begin{remark}
Corollary~\ref{coro:clt_topograd} provides a statistical method to determine the projected topological gradient's sign and thus detect the metal--semiconductor contact.
Specifically, an asymptotic $(1-\alpha)$ confidence interval for
$\langle \delta{\JJ}[\fmeas], \Psi \rangle_{L^2(\varOmega)}$ is obtained from the empirical average
$\langle \overline{\delta{\JJ}_{n}}, \Psi \rangle_{L^2(\varOmega)}$.
This enables the hypothesis test
\[
(\mathcal{H}_0):\quad 
\left\langle 
\delta{\JJ}[ \fmeas], \,
\Psi 
\right\rangle_{L^2(\varOmega)} 
\ge 0,
\]
which can be rejected at level $\alpha$ only if the entire confidence interval lies strictly below zero. 
In this case, a negative projected gradient can be interpreted as statistically significant evidence of a topological change rather than a fluctuation induced by measurement noise. 
Furthermore, the $n^{-1/2}$ convergence rate implied by the CLT justifies sample averaging as an effective variance-reduction strategy, thereby improving the reliability of the inferred contact region.
\end{remark}

To prove Theorem~\ref{thm:clt_topograd}, we first establish Fr\'{e}chet differentiability of the solution operators with respect to the boundary data.

\begin{lemma}
\label{lem:frechet_forward} 
The mapping $\fmeas \longmapsto \uu[\fmeas]$ from $H^{\frac{1}{2}}(\partial\varOmega)$ into $\HH$ is Fr\'echet differentiable. 
Moreover, if $\fmeas \in H^{\frac{1}{2}}(\partial\varOmega)$ is a measured boundary datum and 
$\tilde{\noise} \in H^{\frac{1}{2}}(\partial\varOmega)$ is an arbitrary perturbation, then
\[
D_{\fmeas} \uu[\fmeas](\tilde{\noise})
=
\uu[\tilde{\noise}].
\]
\end{lemma}

\begin{lemma}\label{lem:frechet_v}
The mapping $\fmeas \longmapsto \v[{\fmeas}]$ from $H^{\frac{1}{2}}(\partial\varOmega)$ into $\HH$
is Fr\'echet differentiable.
The derivative $D_{\fmeas} \v[{\fmeas}]({\tilde{{\noise}}})=\linv[{\tilde{{\noise}}}]$ solves
\begin{equation}\label{V-linearized}
-\Delta \linv + {\mu} \,\linv
= -2\, \iu[{\tilde{{\noise}}}]
\quad
\text{in }\varOmega,
\qquad\quad
\partial_\nu \linv - \ii \, \beta \, \linv
=0 \quad
\text{on }\partial\varOmega .
\end{equation}
\end{lemma}

Define the bilinear operator $\biform(\w,\v) \coloneqq w_2 v_1 - w_1 v_2$, for any $\w=w_1+\ii\, w_2,\v=v_1+\ii\, v_2\in\HH$.
Then
\[
\delta{\JJ}[{\fmeas}]
=
\biform(\uu[{\fmeas}],\v[{\fmeas}]).
\]

Since $\biform$ is bounded bilinear and the solution maps are Fr\'echet differentiable, the mapping ${\fmeas} \longmapsto \delta{\JJ}[{\fmeas}]$ is Fr\'echet differentiable with derivative
\begin{equation}\label{FD-Gradient}
D_f\delta{\JJ}[{\fmeas}]({\tilde{{\noise}}})
=
\biform\big(\uu[{\tilde{{\noise}}}],\v[{\fmeas}]\big)
+
\biform\big(\uu[{\fmeas}],\linv[{\tilde{{\noise}}}]\big).
\end{equation}

With the preceding results at hand, we are now in a position to prove Theorem~\ref{thm:clt_topograd}.

\begin{proof}[Proof of Theorem~\ref{thm:clt_topograd}]
Let $H=L^2(\varOmega)$ and define $\YY_i = \delta{\JJ}[\fmeas_{i}^{\noise}] - \delta{\JJ}[\fmeas]$.
By Theorem~\ref{thm:stab_topograd}, $\norm{\YY_i}_H \le C \norm{\noisei}_{H^{\frac{1}{2}}(\partial\varOmega)}$.
Hence, $\EE\!\left[\norm{\YY_i}_H^2\right]<\infty$.
A Taylor expansion at $\fmeas$ yields $\YY_i = D_f\delta{\JJ}[\fmeas](\noisei) + R(\noisei)$, with $\norm{R({\tilde{{\noise}}})}_H \le C \norm{{\tilde{{\noise}}}}_{H^{\frac{1}{2}}(\partial\varOmega)}^2$.
Because the noise has zero mean,
\[
\norm{\EE[\YY_i]}_H
\le
C
\EE
\!\left[
\norm{\noisei}_{H^{\frac{1}{2}}(\partial\varOmega)}^2
\right]
\stackrel{\eqref{bias_condition}}{=}
o\left( \frac{1}{\sqrt{n}}\right).
\]

Let $\widetilde{\YY}_i = \YY_i - \EE[\YY_i]$ denote the centered random variables.
These are i.i.d. in $H$ with finite second moment, so the Hilbert-space CLT gives
\[
\mathcal{S}_{i}^{n} \{ \widetilde{\YY}_i \}
\xlongrightarrow{d}
Z .
\]

Finally,
\[
\mathcal{S}_{i}^{n} \{ \YY_i \}
=
\mathcal{S}_{i}^{n} \{ \widetilde{\YY}_i \}
+
\sqrt{n}\,\EE[\YY_i],
\]
and the bias term vanishes by \eqref{bias_condition}, proving \eqref{CV-CLT1}.
\end{proof}

We conclude the section with the following remark.
\begin{remark}
The CLT provides a rigorous statistical justification for the robustness of the topological gradient with respect to measurement noise.

\begin{enumerate}
\item[(i)]
The convergence \eqref{CV-CLT1}
implies that the empirical mean converges to the true gradient, confirming asymptotic unbiasedness of the reconstruction.

\item[(ii)]
The convergence rate $O(n^{-\frac{1}{2}})$ is optimal for averaging independent measurements and shows that statistical uncertainty decreases with the number of samples.

\item[(iii)]
The small--noise condition \eqref{bias_condition} guarantees that nonlinear bias terms are negligible in the limit $n\to\infty$, so that stochastic fluctuations dominate the asymptotic behaviour.
\end{enumerate}
\end{remark}

\section{Shape Sensitivity Analysis}\label{sec:SSA}

In this section, we study the shape sensitivity of the cost functional $\JJ$. While the topological gradient indicates regions where adding or removing small inclusions most reduces the functional, it does not capture the effect of interface deformations. Shape sensitivity complements this by describing smooth variations of the interface, allowing more accurate geometry reconstruction. Combined, these approaches provide an efficient optimization strategy: the topological gradient identifies potential topological changes, and the shape derivative refines the interface to improve accuracy and convergence.

Let $\mathcal{A}$ denote the admissible set of subdomains:
\[
\mathcal{A} \coloneqq \Bigl\{ \omega \Subset \varOmega \ \Big| \ \abs{\omega} > 0, \ \partial \omega \in C^{1,1} \Bigr\},
\]
and let $\varOmega_\circ \Subset \varOmega$ be a fixed Lipschitz subdomain chosen large enough to contain all admissible inclusions, i.e., $\omega \Subset \varOmega_\circ \Subset \varOmega$ for all $\omega \in \mathcal{A}$. 
This ensures a uniform positive distance between $\partial \omega$ and $\partial \varOmega$, so that admissible deformations of $\omega$ are given by vector fields
\[
{\sfTheta} \coloneqq \{ \VV \in C^{1,1}(\mathbb{R}^{d})^d \mid \operatorname{supp}(\VV) \subset \varOmega_\circ \}.
\]

For $\VV \in {\sfTheta}$, the perturbed domain is $\omega_t \coloneqq T_t(\omega)$, where $T_t \coloneqq \mathrm{Id} + t\,\VV$ perturbs the identity map $\mathrm{Id} : \overline{\varOmega} \to \overline{\varOmega}$. Since $\VV$ is $C^{1,1}$ with compact support in $\varOmega_\circ \subset \varOmega$, for sufficiently small $t>0$ the map $T_t$ is a $C^{1,1}$ diffeomorphism that leaves $\partial \varOmega$ fixed, and the perturbed interface is $\partial \omega_t \coloneqq T_t(\partial \omega)$. 
The normal component of $\VV$ at the interface is denoted $\Vn \coloneqq \langle \VV, \nu \rangle$.

The following result characterizes the shape derivative of the state $\uu$ under admissible perturbations of the interface $\partial \omega$.
\begin{theorem}\label{thm:shape_derivative_of_the_state}
Let $\omega \in \mathcal{A}$ and $\VV \in {\sfTheta}$. 
Then the state $\uu \in \HH$ (solution of \eqref{eq:CCBM}) is shape differentiable in the direction $\VV$, and its shape derivative $\uprime \in \HH$ satisfies the transmission problem
\begin{equation}
\label{eq:ccbm_shape_derivative}
\left\{
\begin{aligned}
\Delta \uprime &= \mu_0 \, \uprime && \text{in } \omega,\\
\Delta \uprime &= 0 && \text{in } \varOmega \setminus \overline{\omega},\\
\llbracket \uprime \rrbracket&= 0, \qquad \llbracket \partial_\nu \uprime\rrbracket  \, = \, \mu_0 \, \uu \, \Vn && \text{on } \partial \omega,\\
\partial_\nu \uprime + \ii \, \beta \, \uprime &= 0 && \text{on } \partial \varOmega,
\end{aligned}
\right.
\end{equation}
$\llbracket \cdot \rrbracket$ denotes the jump across $\partial \omega$. 
Equivalently, in weak form, $\uprime \in \HH$ satisfies Problem~\ref{prob:shape_derivative_weak}.
\end{theorem}
%
%
\begin{problem}
\label{prob:shape_derivative_weak}
Given the state $\uu \in \HH$, find the shape derivative $\uprime \in \HH$ such that
\begin{equation}\label{eq:shape_derivative_weak}
\hat{a}(\uprime, \w) = \hat{F}(\w), \quad \forall \w \in \HH,
\end{equation}
where the sesquilinear form $\hat{a}(\cdot,\cdot)$ and the linear form $\hat{F}(\cdot)$ are defined by
\begin{align*}
\hat{a}(\uprime, \w) 
&\coloneqq (\nabla \uprime, \nabla \w)_{0,\varOmega \setminus \overline{\omega}} 
+ (\nabla \uprime, \nabla \w)_{0,\omega} 
+ (\mu_0 \, \uprime, \w)_{0,\omega} 
+ \ii \, \beta \, \langle \uprime, \w \rangle_{0,\partial \varOmega}, \\
\hat{F}(\w) 
&\coloneqq \int_{\partial \omega} \llbracket\partial_\nu \uprime\rrbracket \, \overline{\w} \, \dss
= \int_{\partial \omega} \mu_0 \, \uu \, \Vn \, \overline{\w} \, \dss.
\end{align*}
\end{problem}
The well-posedness of Problem~\ref{prob:shape_derivative_weak} follows from the complex Lax--Milgram lemma, and the proof of Theorem~\ref{thm:shape_derivative_of_the_state} proceeds along the same lines as in \cite{AfraitesDambrineKateb2007} and is therefore omitted.

The following lemma characterizes the first-order shape derivative of the cost functional $\JJ$.  
\begin{lemma}\label{lem:shape_derivative}
Let $\uu = \ru +  \ii \, \iu \in \HH$ be the solution of the state problem \eqref{eq:CCBM} with Cauchy data  $\gdata \in H^{-\frac{1}{2}}(\partial \varOmega)$ and $\fmeas \in H^{\frac{1}{2}}(\partial \varOmega)$.  
Then the cost functional\footnote{Since $\varOmega = \omega \cup (\varOmega \setminus \overline{\omega})$ and $\omega \cap (\varOmega \setminus \overline{\omega}) = \varnothing$, we have 
$\int_{\omega \cup (\varOmega \setminus \overline{\omega})} \cdot = \int_{\varOmega} \cdot$, 
equivalently 
$\int_{\varOmega} \cdot = \int_{\omega} \cdot + \int_{\varOmega \setminus \overline{\omega}} \cdot$, 
and also 
$\int_{\varOmega} \cdot = \int_{\varOmega} (\chi_{\omega} + \chi_{\varOmega \setminus \overline{\omega}})\, \cdot$.}
\[
\mathcal{J}(\omega):=\JJ(\mu_0\chi_\omega) = \intO{\iu^2} 
\] 
is shape differentiable. Its shape derivative in the direction $\VV \in \sfTheta$ is
\begin{equation}\label{eq:shape_gradient}
d\mathcal{J}(\omega)[\VV] = \int_{\partial \omega} G  \, \Vn \, \dss,
\end{equation}
where $G=\mu_0 \, (\iu \, \vartheta_r - \ru \, \vartheta_i)$ and $\vartheta = \vartheta_r +  \ii \, \vartheta_i \in \HH$ is the solution of the adjoint problem
\begin{equation}\label{eq:adjoint_system}
\left\{
\begin{aligned}
-\Delta \vartheta + \mu_0 \, \vartheta &= -2 \, \iu && \text{in } \omega,\\
-\Delta \vartheta &= -2 \, \iu && \text{in } \varOmega \setminus \overline{\omega},\\
\llbracket\vartheta\rrbracket &= 0, \quad \llbracket\partial_\nu \vartheta\rrbracket = 0 && \text{on } \partial \omega,\\
\partial_\nu \vartheta - \ii \, \beta \, \vartheta &= 0 && \text{on } \partial \varOmega.
\end{aligned}
\right.
\end{equation}
\end{lemma}
The weak formulation of \eqref{eq:adjoint_system} reads as follows:
\begin{problem}\label{prob:adjoint_weak}
Given the state $\uu \in \HH$, find $\vartheta\in \HH$ such that
\begin{equation}\label{eq:adjoint_weak}
a^{\dagger}(\vartheta, \w) = F^{\dagger}(\w), \quad \forall \w \in \HH,
\end{equation}
where the sesquilinear form $a^{\dagger}(\cdot,\cdot)$ and the linear form $F^{\dagger}(\cdot)$ are defined by
\begin{align*}
a^{\dagger}(\vartheta, \w) 
&\coloneqq (\nabla \vartheta, \nabla \w)_{0,\varOmega \setminus \overline{\omega}} + (\nabla \vartheta, \nabla \w)_{0,\omega} + (\mu_0 \, \vartheta, \w)_{0,\omega} - \ii \, \beta \, \inS{\vartheta, \w},
\\
F^{\dagger}(\w) 
&\coloneqq (-2\,\iu, \w)_{0,\varOmega}.
\end{align*}
\end{problem}
The well-posedness of Problem~\ref{prob:adjoint_weak} follows from the complex version of the Lax--Milgram lemma.
\begin{proof}[Proof of Lemma~\ref{lem:shape_derivative}]
Since the support of the cost functional is strictly separated from the fixed boundary $\partial \varOmega$, the perturbed domain does not intersect $\partial \varOmega$. Hence $\mathcal{J}$ is shape differentiable.  

Applying the chain rule for shape derivatives, we obtain
\[
d\mathcal{J}(\omega)[\VV] = \int_{\varOmega \setminus \overline{\omega}} 2 \, \iu \, \uprime_i \, \dx + \int_{\omega} 2 \, \iu \, \uprime_i \, \dx.
\]

To express this as a boundary integral, we test \eqref{eq:ccbm_shape_derivative} against the adjoint solution $\vartheta$.
Integration by parts on each subdomain gives
\begin{equation}\label{eq:first_equation}
(\nabla \uprime, \nabla \vartheta)_{0,\varOmega \setminus \overline{\omega}}
+ (\nabla \uprime, \nabla \vartheta)_{0,\omega}
+ (\mu_0 \, \uprime, \vartheta)_{0,\omega}
+ \int_{\partial \omega} \llbracket \partial_\nu \uprime\rrbracket \, \overline{\vartheta} \, \dss
+ \ii \, \langle \uprime, \vartheta \rangle = 0,
\end{equation}
where the boundary term on $\partial \varOmega$ is interpreted as an $H^{-\frac{1}{2}}$--$H^{\frac{1}{2}}$ duality pairing.  

Similarly, taking $\uprime$ as a test function in \eqref{eq:adjoint_weak} and applying integration by parts, we obtain
\begin{equation}\label{eq:second_equation}
a^{\dagger}(\vartheta, \uprime) = F^{\dagger}(\uprime).
\end{equation}

Subtracting the complex conjugate of equation \eqref{eq:first_equation} from \eqref{eq:second_equation} and using the jump condition $\llbracket \partial_\nu \uprime\rrbracket = \mu_0 \, \uu \, \Vn$ on $\partial \omega$, we obtain
\[
\intO{2 \, \iu \, \uprime_i} = - \Im \Big\{ \int_{\partial \omega} \mu_0 \, \overline{\uu} \, \vartheta \, \Vn \, \dss \Big\}
= \int_{\partial \omega} \mu_0 \, (\iu \, \vartheta_r - \ru \, \vartheta_i) \, \Vn \, \dss,
\]
which gives the boundary expression \eqref{eq:shape_gradient}.
\end{proof}
\begin{remark}
The derivation above assumes $\partial \omega \in C^{1,1}$ and $\VV \in C^{1,1}_c(\varOmega_\circ)$, which guarantees that the normal vector and the perturbed interface are defined pointwise.  
For the weak formulation, it suffices that $\partial \omega$ is Lipschitz and $\VV \in W^{1,\infty}_c(\varOmega_\circ)^d$, so that $T_t = \mathrm{Id} + t \VV$ is a bi-Lipschitz map and the shape derivative is defined in the weak sense.
Under these minimal assumptions, the shape derivative of the cost functional can be rigorously obtained using the rearrangement method \cite{IKP2006,IKP2008}.
\end{remark}

\begin{remark}
Similar to the expression of the topological gradient~\eqref{eq:topological_gradient}, the structure of the shape gradient~\eqref{eq:shape_gradient} does not depend on the free parameter $\beta > 0$.
Nevertheless, this parameter can improve the reconstruction, particularly when shape optimization is employed.
In the numerical experiments presented in the following sections, we set $\beta = 200$, unless otherwise stated.
\end{remark}

\section{Reconstruction via topological gradient}
We present concrete numerical examples of Problem~\ref{eq:main_problem} and reconstruct the unknown regions using the topological gradient \eqref{eq:topological_gradient} through a non-iterative, one-shot algorithm (Algorithm~\ref{alg:one-shot}).
\subsection{Topological identification algorithm and computational setup}
The algorithm proceeds from an empty configuration, $\omega = \emptyset$, corresponding to $\mu = 0$.
\begin{algorithm}[H]
\caption{One-shot topological identification algorithm}
\label{alg:one-shot}
\begin{itemize}
    \item[\textbf{0.}] Specify the domain $\varOmega$, inclusion coefficient $\mu_0$, Cauchy data $(\fmeas,\gdata)$, and noise level $\delta$ (if applicable).
    \item[\textbf{1.}] Solve the state problem:
    \[
    -\Delta \uu[0] = 0 \quad \text{in } \varOmega, \qquad
    \partial_\nu \uu[0] + \ii \beta \, \uu[0] = \gdata + \ii \beta \, \fmeas \quad \text{on } \partial\varOmega.
    \]
    \item[\textbf{2.}] Solve the adjoint problem:
    \[
    -\Delta \v[0] = -2\,\iu[0] \quad \text{in } \varOmega, \qquad
    \partial_\nu \v[0] - \ii \beta \, \v[0] = 0 \quad \text{on } \partial\varOmega.
    \]
    \item[\textbf{3.}] Compute the topological gradient:
    \[
    \delta\JJ[0](\xi) \coloneqq \iu[0](\xi) \, \rv[0](\xi) - \ru[0](\xi) \, \iv[0](\xi), \qquad \xi \in \varOmega.
    \]
    \item[\textbf{4.}] Locate the negative local minima of $\delta\JJ[0]$ to estimate the contact region.
\end{itemize}
\end{algorithm}
The general computational setup is as follows. The computational domain is the unit square $[-0.5,0.5]^2$. Synthetic Cauchy data $(\gdata, \fmeas)$ are generated by solving the forward problem with the exact inclusion geometries $\omega$. To avoid inverse crimes (see \cite[p.~154]{ColtonKress2013}), the forward problem is solved on a fine $200 \times 200$ mesh $\mathcal{T}_h^{\star}$ using $\mathbb{P}_2$ finite elements, while the inversion is performed on a coarser $100 \times 100$ mesh $\mathcal{T}_h$ with $\mathbb{P}_1$ elements. 
Noise is introduced to assess robustness by perturbing the forward solution: $u^\delta = \bigl(1 + \delta \, \mathrm{g.n.}\bigr) u$, where ``g.n'' denotes Gaussian noise with zero mean and standard deviation $\norm{u}_{L^\infty(\varOmega)}$, and $\delta$ specifies the noise level. The corresponding noisy boundary measurement is $f|_{\partial\varOmega} = u^\delta|_{\partial\varOmega}$.
All computations are carried out in {\sc FreeFem++} \cite{Hecht2012} on a MacBook Pro with an Apple M1 chip and 16\,GB of RAM.

\begin{remark}\label{rem:isolevel_detection}
The one-shot topological gradient algorithm is applied to estimate the unknown contact region $\omegastar$. The location of $\omegastar$ is taken as $\xi^\star = \operatorname*{arg\,min}_{\xi \in \varOmega} \delta\JJ[0](\xi)$,
corresponding to the most negative value of the topological gradient, while the size of the reconstructed region is inferred from an appropriate level set of $\delta\JJ[0]$. We emphasize that this approach is heuristic: it provides a simple and computationally efficient estimate of both location and size without iterative optimization. Similar topological gradient methods have been successfully applied to a variety of inverse detection problems 
\cite{amstutz2006new,ferchichi2013detection,hrizione,jleli2015topological,samet2003topological}.

\end{remark}

\textit{Detection and visualization of the minimum region.}
The distributed topological gradient is evaluated on an unstructured triangulation, and candidate minima are identified algorithmically as strictly negative nodal values that are locally minimal and strictly convex with respect to a finite-ring neighborhood on the mesh.
A union--find activation strategy is employed to ensure consistency of connected components during the detection process and to reduce spurious minima.

The center of the detected contact region is approximated by averaging the locations of the identified local minima. An effective radius is inferred heuristically from the relative depth of these minima with respect to the global range of the topological gradient and should be interpreted as a qualitative indicator of size rather than a sharp geometric estimate.

For visualization, a dense family of iso-contours of the topological gradient is displayed. Moreover, for every detected minimum, the lowest iso-level region---bounded by the minimum value and the first contour level---is filled, thereby emphasizing the corresponding basin of attraction while maintaining the global contour structure.
\subsection{Numerical examples using the topological gradient}
The results in Figure~\ref{fig:simple_cases} demonstrate that the topological gradient method reliably localizes single circular contact regions across a range of radii and positions. The black curves indicate the true contact region boundaries, while the magenta markers correspond to the detected local minima of the topological gradient. In all cases, the global minimum lies in close proximity to the true center, confirming the robustness of the localization step. The cyan isolevel regions delineate the attraction basins associated with each minimum and provide a practical means for estimating contact region size. Nevertheless, the accuracy of the size estimation depends on the selected threshold (see Figure~\ref{fig:simple_cases}(g)--(h)), particularly for very small contact regions or those located near the boundary.

Figure~\ref{fig:intensity_tests} illustrates the effect of resistivity contrast on the topological gradient distribution. Higher contrast leads to deeper and more sharply localized minima, while lower contrast produces flatter profiles with reduced detectability. 
This strategy for localizing inclusions has been previously employed in \cite{Rabago2025,RabagoAzegami2018} in imaging modalities such as thermal and impedance imaging.
In the present work, we adapt this approach to infer the contact region location and to improve the initialization of the algorithm for multiple contact regions. One-dimensional slices further indicate that the depth of the minima correlates with the contact region strength, although it does not uniquely determine the contact region size.

The detection of multiple subregions is examined in Figures~\ref{fig:two_medium_sizes}--\ref{fig:two_unequal_square_sizes}. For two medium-sized contact regions (Figure~\ref{fig:two_medium_sizes}), two well-separated minima clearly coincide with the true subregions, and the number of pronounced minima provides a reliable estimate of the number of contact regions. When the contact regions become extremely small (Figure~\ref{fig:two_small_sizes}), the associated minima are shallower and more sensitive to boundary effects. In such cases, the choice of boundary input plays an important role: spatially varying excitations (e.g., $g = |x|$) may enhance sensitivity in specific regions and improve separation of nearby small contact regions.

Unequal sizes introduce additional asymmetry (Figure~\ref{fig:two_unequal_square_sizes}). Larger contact regions generate deeper and wider wells in the gradient field, whereas smaller contact regions produce weaker signatures. Proximity to the boundary may distort the gradient landscape and slightly shift the apparent minima inward, indicating that boundary interactions must be considered in quantitative interpretation.

The extension to three contact regions is presented in Figure~\ref{fig:three_subregions_test}. For unequal sizes, the gradient landscape exhibits three minima with clearly differentiated depths, reflecting the size disparity. For three equal medium-sized contact regions, three distinct wells remain visible and well separated, demonstrating that the method scales reasonably to moderate multi-region scenarios. However, when the contact regions are equal but small, the minima become less pronounced and the basins begin to overlap. Although three local minima can still be identified, their separation is reduced and the risk of merging increases. These results indicate that, while the number of dominant minima often correlates with the number of contact regions, this correspondence weakens as size decreases or contact regions become closer.

Figure~\ref{fig:noisy_three_subregions_test} evaluates robustness under $10\%$ contaminated measurements. The overall structure of the gradient field is preserved, and the principal minima remain detectable for unequal and medium-sized contact regions. However, noise introduces spurious oscillations and shallow local extrema, particularly near the boundary. For small contact regions, the contrast between true minima and noise-induced artifacts decreases significantly, making identification more delicate. Despite this degradation, the primary wells remain located near the true subregions, indicating a satisfactory level of stability with respect to moderate measurement noise.

\begin{figure}[H]
\centering

    \begin{subfigure}{0.3\textwidth}
	\includegraphics[width=\textwidth]{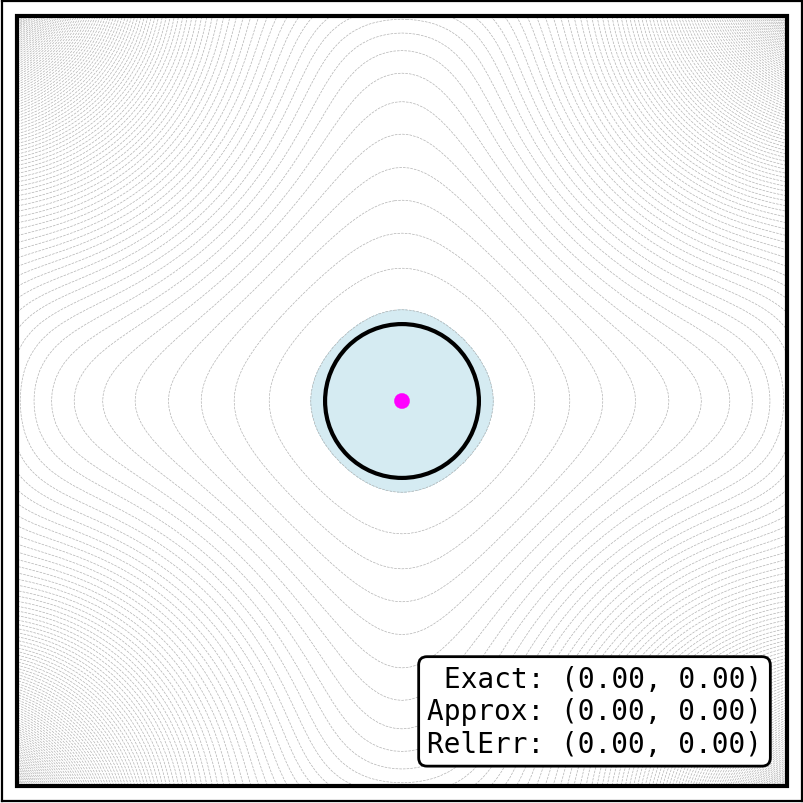}
        \caption{}
        \label{fig:simple_img1}
    \end{subfigure}
    \hfill
    \begin{subfigure}{0.3\textwidth}
	\includegraphics[width=\textwidth]{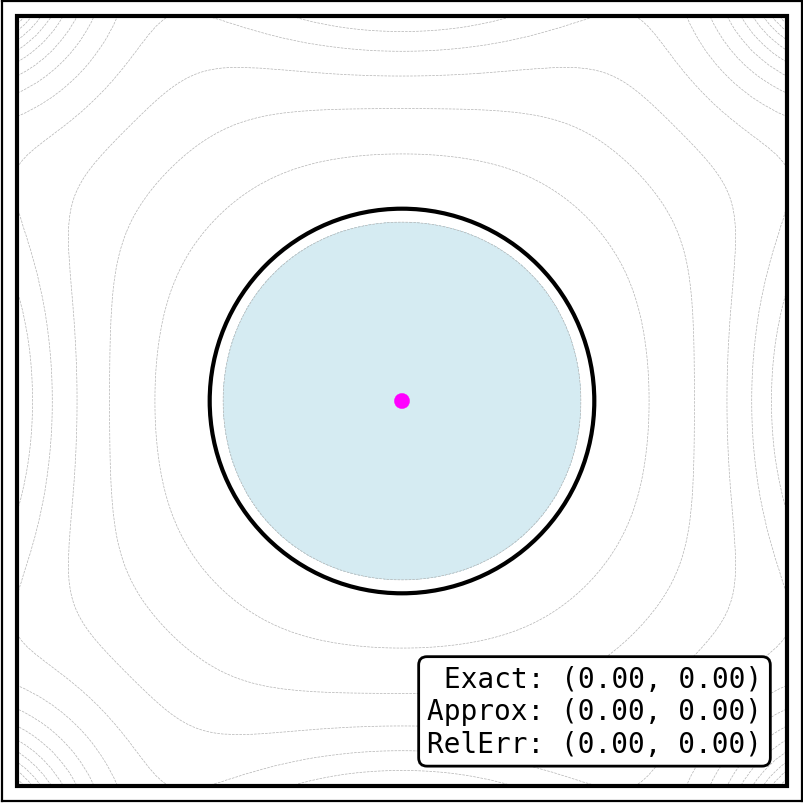}
       \caption{}
        \label{fig:simple_img2}
    \end{subfigure}
    \hfill
    \begin{subfigure}{0.3\textwidth}
            \includegraphics[width=\textwidth]{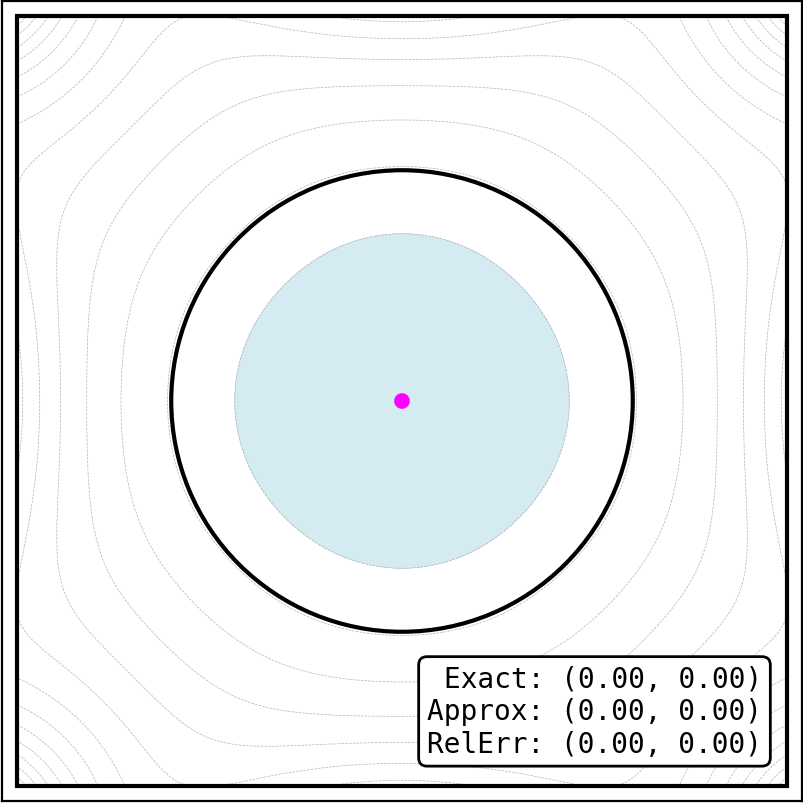}
       \caption{}
        \label{fig:simple_img3}
    \end{subfigure}

    \begin{subfigure}{0.3\textwidth}
        \includegraphics[width=\textwidth]{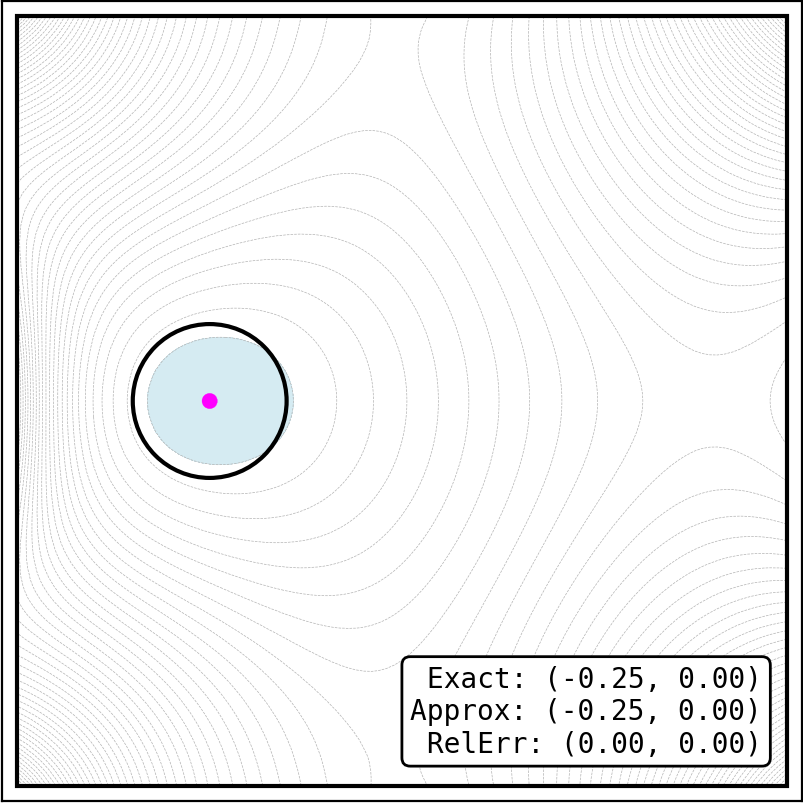}
        \caption{}
        \label{fig:simple_img4}
    \end{subfigure}
    \hfill
    \begin{subfigure}{0.3\textwidth}
        \includegraphics[width=\textwidth]{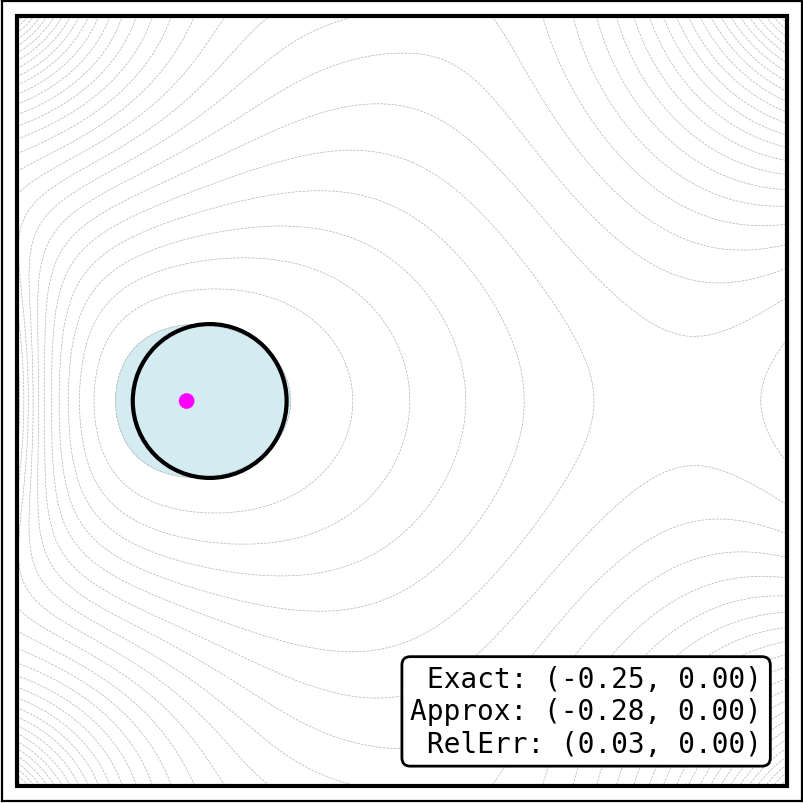}

        \caption{}
        \label{fig:simple_img5}
    \end{subfigure}
    \hfill
    \begin{subfigure}{0.3\textwidth}
        \includegraphics[width=\textwidth]{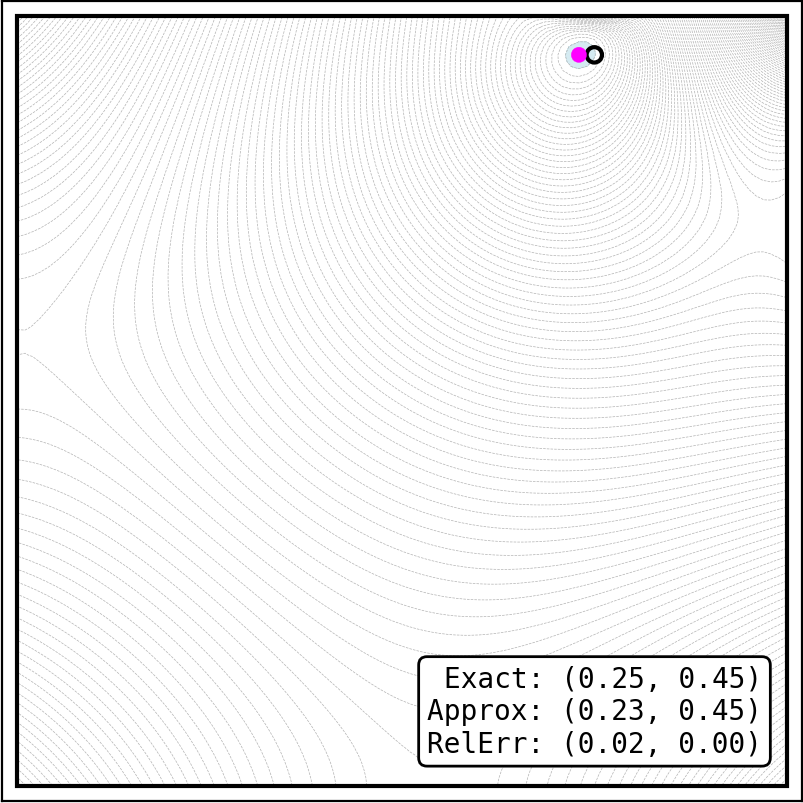}
        \caption{}
        \label{fig:simple_img6}
    \end{subfigure}

    \begin{subfigure}{0.3\textwidth}
        \includegraphics[width=\textwidth]{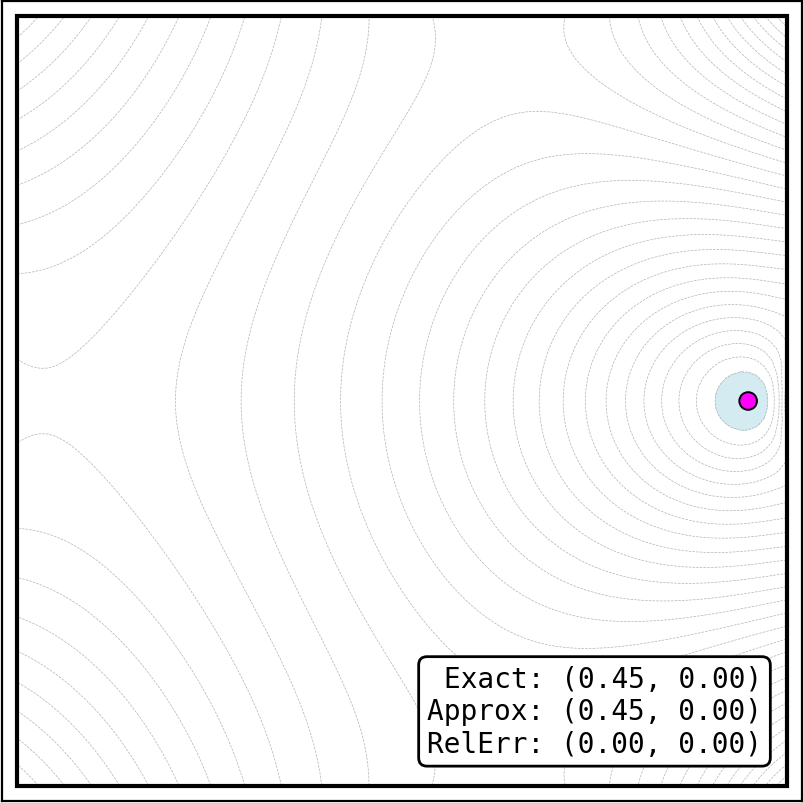}
        \caption{}
        \label{fig:simple_img7}
    \end{subfigure}
    \hfill
    \begin{subfigure}{0.3\textwidth}
                \includegraphics[width=\textwidth]{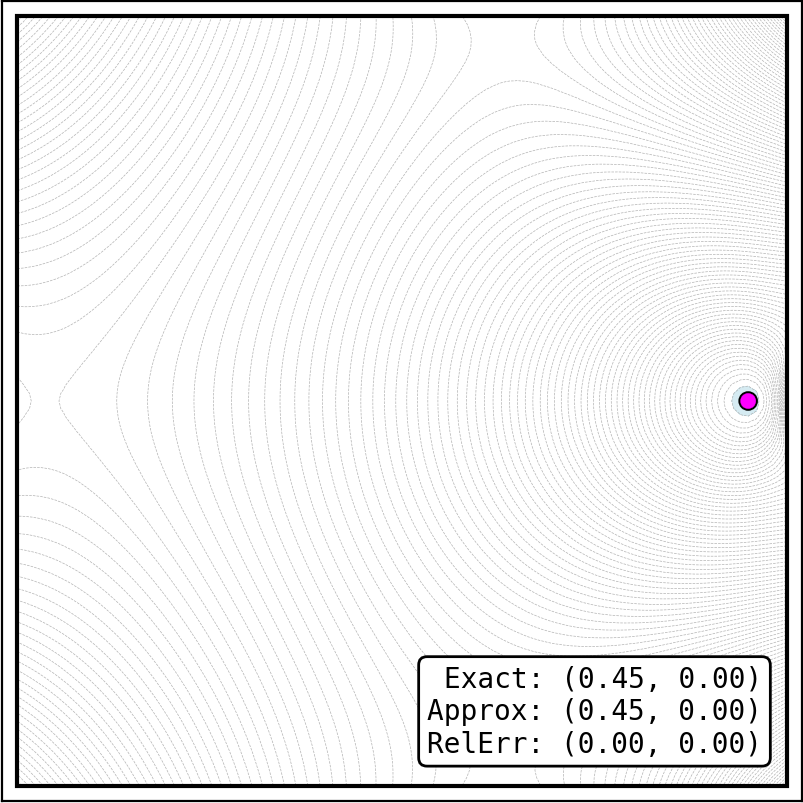}
        \caption{}
        \label{fig:simple_img8}
    \end{subfigure}
    \hfill
    \begin{subfigure}{0.3\textwidth}
        \includegraphics[width=\textwidth]{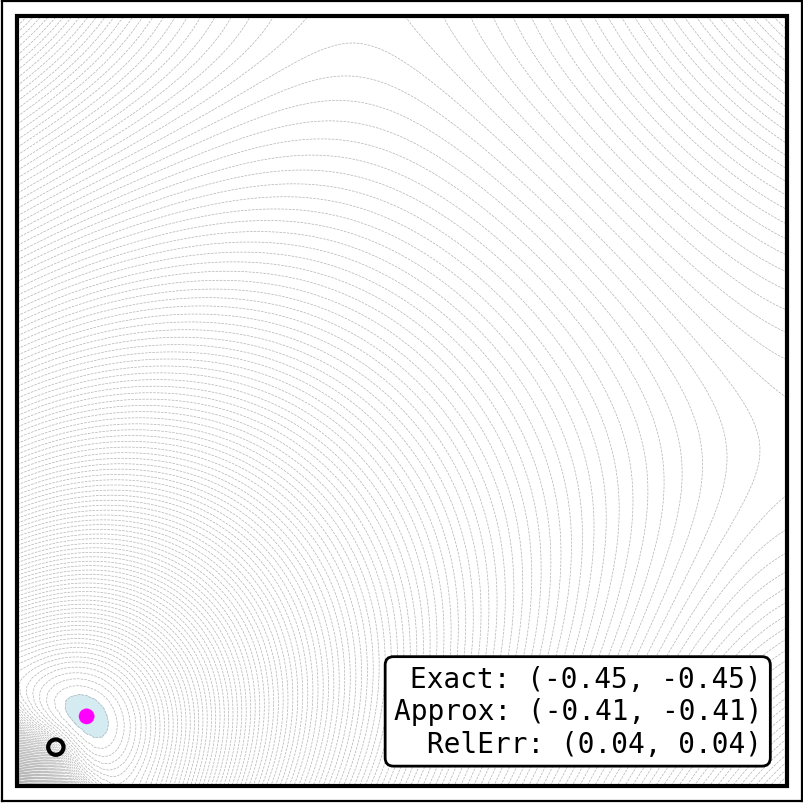}
       \caption{}
        \label{fig:simple_img9}
    \end{subfigure}

    \caption{(a)--(i): nine representative simple cases with $g = 1$.}
    \label{fig:simple_cases}
\end{figure}

In summary, the numerical experiments confirm that the topological gradient provides a reliable and computationally efficient first-step localization tool. The number of significant minima generally reflects the number of contact regions in configurations involving moderate size and contrast. However, this relationship is not strictly one-to-one: proximity, size disparity, boundary effects, excitation choice, and measurement noise may produce merged, distorted, or weak minima. Consequently, while the topological gradient offers robust qualitative localization, accurate counting and sizing---especially in complex or noisy multi-region scenarios---benefit from complementary information or subsequent refinement procedures.

\begin{figure}[htbp!]
    \centering
    \begin{subfigure}{0.4\textwidth}
	\includegraphics[width=\textwidth]{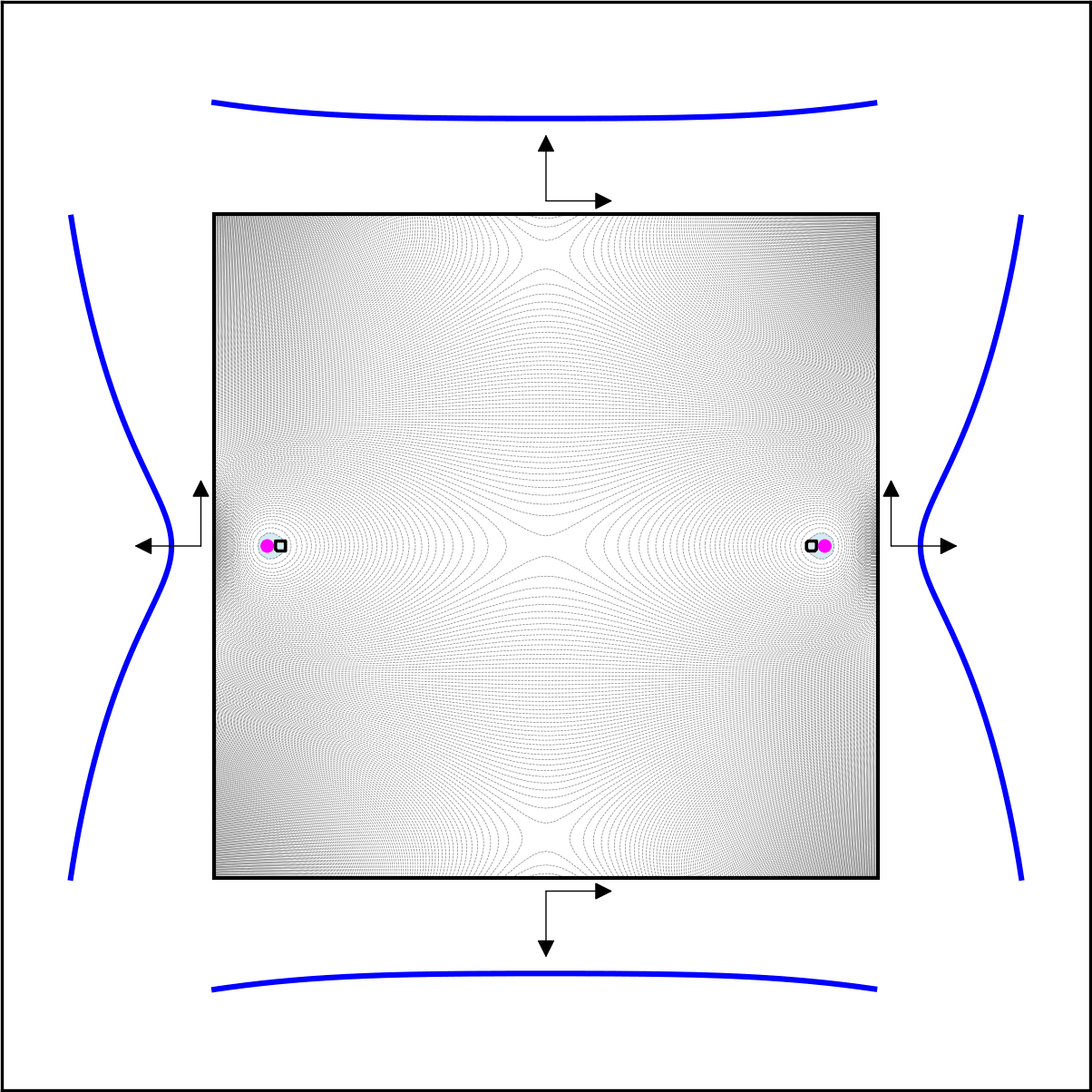}
        \label{fig:intensity_test_img1}
    \end{subfigure}
\qquad
    \begin{subfigure}{0.4\textwidth}
	\includegraphics[width=\textwidth]{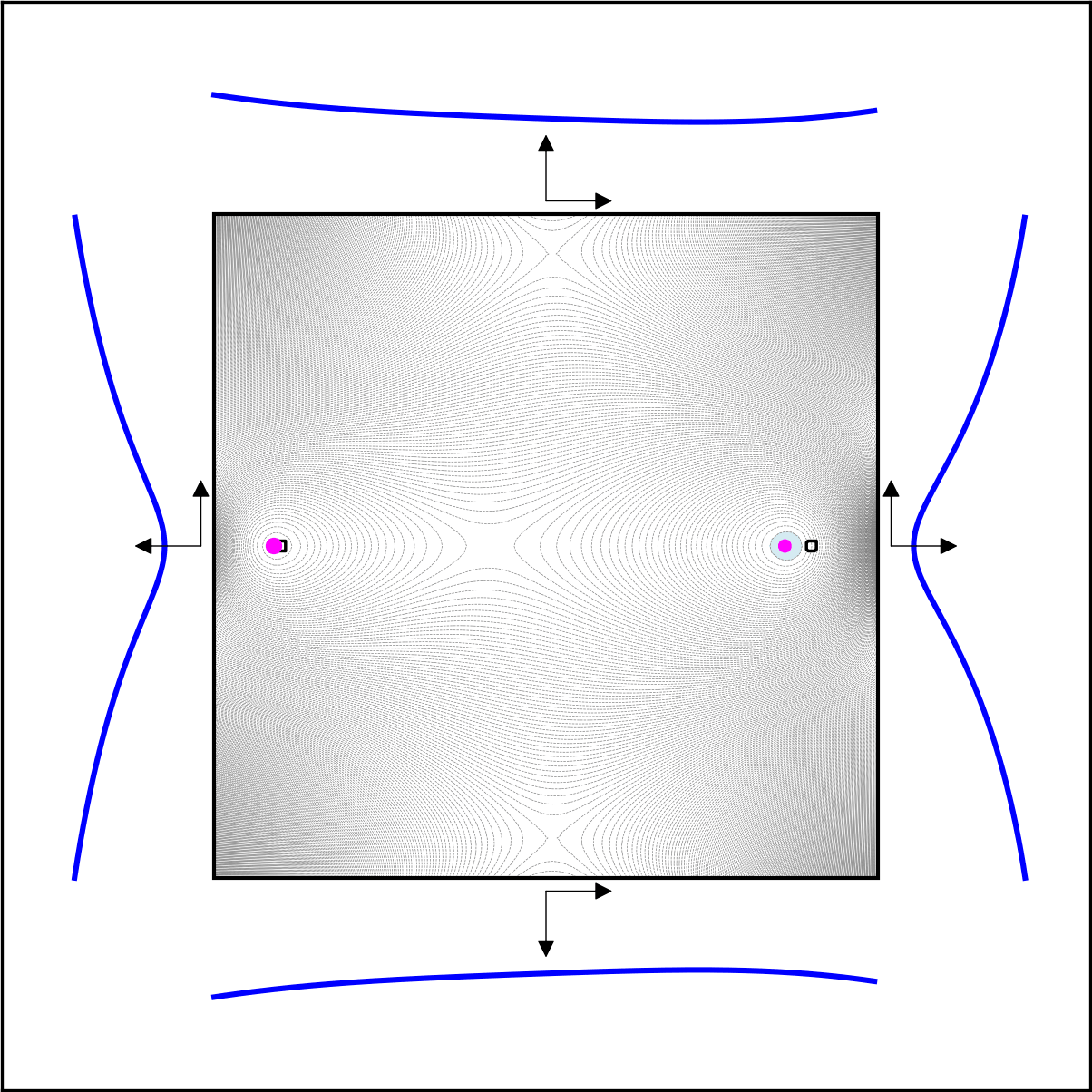}
        \label{fig:intensity_test_img2}
    \end{subfigure}
    \\
    \begin{subfigure}{0.4\textwidth}
	\includegraphics[width=\textwidth]{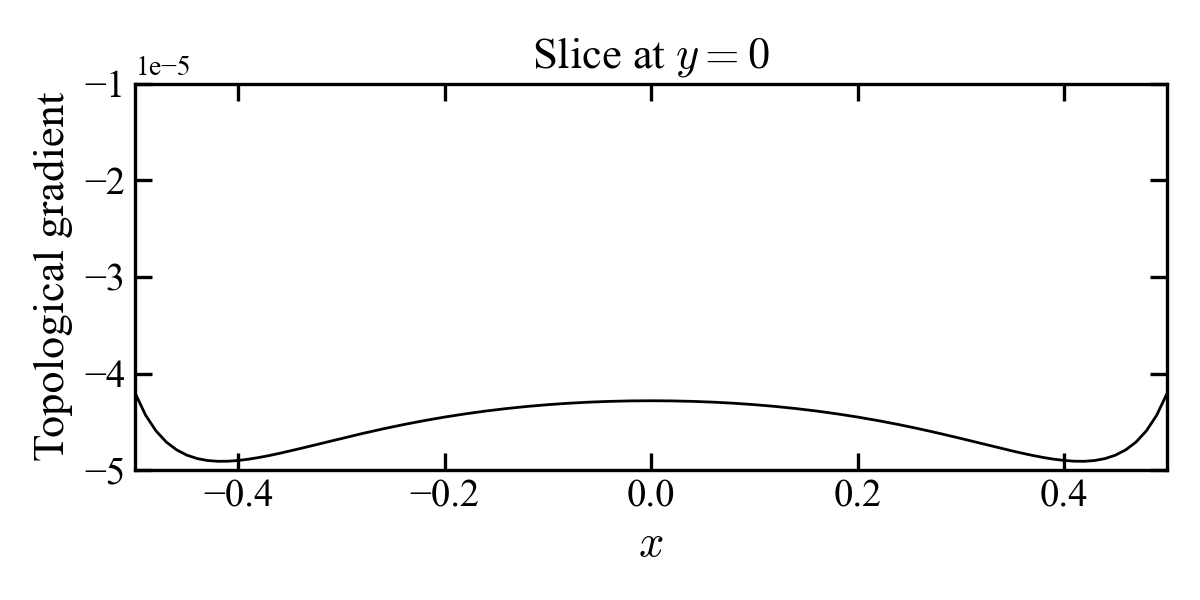}
        \label{fig:intensity_test_img3}
    \end{subfigure}
    \qquad
    \begin{subfigure}{0.4\textwidth}
	\includegraphics[width=\textwidth]{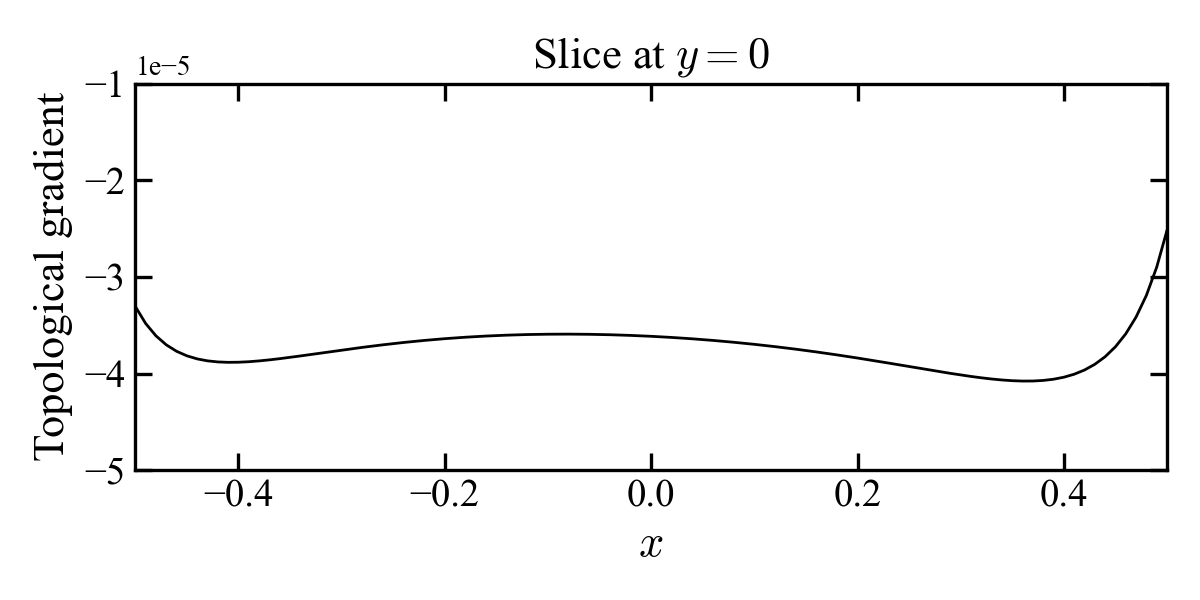}
        \label{fig:intensity_test_img4}
    \end{subfigure}
    \caption{Effect of resistivity on topological gradient distribution}
    \label{fig:intensity_tests}
\end{figure}

\begin{figure}[htbp!]
    \centering
    \begin{subfigure}{0.3\textwidth}
	\includegraphics[width=\textwidth]{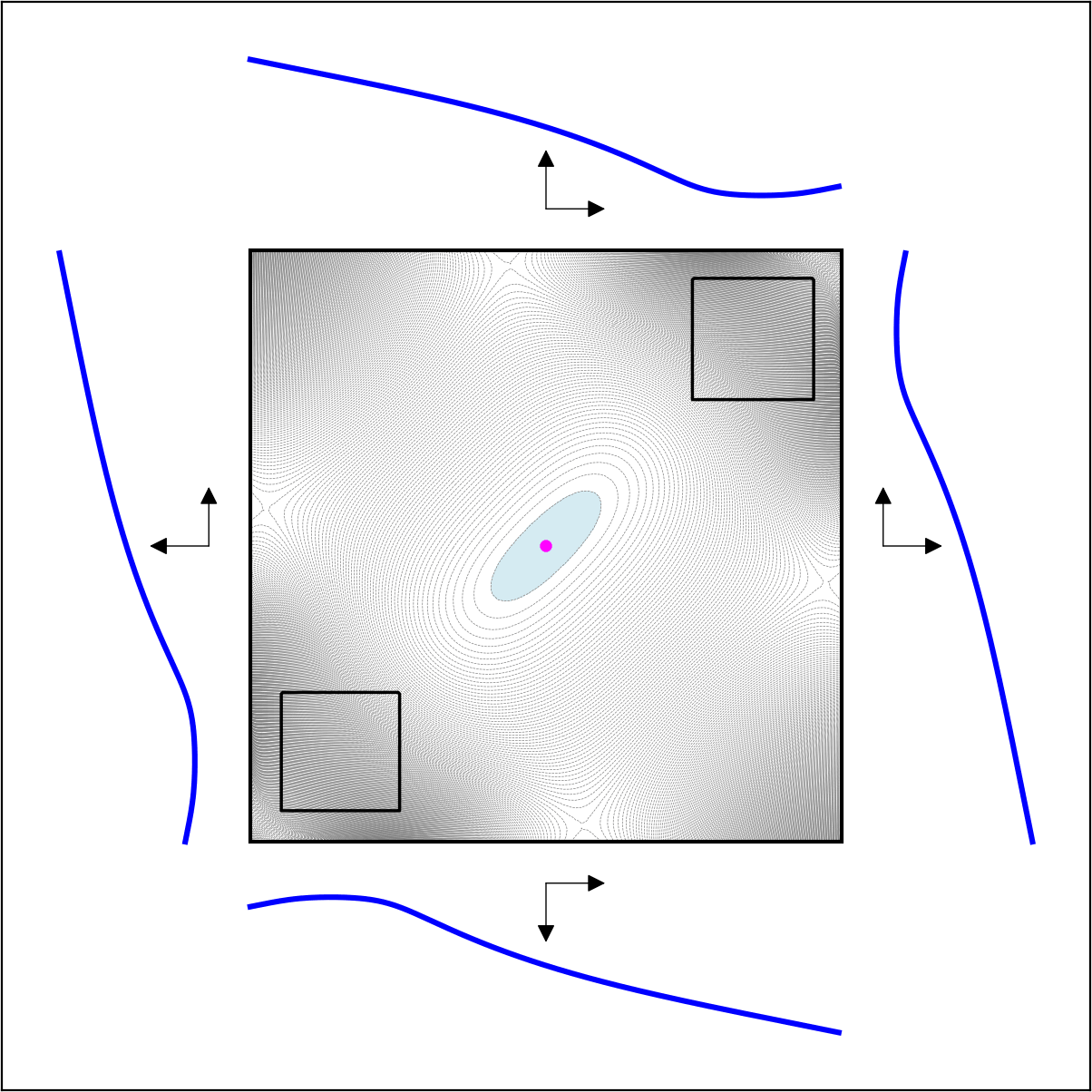}
        \label{fig:two_medium_size_img1}
    \end{subfigure}
    \hfill
    \begin{subfigure}{0.3\textwidth}
	\includegraphics[width=\textwidth]{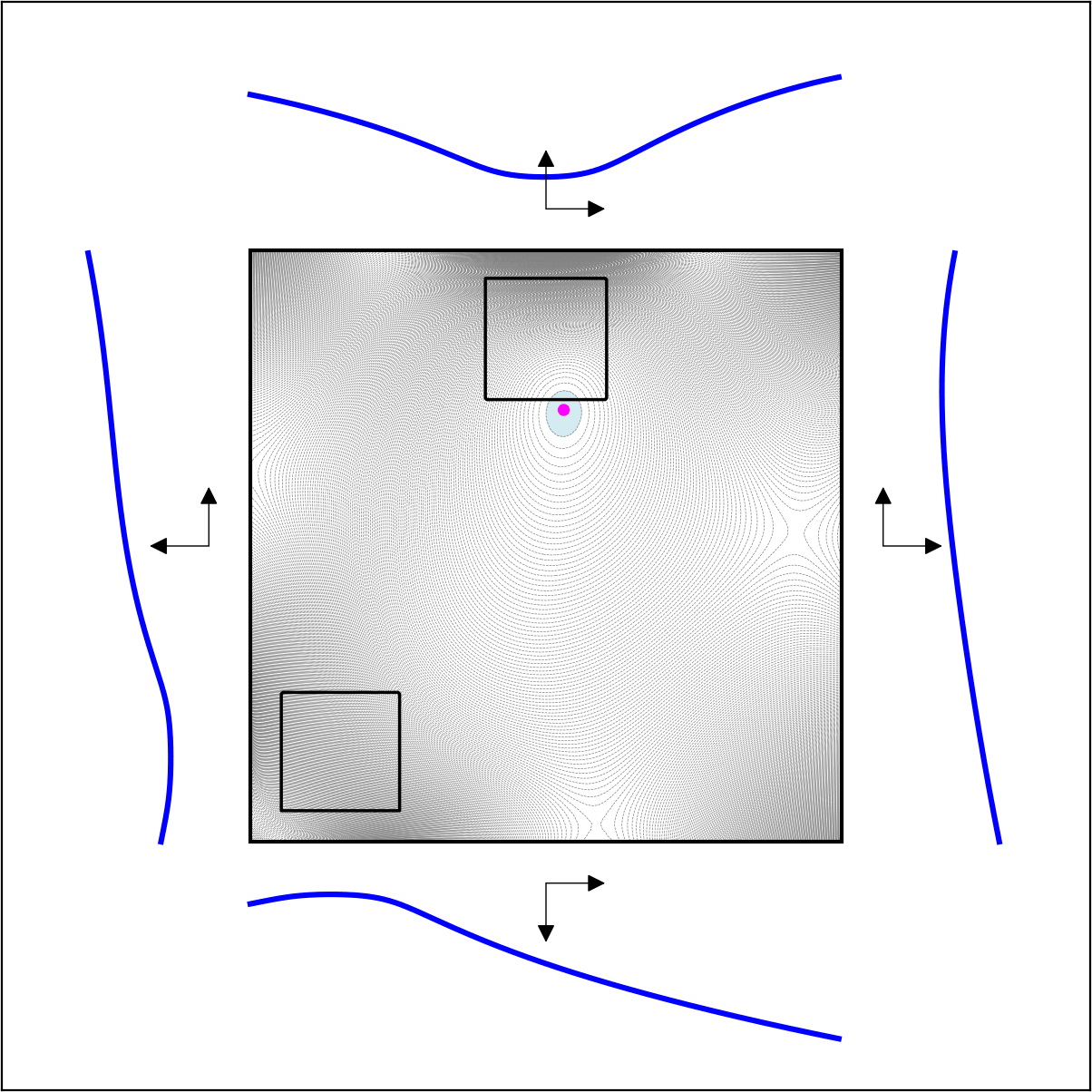}
        \label{fig:two_medium_size_img2}
    \end{subfigure}
    \hfill
    \begin{subfigure}{0.3\textwidth}
            \includegraphics[width=\textwidth]{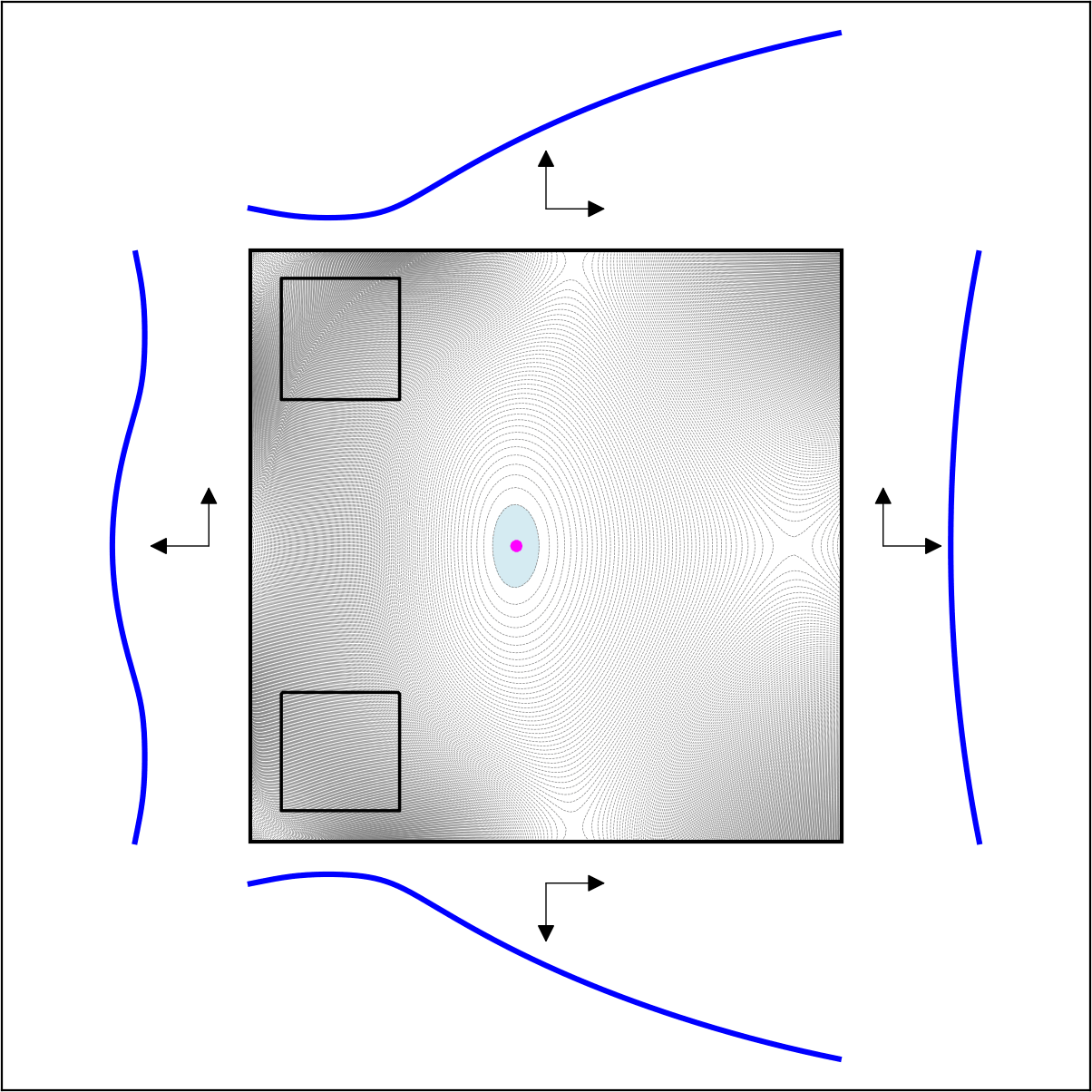}
        \label{fig:two_medium_size_img3}
    \end{subfigure}

    \caption{Detection of two medium-sized subregions with $g = 1$. The blue curves denote the measured boundary potentials; the positive $x$- and $y$-directions are indicated by right-angle arrows.}
    \label{fig:two_medium_sizes}
\end{figure}

\begin{figure}[htbp!]
    \centering
    \begin{subfigure}{0.3\textwidth}
	\includegraphics[width=\textwidth]{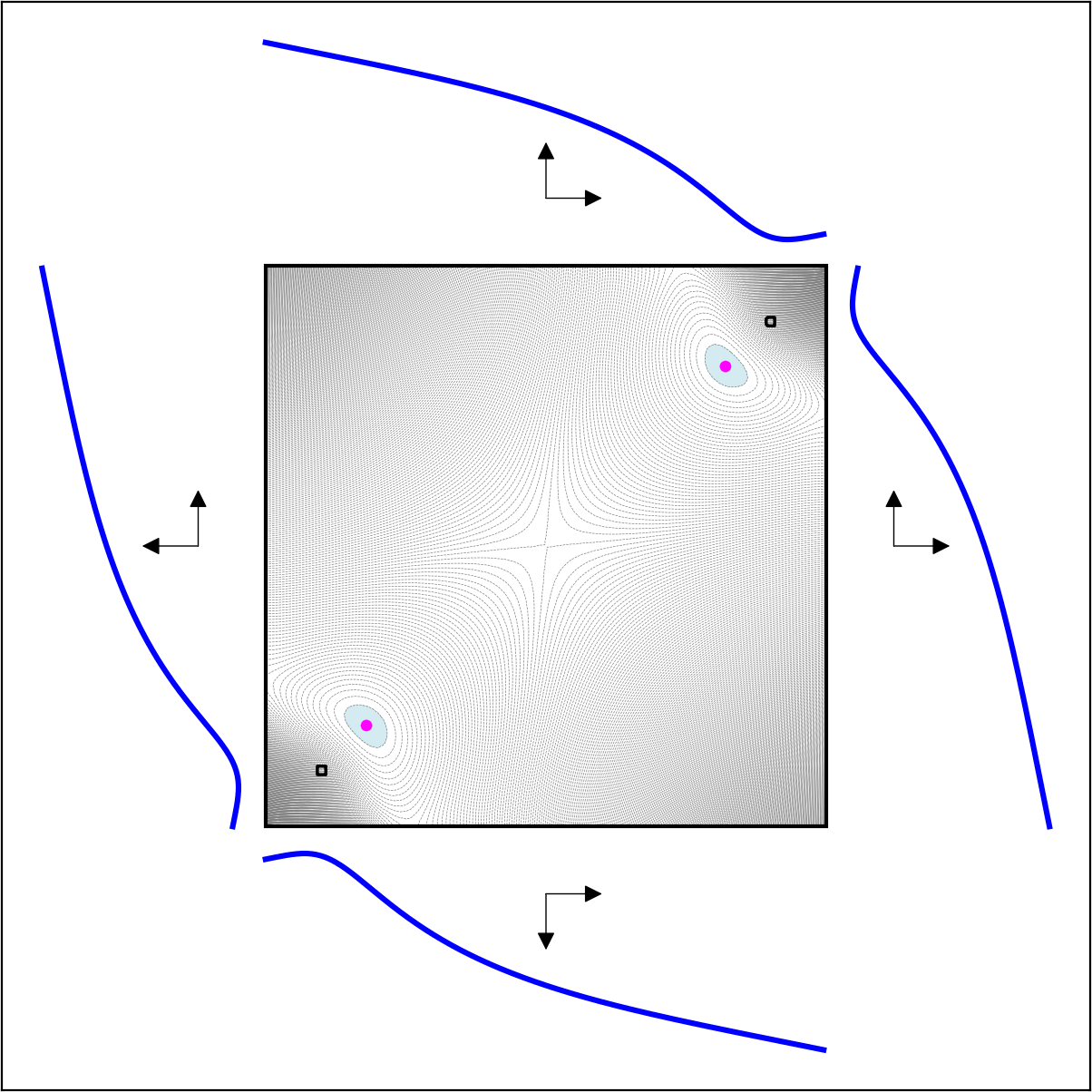}
        \label{fig:two_small_sizes_img1}
    \end{subfigure}
    \hfill
    \begin{subfigure}{0.3\textwidth}
	\includegraphics[width=\textwidth]{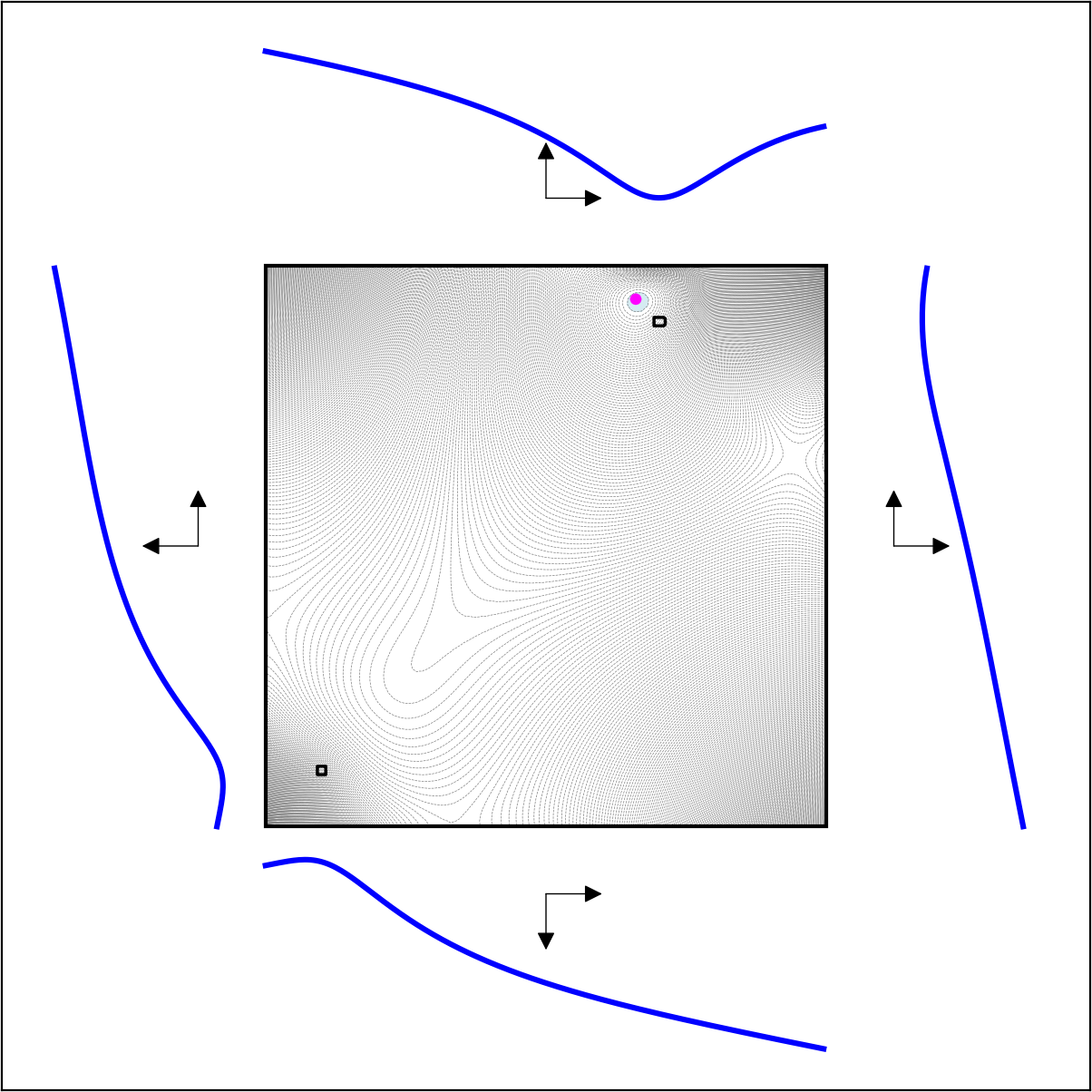}
        \label{fig:two_small_sizes_img2}
    \end{subfigure}
    \hfill
    \begin{subfigure}{0.3\textwidth}
	\includegraphics[width=\textwidth]{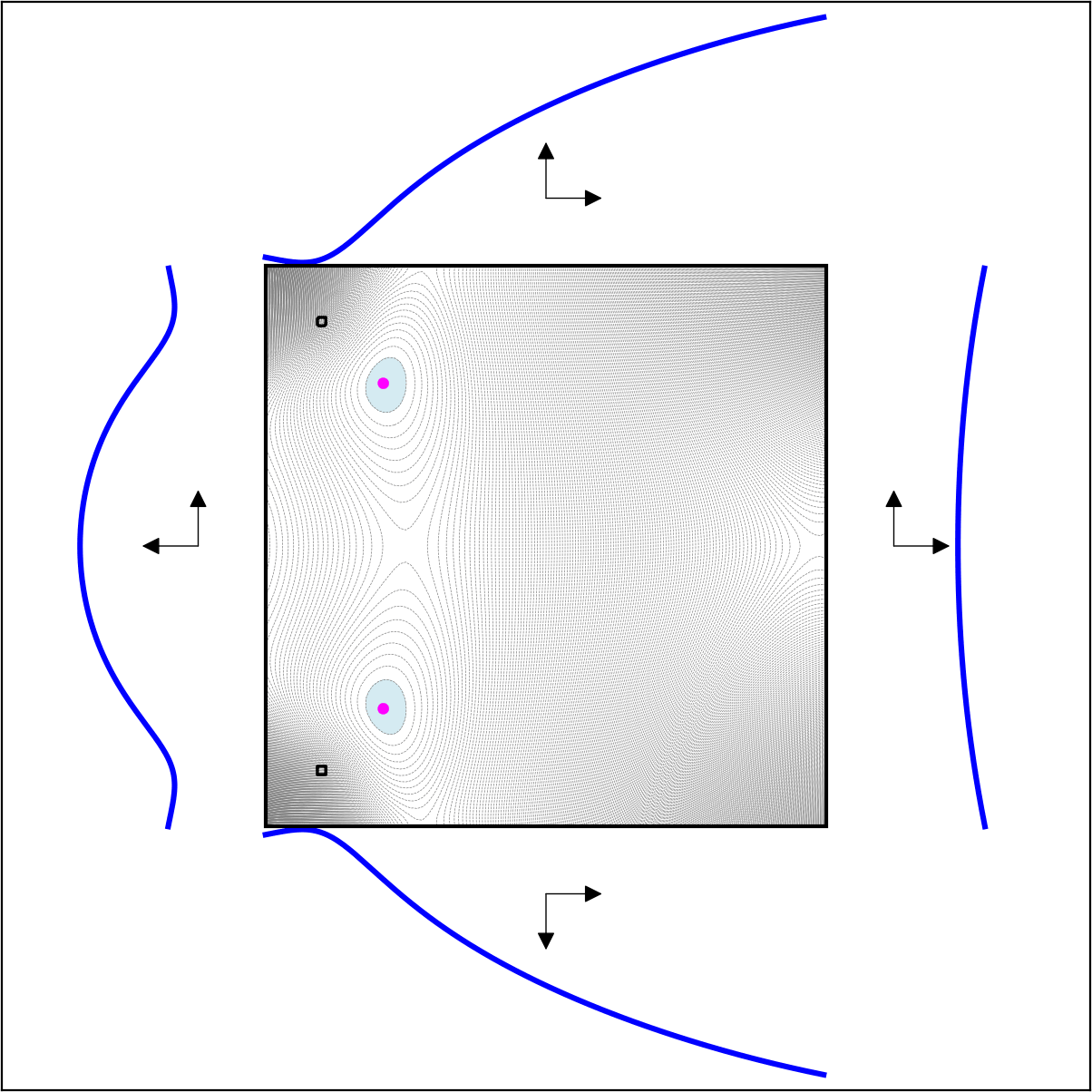}
        \label{fig:two_small_sizes_img3}
    \end{subfigure}
    \begin{subfigure}{0.3\textwidth}
	\includegraphics[width=\textwidth]{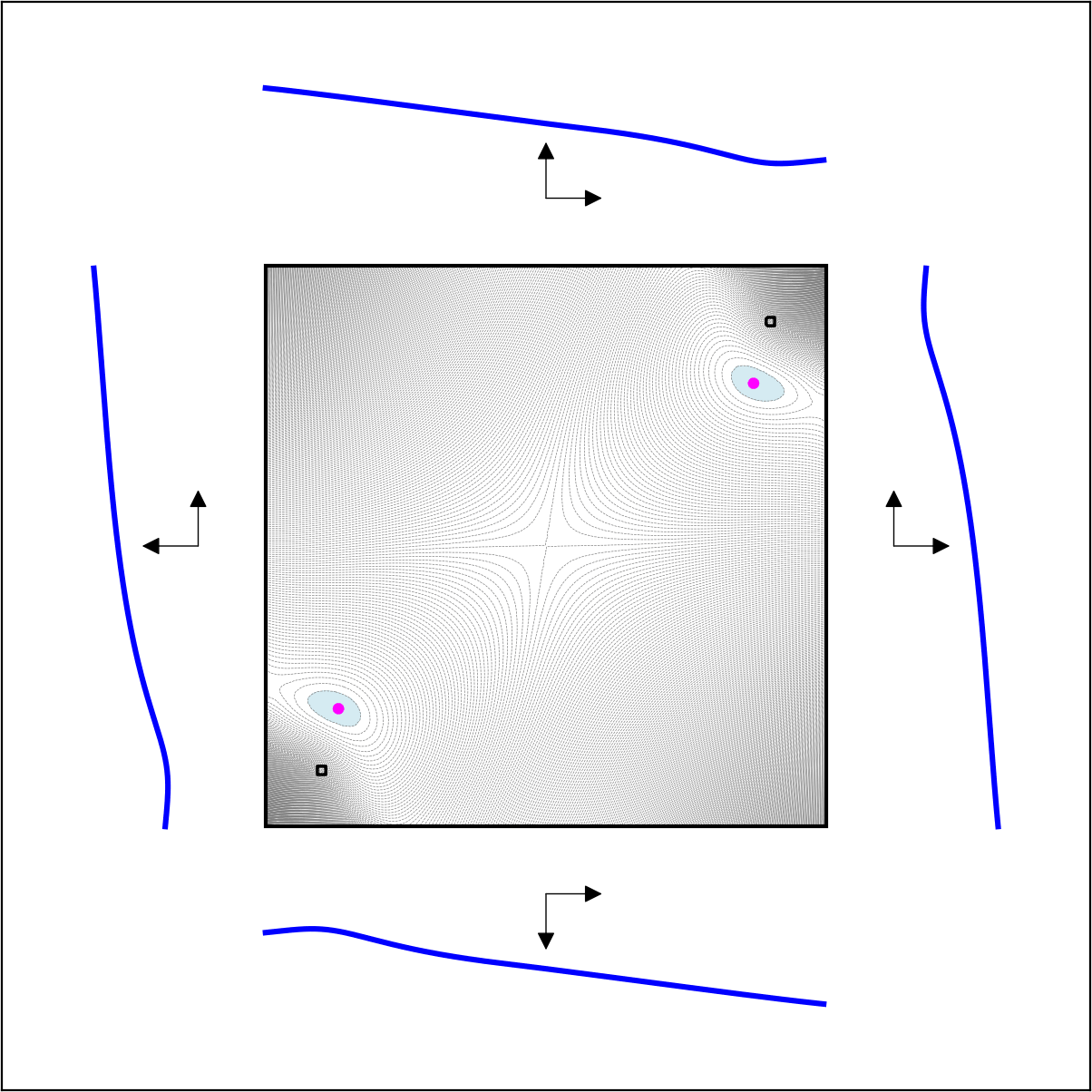}
        \label{fig:two_small_sizes_img4}
    \end{subfigure}
    \hfill
    \begin{subfigure}{0.3\textwidth}
	\includegraphics[width=\textwidth]{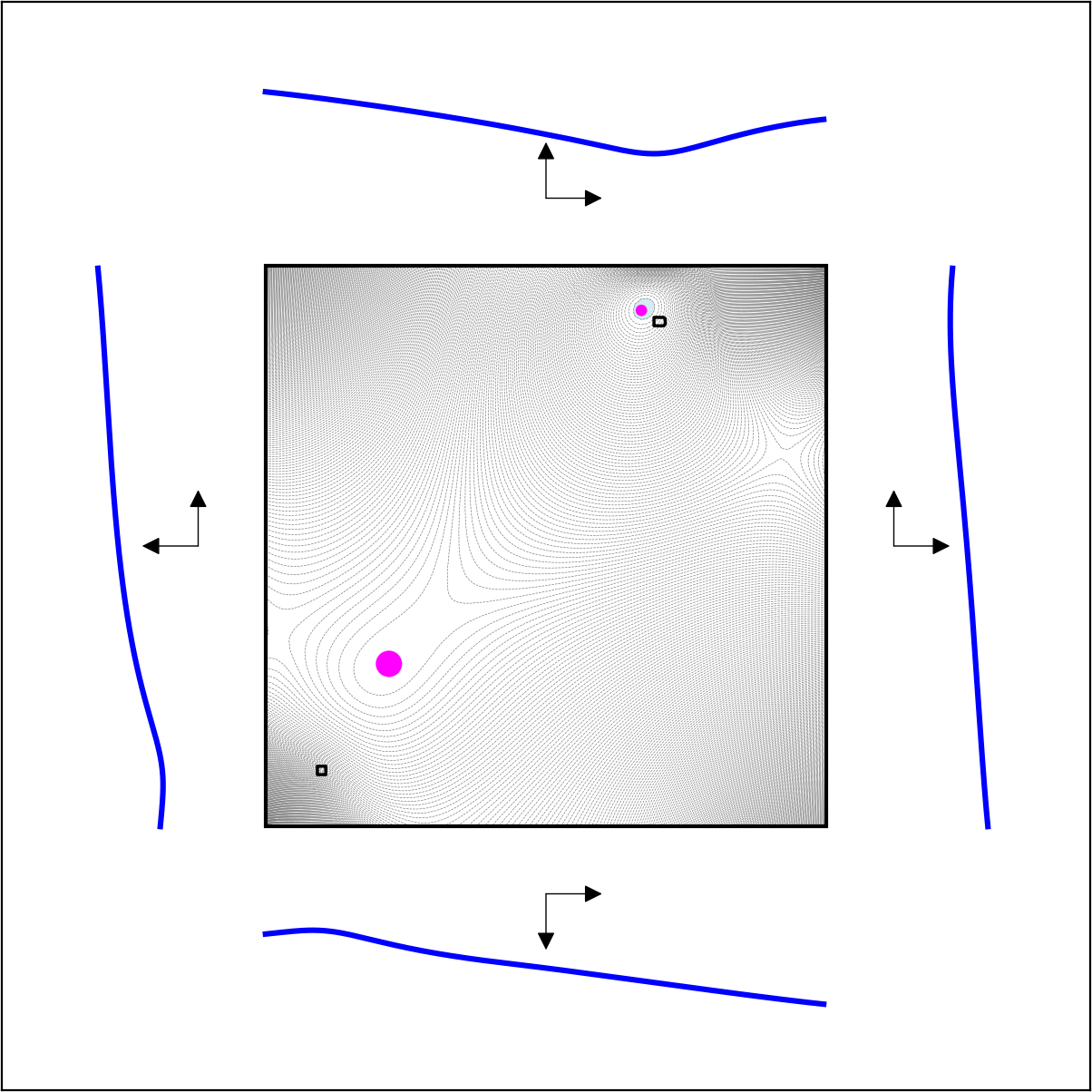}
        \label{fig:two_small_sizes_img5}
    \end{subfigure}
    \hfill
    \begin{subfigure}{0.3\textwidth}
	\includegraphics[width=\textwidth]{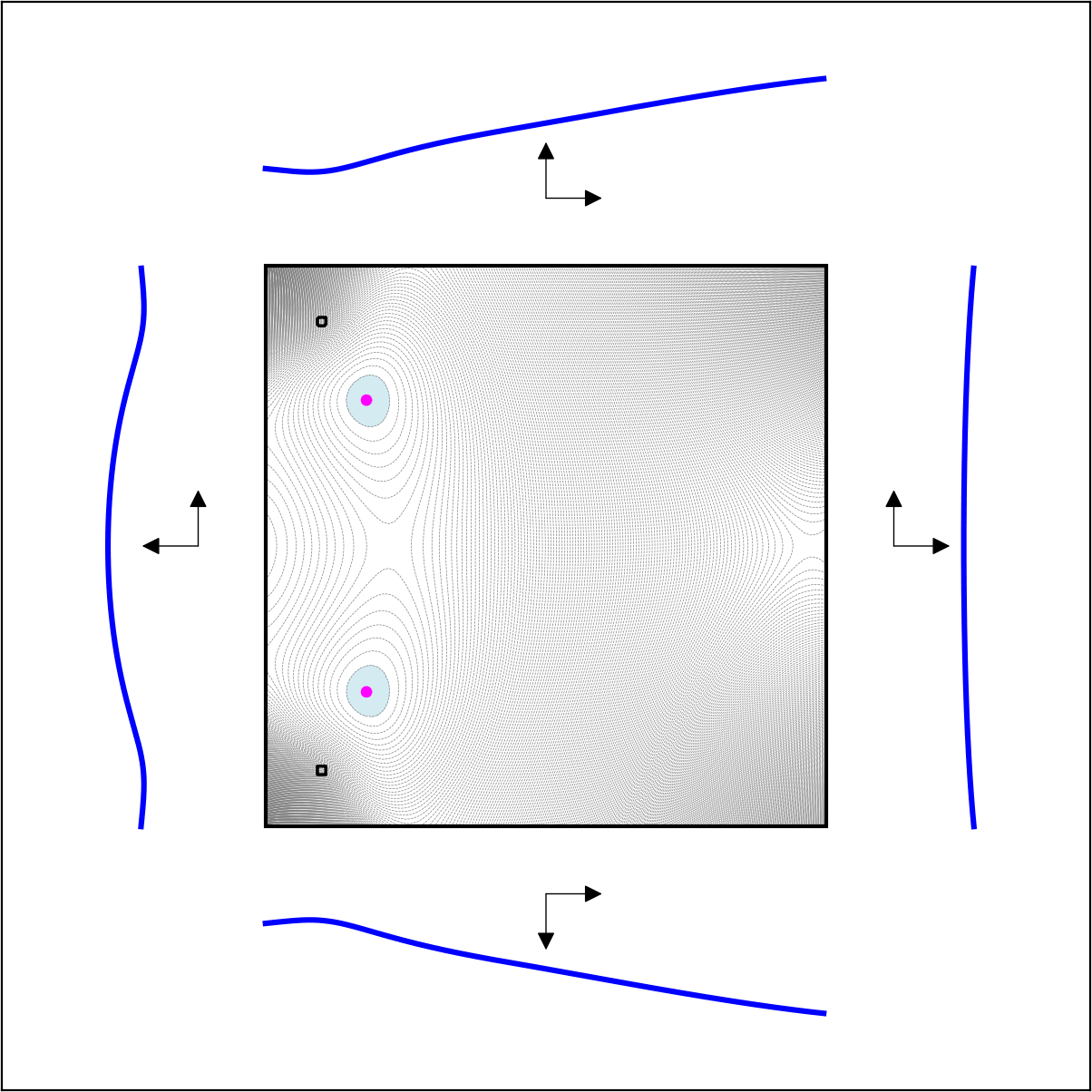}
        \label{fig:two_small_sizes_img6}
    \end{subfigure}
\caption{Comparison of the detection of two extremely small-sized subregions using $g = 1$ (top) and $g = |x|$ (bottom).}
    \label{fig:two_small_sizes}
\end{figure}

\begin{figure}[htbp!]
    \centering
    \begin{subfigure}{0.3\textwidth}
	\includegraphics[width=\textwidth]{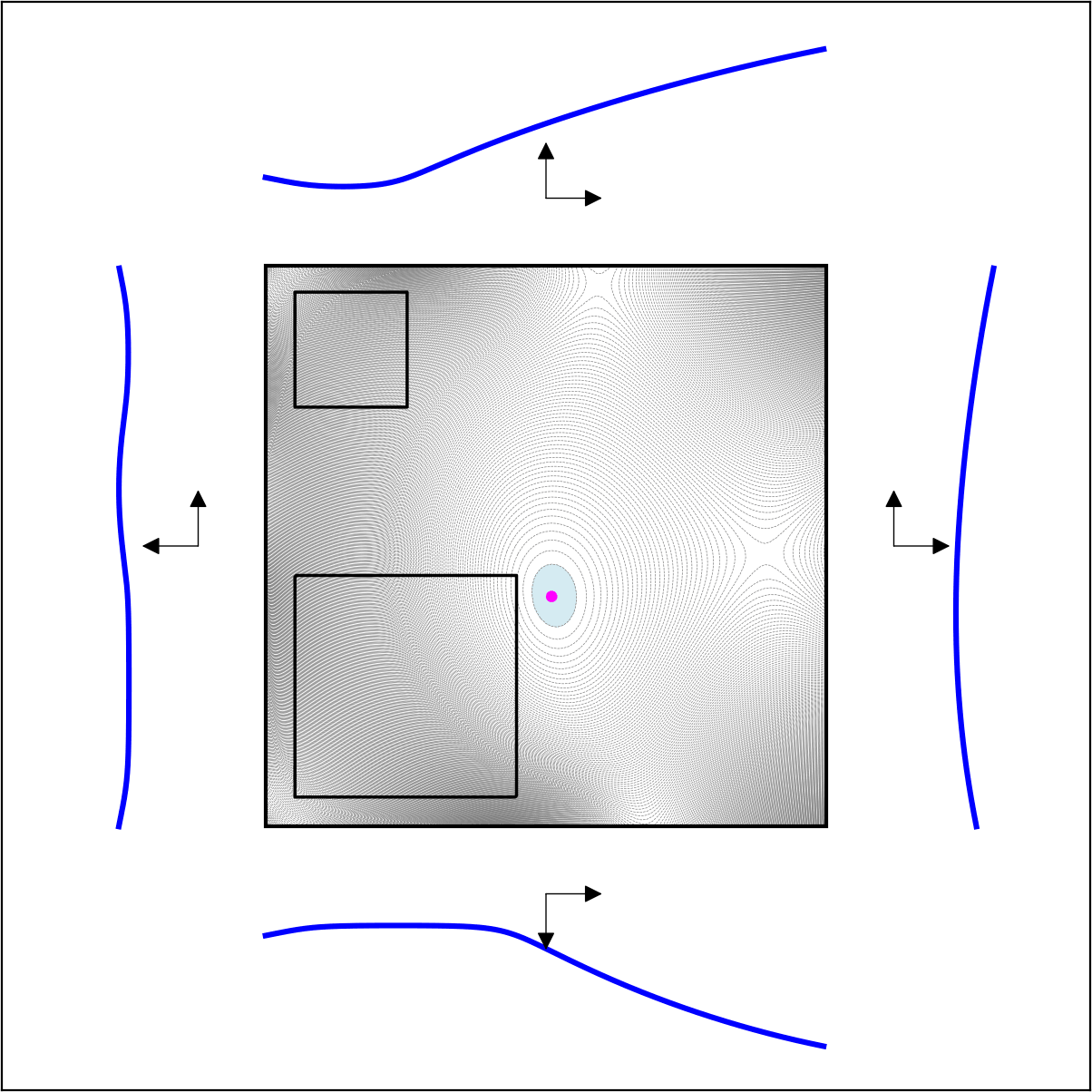}
        \label{fig:two_unequal_sizes_img1}
    \end{subfigure}
    \hfill
    \begin{subfigure}{0.3\textwidth}
	\includegraphics[width=\textwidth]{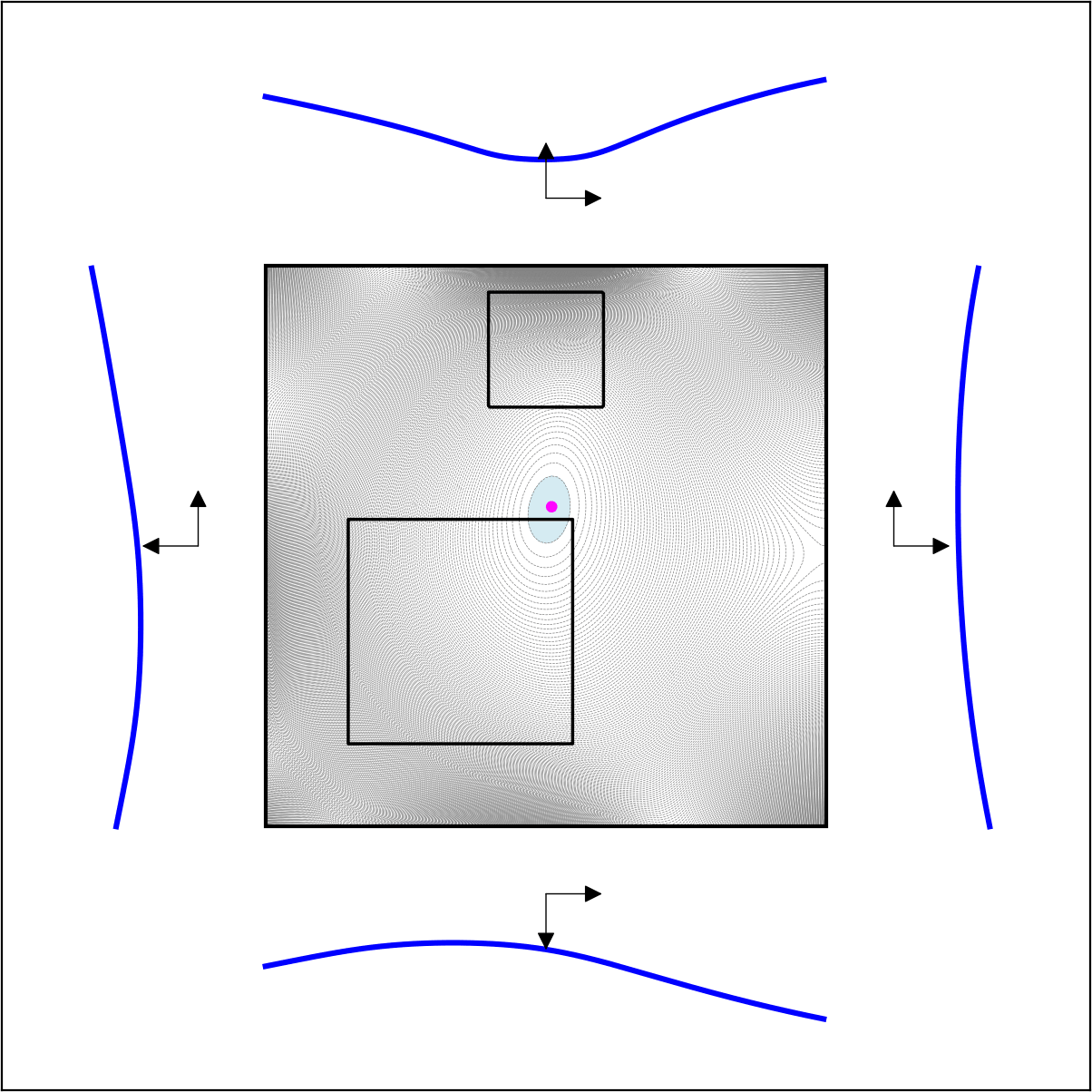}
        \label{fig:two_unequal_sizes_img2}
    \end{subfigure}
    \hfill
    \begin{subfigure}{0.3\textwidth}
	\includegraphics[width=\textwidth]{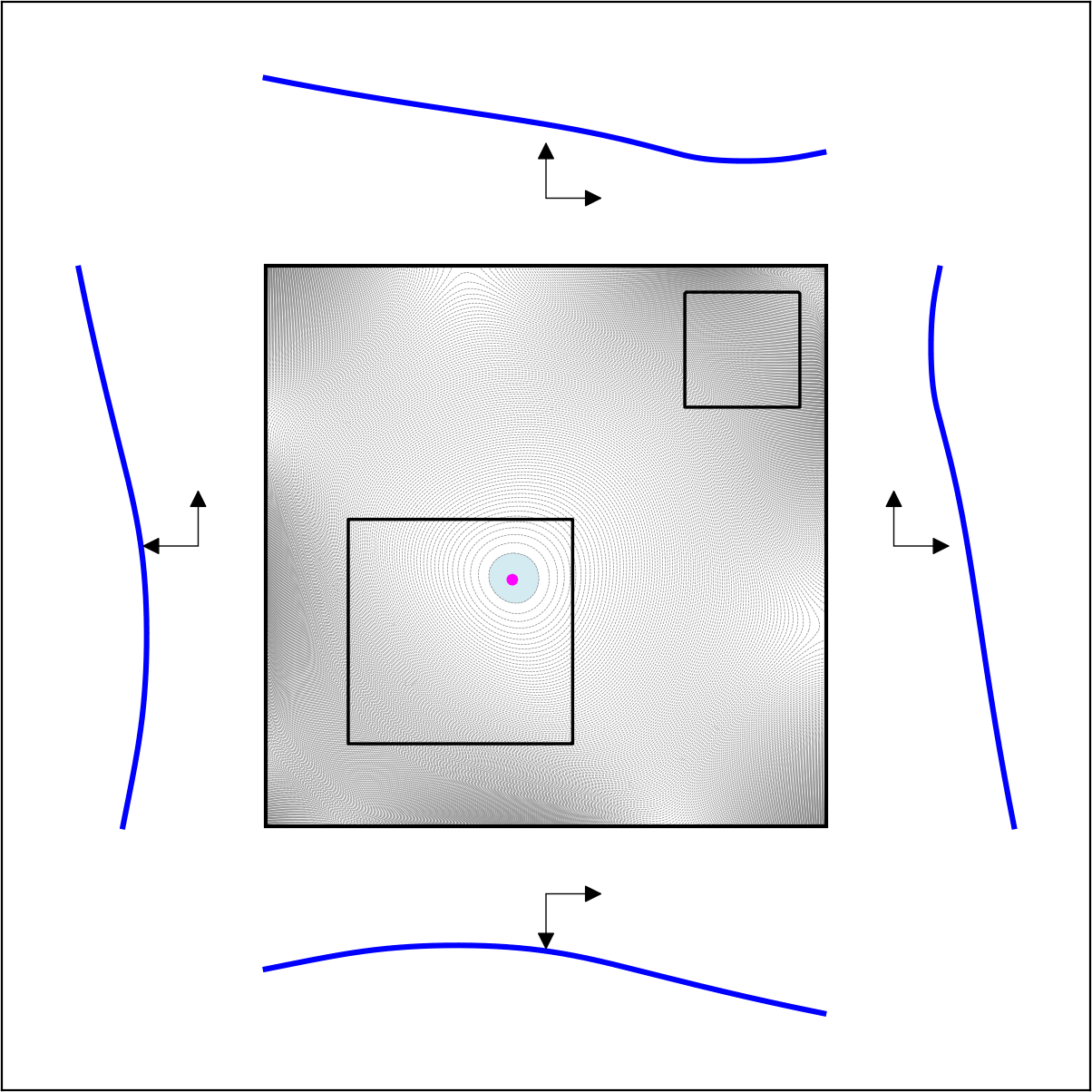}
        \label{fig:two_unequal_sizes_img3}
    \end{subfigure}
\caption{Comparison of the detection of unequal sizes subregions using $g = 1$ (top).}
    \label{fig:two_unequal_square_sizes}
\end{figure}

\begin{figure}[htbp!]
    \centering
    \begin{subfigure}{0.3\textwidth}
	\includegraphics[width=\textwidth]{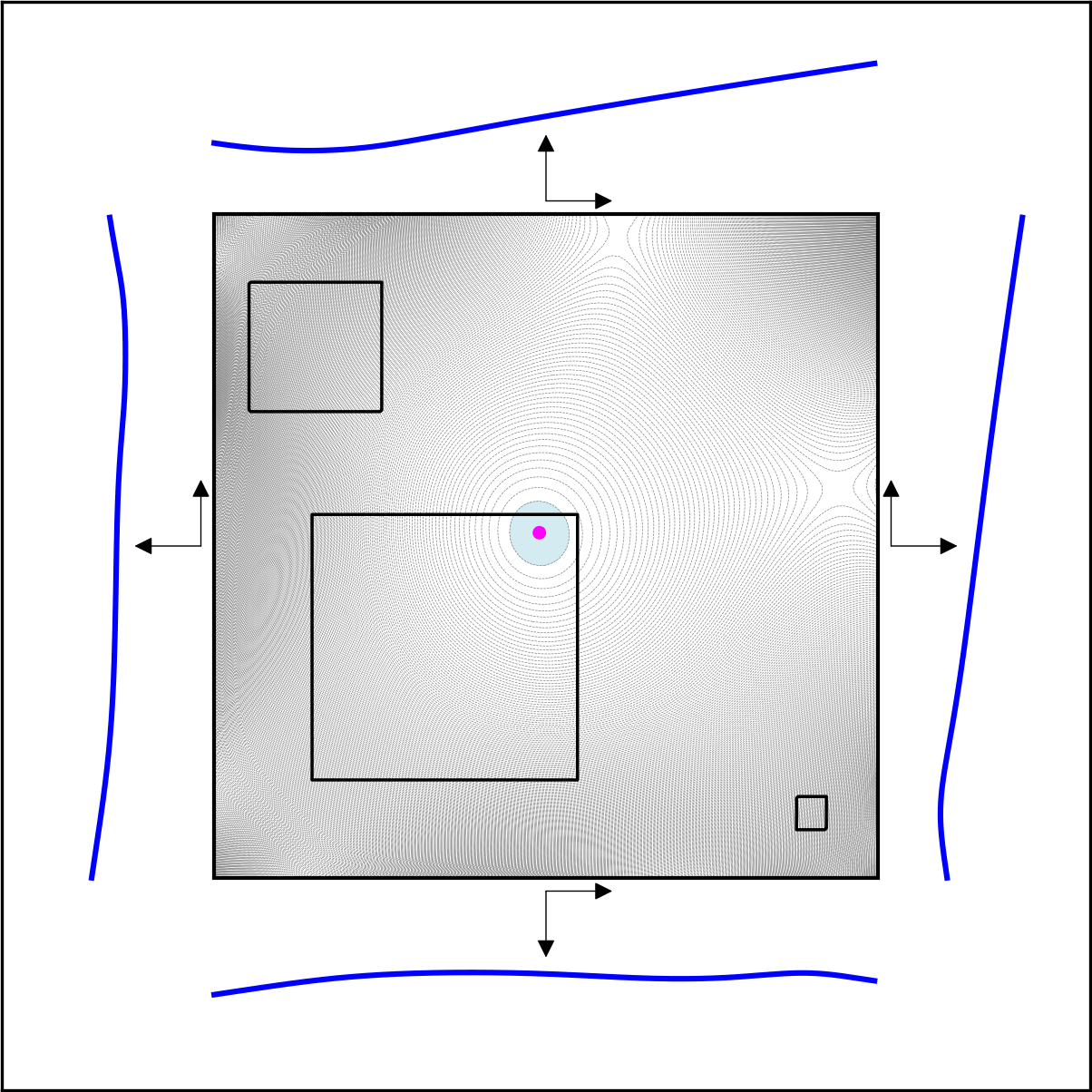}
        \label{fig:three_subregions_img1}
    \end{subfigure}
    \hfill
    \begin{subfigure}{0.3\textwidth}
	\includegraphics[width=\textwidth]{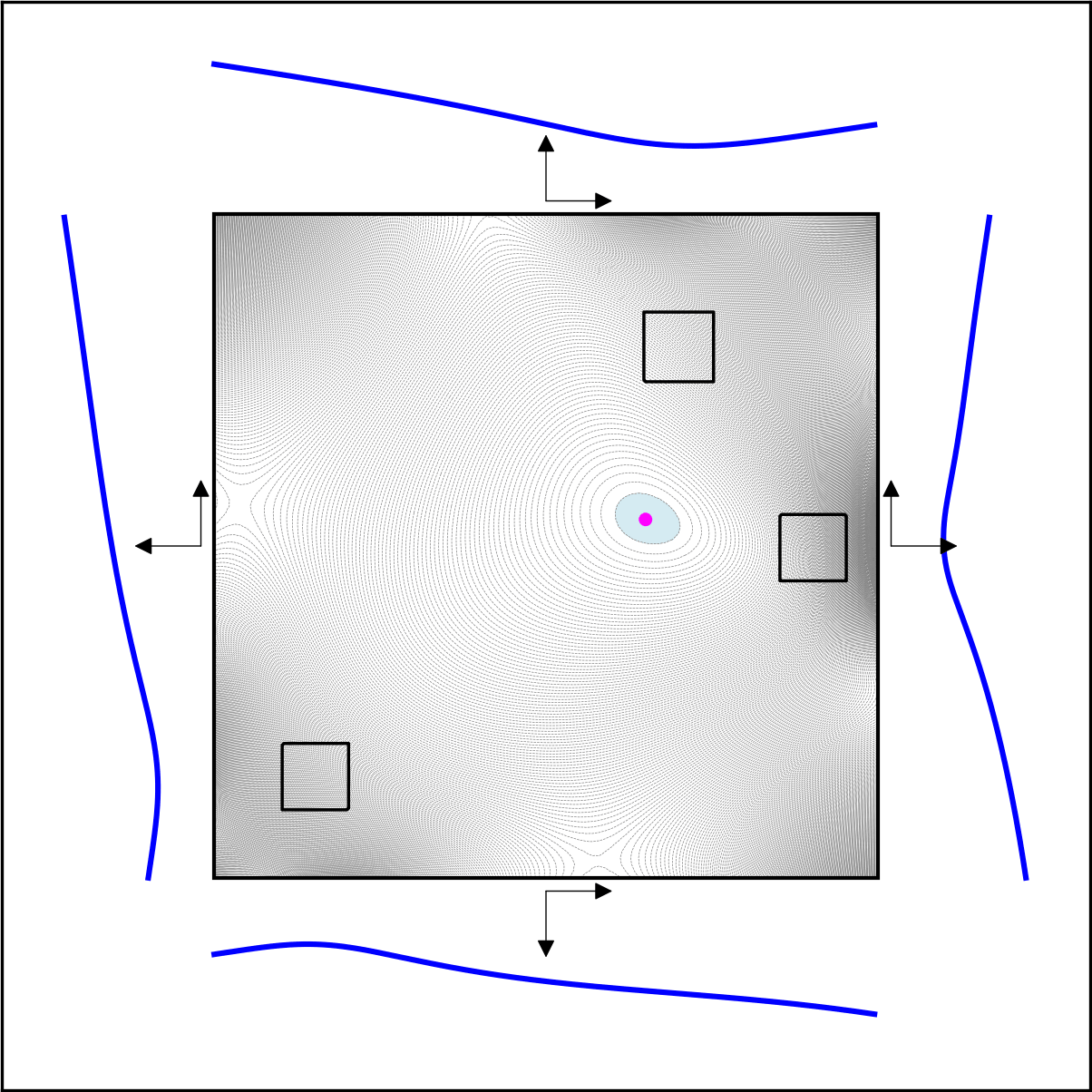}
        \label{fig:three_subregions_img2}
    \end{subfigure}
    \hfill
    \begin{subfigure}{0.3\textwidth}
	\includegraphics[width=\textwidth]{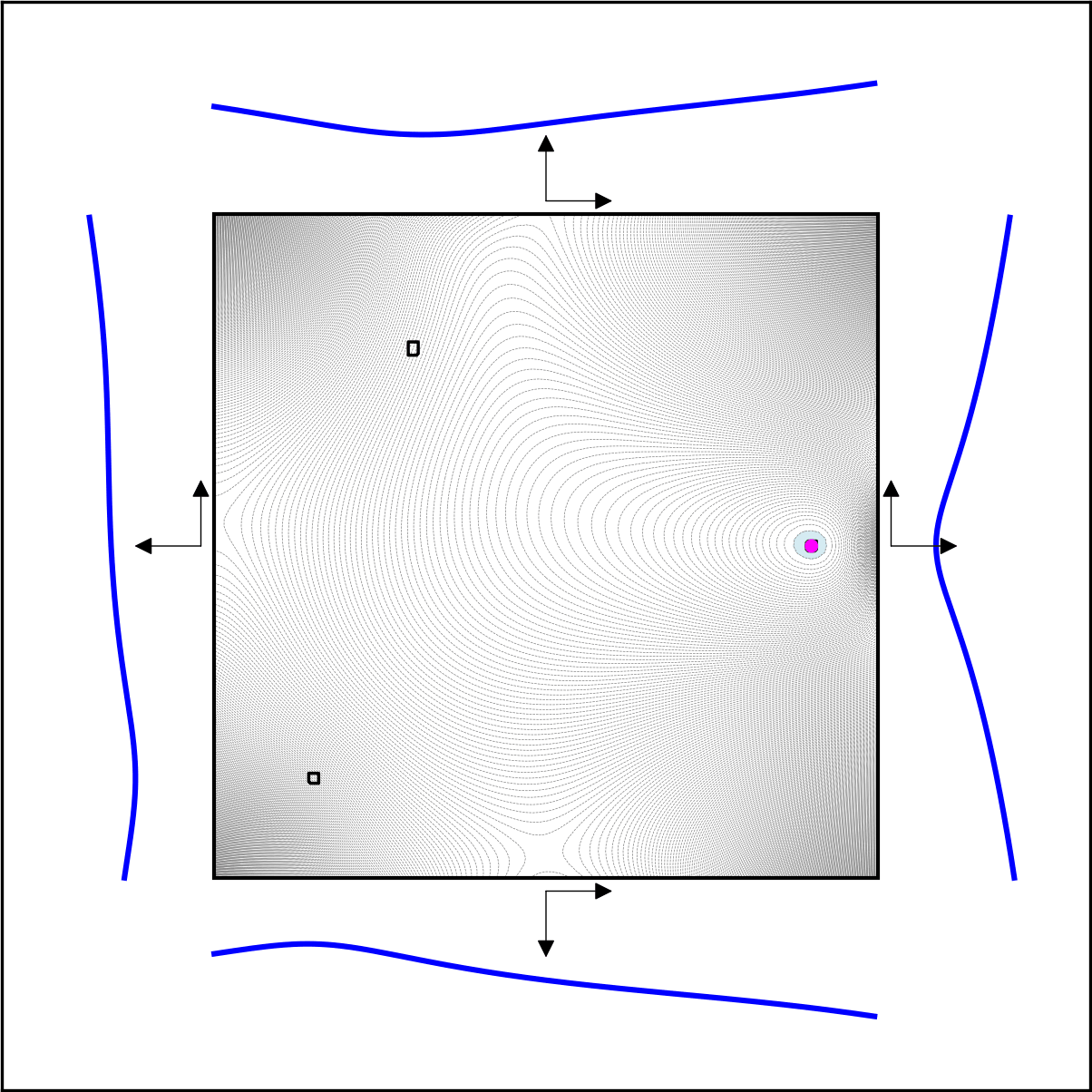}
        \label{fig:three_subregions_img3}
    \end{subfigure}
\caption{Comparison of the detection of three subregions using $g = 1$.}
    \label{fig:three_subregions_test}
\end{figure}

\begin{figure}[htbp!]
    \centering
    \begin{subfigure}{0.3\textwidth}
	\includegraphics[width=\textwidth]{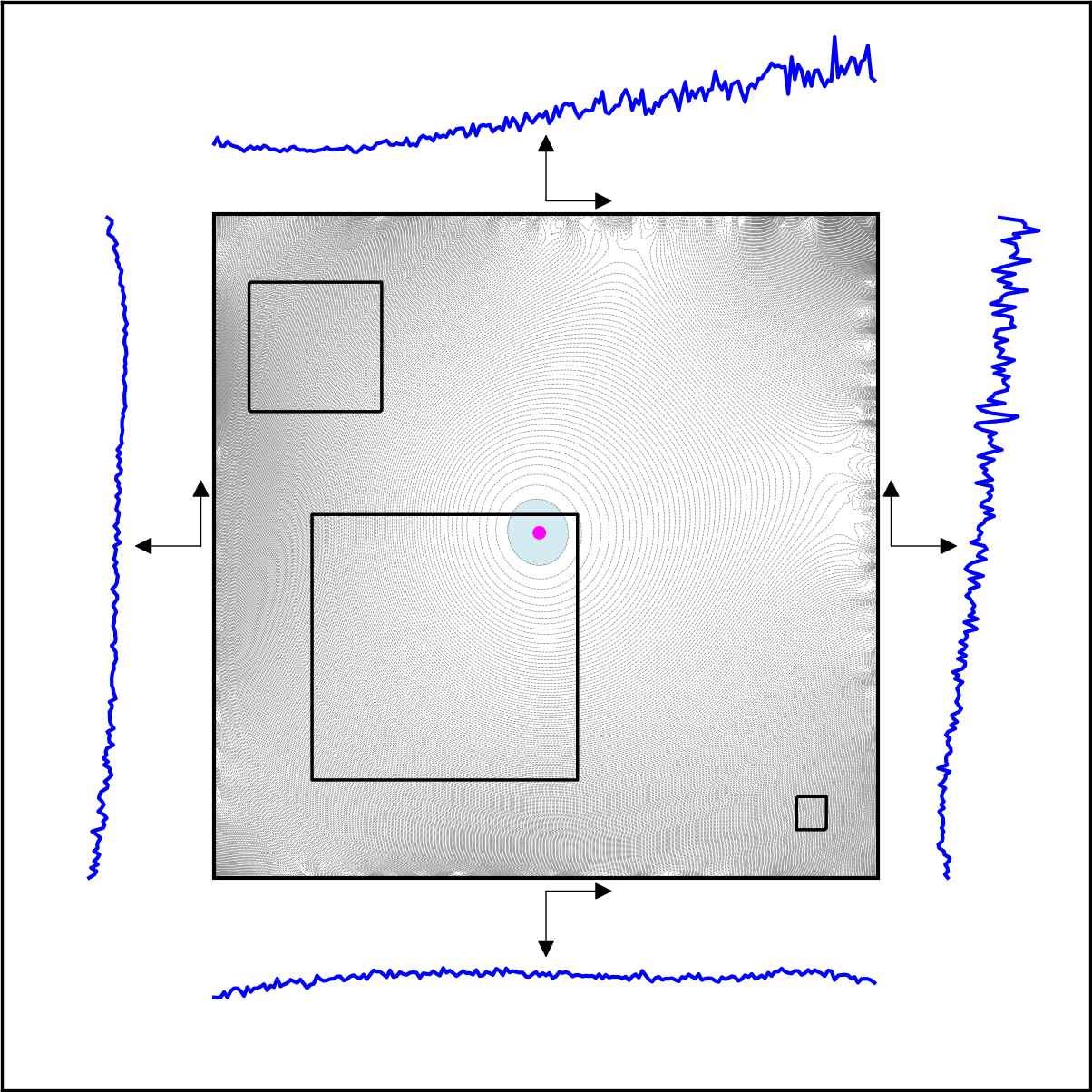}
        \label{fig:three_subregions_img1}
    \end{subfigure}
    \hfill
    \begin{subfigure}{0.3\textwidth}
	\includegraphics[width=\textwidth]{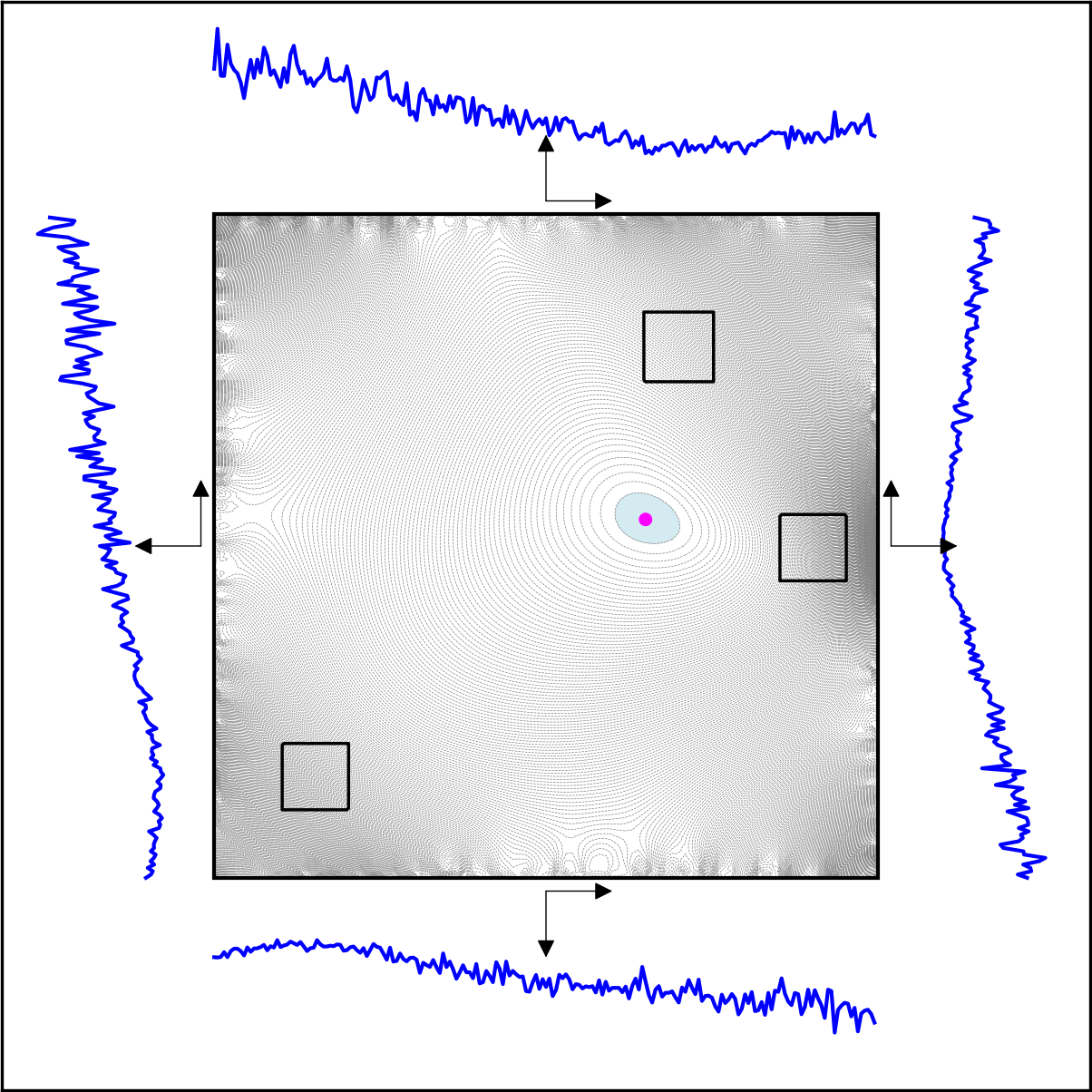}
        \label{fig:three_subregions_img2}
    \end{subfigure}
    \hfill
    \begin{subfigure}{0.3\textwidth}
	\includegraphics[width=\textwidth]{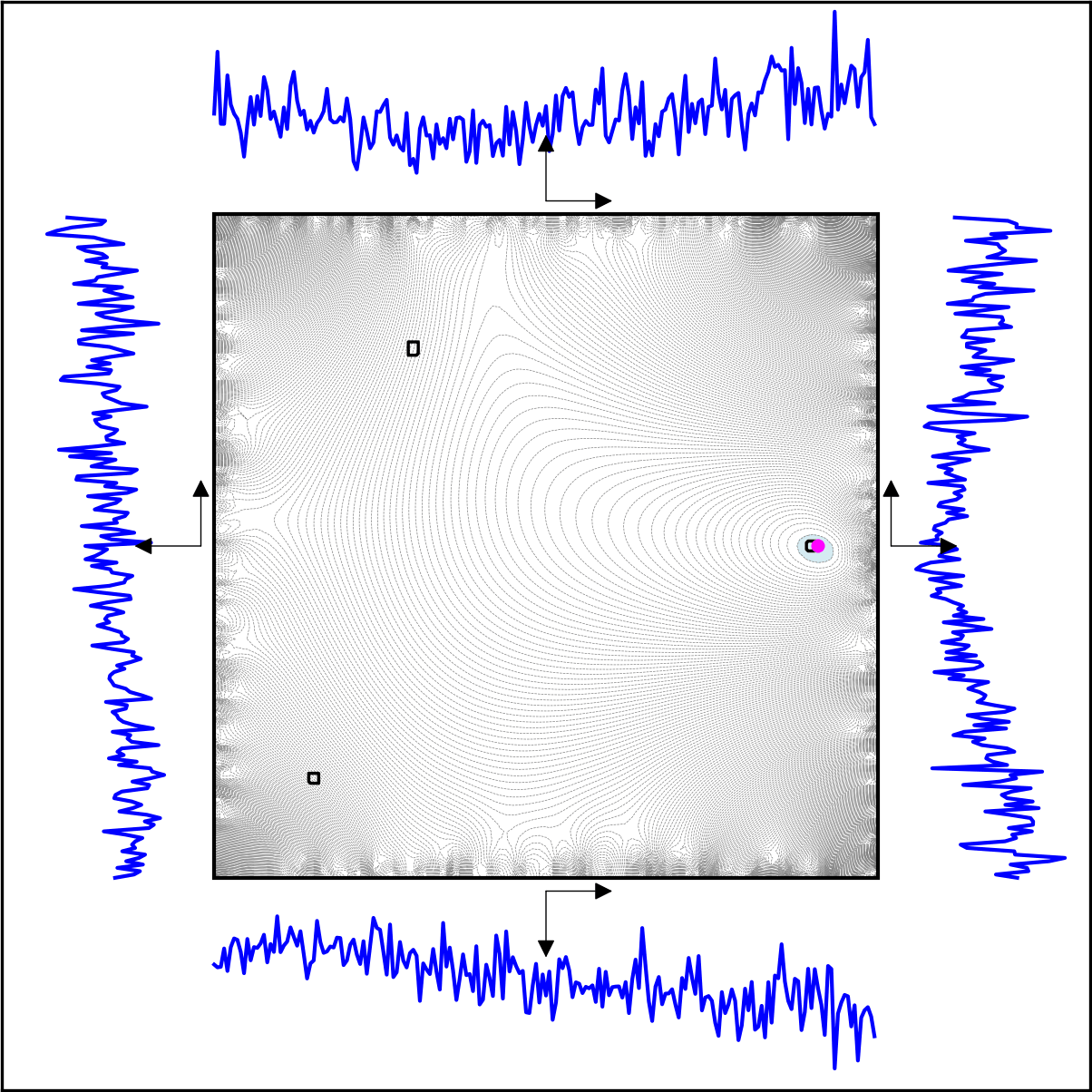}
        \label{fig:three_subregions_img3}
    \end{subfigure}
\caption{Comparison of the detection of three subregions using $g = 1$ (top) with $10\%$ contaminated measurement.}
    \label{fig:noisy_three_subregions_test}
\end{figure}

\newpage

\subsection{Statistical detection of the contact region}\label{subsec:statistical_detection}
We propose a numerical framework for the statistical identification of the contact region 
$\omegastar \Subset \varOmega$ in the presence of measurement noise.
The approach combines the computation of the topological gradient with a statistical decision criterion derived from Corollary~\ref{coro:clt_topograd}.

The main idea is to evaluate, at each candidate point $(x_i,y_j)\in\varOmega$, the sign of the localized projection
\[
\left\langle \delta{\JJ}, {\Proj}_{i,j} \right\rangle_{L^2(\varOmega)} .
\]
A negative value indicates that introducing a small inclusion at that location decreases the cost functional, which is interpreted as evidence of contact.

The corresponding statistical hypothesis test is
\[
(\mathcal{H}_0) : \left\langle \delta{\JJ}, {\Proj}_{i,j} \right\rangle_{L^2(\varOmega)} \ge 0,
\qquad
(\mathcal{H}_1) : \left\langle \delta{\JJ}, {\Proj}_{i,j} \right\rangle_{L^2(\varOmega)} < 0 .
\]
Rejection of $(\mathcal{H}_0)$ therefore indicates the presence of contact at $(x_i,y_j)$ with a prescribed confidence level.
Since only noisy realizations of the gradient are available, the test is performed using Monte Carlo sampling.

The numerical procedure consists of the following steps.

\textit{Step 1: Initialization.}
\begin{itemize}
\item[-] Generate the computational meshes $\mathcal{T}_h$ and $\mathcal{T}_h^{\star}$.
\item[-] Specify a reference contact region $\omegastar$ and the physical parameters.
\end{itemize}

\textit{Step 2: Reference computation (noise-free case).}
We compute the reference forward and adjoint fields corresponding to the boundary datum ${\fmeas}$, following essentially Steps~1--3 of Algorithm~\ref{alg:one-shot}.

\textit{Step 3: Monte Carlo sampling under noisy data.}
To quantify the variability induced by measurement noise, we generate independent perturbations 
${\noise}^{n}$ with prescribed standard deviation.
For each realization $n=1,\dots,N_{\mathrm{mc}}$:

\begin{itemize}
\item[-] Solve the state problem with perturbed boundary data ${\fmeas} + {\noise}^{n}$.
\item[-] Solve the corresponding adjoint problem.
\item[-] Compute the noisy topological gradient
\[
\delta{\JJ}_n
=
\iu[{\fmeas} + {\noise}^{n}]
\,\rv[{\fmeas} + {\noise}^{n}]
-
\ru[{\fmeas} + {\noise}^{n}]
\,\iv[{\fmeas} + {\noise}^{n}].
\]
\end{itemize}

This produces a family of random fields $\{\delta{\JJ}_n\}$ approximating the distribution of the gradient.

\textit{Step 4: Local statistical evaluation.}
We introduce a set of scanning points
\[
\{(x_i,y_j)\}_{i,j=1}^{N_{\mathrm{scan}}} \subset \varOmega,
\]
and define localized test functions
\[
{\Proj}_{i,j}(x,y)
=
\exp\!\Big(
-\frac{(x-x_i)^2+(y-y_j)^2}{\varsigma^2}
\Big),
\]
where $\varsigma$ controls the spatial resolution.

For each $(x_i,y_j)$:

\begin{itemize}
\item[-] Compute projections $\delta{\JJ}_n^{(i,j)} = \displaystyle\intO{ \delta{\JJ}_n \, {\Proj}_{i,j} }$.

\item[-] Compute the empirical mean and standard deviation
\[
\overline{\delta{\JJ}}^{(i,j)}
=
\frac{1}{N_{\mathrm{mc}}}
\sum_{n=1}^{N_{\mathrm{mc}}}
\delta{\JJ}_n^{(i,j)},
\qquad
s^{(i,j)}
=
\sqrt{
\frac{1}{N_{\mathrm{mc}}-1}
\sum_{n=1}^{N_{\mathrm{mc}}}
\big(
\delta{\JJ}_n^{(i,j)}
-
\overline{\delta{\JJ}}^{(i,j)}
\big)^2
}.
\]

\item[-] Construct the asymptotic confidence interval
 \[
    \mathrm{CI}_{1-\alpha}^{(i,j)} =
    \Big[\overline{\delta{\JJ}}^{(i,j)}-
z_{1-\alpha/2}
\frac{s^{(i,j)}}{\sqrt{N_{\mathrm{mc}}}},\quad
    \overline{\delta{\JJ}}^{(i,j)}
+
z_{1-\alpha/2}
\frac{s^{(i,j)}}{\sqrt{N_{\mathrm{mc}}}}\Big].
    \]
\end{itemize}

\textit{Step 5: Statistical decision and visualization.}
The null hypothesis is rejected whenever the upper confidence bound is negative:
\[
\overline{\delta{\JJ}}^{(i,j)}
+
z_{1-\alpha/2}
\frac{s^{(i,j)}}{\sqrt{N_{\mathrm{mc}}}}
< 0.
\]
Points satisfying this condition are classified as contact with confidence level $1-\alpha$.

The map
\[
(x_i,y_j)
\longmapsto
\overline{\delta{\JJ}}^{(i,j)}
+
z_{1-\alpha/2}
\frac{s^{(i,j)}}{\sqrt{N_{\mathrm{mc}}}}
\]
therefore provides a spatial confidence indicator for the location of $\omegastar$.

\begin{remark}
The Gaussian test function
\[
{\Proj}_{i,j}(x,y)
=
\exp\!\Big(
-\frac{(x-x_i)^2+(y-y_j)^2}{\varsigma^2}
\Big)
\]
is used to obtain a smooth spatial localization around each scanning point.
This choice improves numerical stability and reduces sensitivity to noise by performing a local averaging of the topological gradient.
\end{remark}

The proposed procedure provides a statistically robust identification of the contact region~$\omegastar$ by combining the deterministic topological gradient with Monte Carlo sampling to account for measurement noise.
For each grid point $\{(\bx_i,\by_j)\}_{i,j=1}^{N_{\mathrm{scan}}}$, a confidence interval is constructed based on the central limit theorem.
Points for which the upper confidence bound
\begin{equation}\label{eq:upper_confident_bound}
\mathrm{CI}_{\mathrm{upper}}^{(i,j)}
=
\overline{\delta{\JJ}}^{(i,j)} 
+ 
z_{1-\alpha/2} 
\frac{s^{(i,j)}}{\sqrt{N_{\mathrm{mc}}}}
\end{equation}
is negative correspond, at confidence level $1-\alpha$, to locations where contact is detected.
The resulting detection map thus identifies the probable location of $\omegastar$ while quantifying the associated uncertainty.

The performance of the proposed statistical topological gradient is assessed by reconstructing a contact region defined as in Figure~\ref{fig:simple_img1}, namely a circle centered at $(0,0)$ with radius $0.1$. More precisely, to delineate the most significant contact region, or ``confidence zone'', we identify all points where the statistical topological gradient is negative. At these locations, the null hypothesis $(\mathcal{H}_0)$ is rejected at the 95\% confidence level (i.e., $\alpha=0.05$), corresponding to the critical value $z_{1-\alpha/2}=1.96$ of the standard normal distribution. Within this subset, points lying within $5\%$ of the global minimum define the ``most negative red zone'', highlighting the region of strongest contact. The obtained results are illustrated in Figure~\ref{fig:stat_detection}.

\begin{figure}[htbp!]
    \centering
    \begin{subfigure}{0.34\textwidth}
        \includegraphics[width=\textwidth]{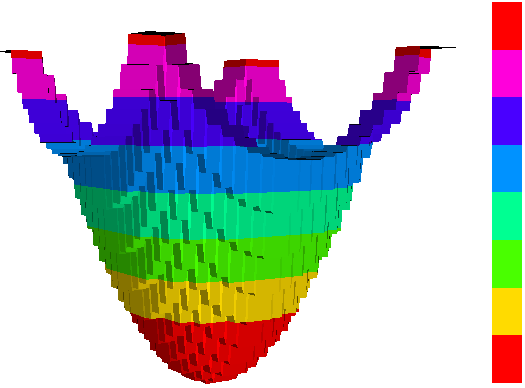}
        \label{fig:stat_neg_zone}
    \end{subfigure}
    \hfill
    \begin{subfigure}{0.25\textwidth}
        \includegraphics[width=\textwidth]{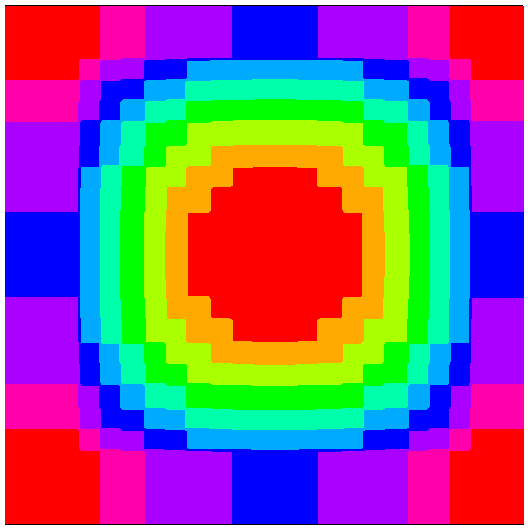}
        \label{fig:stat_conf_zone}
    \end{subfigure}
    \hfill
    \begin{subfigure}{0.25\textwidth}
        \includegraphics[width=\textwidth]{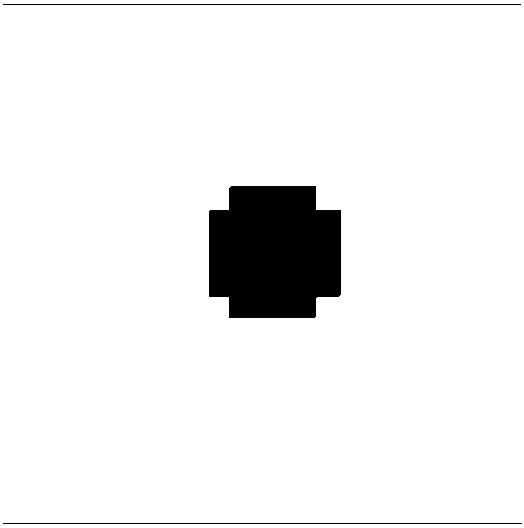}
        \label{fig:stat_strong_contact}
    \end{subfigure}
\caption{Statistical detection of the contact region $\omegastar$: 
(a) most negative topological gradient points, 
(b) confidence zones rejecting $(\mathcal{H}_0)$, and 
(c) filtered red zone highlighting the strongest contact.
Here, $\delta=[0.,\, 0.001,\, 0.005,\, 0.01,\, 0.02,\, 0.05,\, 0.1]$ and ($N_{\mathrm{mc}}, N_{\mathrm{scan}}) = (100, 20)$.}
\label{fig:stat_detection}
\end{figure}

Figure~\ref{gradientprojectcv} illustrates the Monte Carlo convergence of the projected topological gradient. 
To quantify the convergence, we define the error at each scanning point $(\bx_i,\by_j)$ as
\[
\mathrm{Err}_{i,j}(n)
= \Biggl|
\frac{1}{n} \sum_{k=1}^{n} 
\langle \delta \JJ_k, \Proj_{i,j} \rangle
-
\langle \delta \JJ, \Proj_{i,j} \rangle
\Biggr|.
\]
From a theoretical standpoint, the central limit theorem predicts that this error decreases at the rate $N_{\mathrm{mc}}^{-1/2}$, since the variance of the empirical mean scales like $1/N_{\mathrm{mc}}$. 
Numerically, this behavior is confirmed by the computed error curve, which decays almost linearly on the log--log scale and runs parallel to the reference slope. 
The excellent agreement demonstrates that the numerical procedure is statistically stable and that the Monte Carlo averaging converges at the theoretically expected rate.

\begin{figure}[htbp!]
\centering
\includegraphics[width=0.8\textwidth]{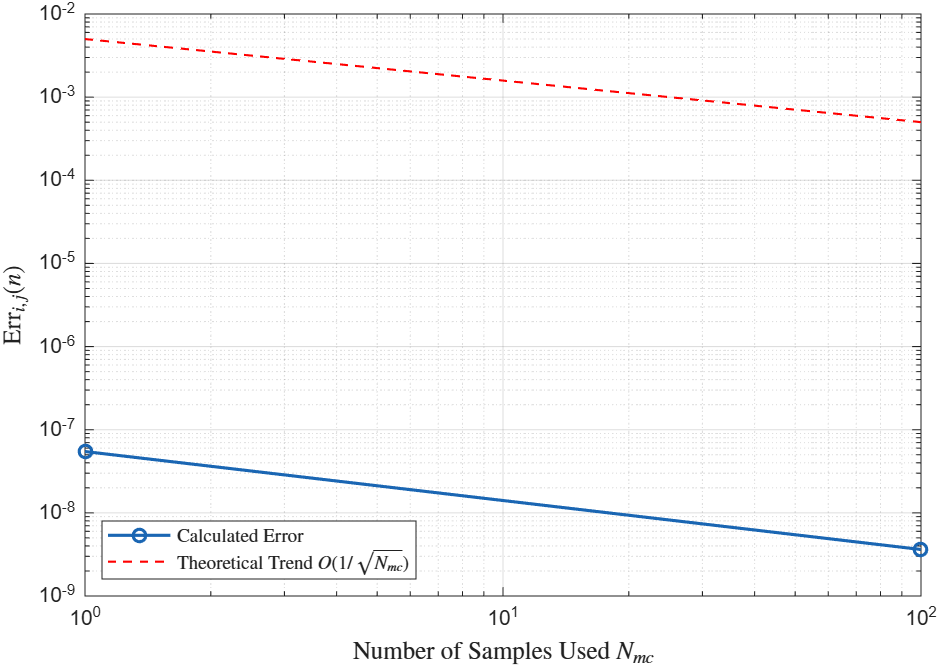} 
\caption{Monte Carlo convergence of the projected topological gradient.}
\label{gradientprojectcv}
\end{figure}

Further statistical detection experiments are presented in Figures~\ref{fig:onesq_shapes}--\ref{fig:threesq_shapes}. 
These figures show the detection maps obtained with 
$(\delta, N_{\mathrm{mc}}, N_{\mathrm{scan}}) = (0.1, 100, 40)$. 
In line with the detection principle outlined in Remark~\ref{rem:isolevel_detection}, the highlighted regions (shown in light blue) correspond to scanning points $(x_i, y_j)$ for which the upper confidence bound~\eqref{eq:upper_confident_bound} is negative, that is, where contact is detected at confidence level $1-\alpha$. 
The most negative value of the topological gradient is indicated by a magenta point. 
In configurations where multiple local minima are present, several iso-level regions may appear, reflecting the multiplicity of extremal points.

In Figure~\ref{fig:onesq_shapes}, subfigures~\ref{fig:onesq_img1}--\ref{fig:onesq_img5}, the method correctly localizes a single contact region across a range of positions and sizes.  
For centrally located inclusions (e.g., Figures~\ref{fig:onesq_img1}, \ref{fig:onesq_img4}), the confidence region is compact and nearly symmetric. 
When the square is shifted toward the boundary (e.g., Figures~\ref{fig:onesq_img2} and~\ref{fig:onesq_img5}), the localization quality deteriorates and the extremal level of the topological gradient is no longer attained at the true inclusion. 
In these configurations, the minimum is shifted away from the exact location, a behavior that is already observed in the absence of noise and therefore appears to stem from geometric effects rather than statistical variance. 
 
Figure~\ref{fig:twosq_shapes}, subfigures~\ref{fig:twosq_img1}--\ref{fig:twosq_img5}, demonstrates that the procedure identifies multiple disconnected contact zones.  
When the inclusions are well separated (e.g., Figures~\ref{fig:twosq_img3}, \ref{fig:twosq_img4}), two distinct confidence regions emerge, each localized near the corresponding square.  
As the distance between inclusions decreases (e.g., Figures~\ref{fig:twosq_img1}, \ref{fig:twosq_img2}), the detected regions exhibit partial interaction but remain clearly bimodal.  
This behavior is consistent with the linearity of the projected topological gradient and confirms that the statistical test operates locally despite global field interactions.

In Figure~\ref{fig:threesq_shapes}, three contact zones are considered under increasingly asymmetric layouts.  
For moderately separated configurations (e.g., Figures~\ref{fig:threesq_img1}--\ref{fig:threesq_img2}), the method resolves three distinct statistically significant regions.  
When inclusions approach the boundary or are placed near corners (e.g., Figures~\ref{fig:threesq_img4}--\ref{fig:threesq_img5}), the detected regions remain localized but exhibit anisotropic elongation aligned with the underlying sensitivity field.  
No systematic over-detection is observed, suggesting that the Monte Carlo variance estimate is sufficient for $N_{\mathrm{mc}}=100$ at noise level $\delta=0.1$.

Across the tested configurations, the method demonstrates the ability to identify the correct number of sufficiently separated contact components, to localize them with spatial accuracy consistent with the scanning grid resolution ($N_{\mathrm{scan}}=40$), and to maintain robustness under moderate additive noise.

The size of the confidence regions reflects the combined influence of the Gaussian projection and the statistical variability of the estimated gradient. 
The results are in agreement with the CLT-based statistical decision rule and suggest that the proposed framework can provide spatially resolved uncertainty information for contact detection in the considered scenarios.

\begin{figure}[htbp!]
\captionsetup{skip=0pt}
\captionsetup[subfigure]{skip=5pt}
    \centering
    
    \begin{subfigure}{0.19\textwidth}
        \includegraphics[width=\textwidth]{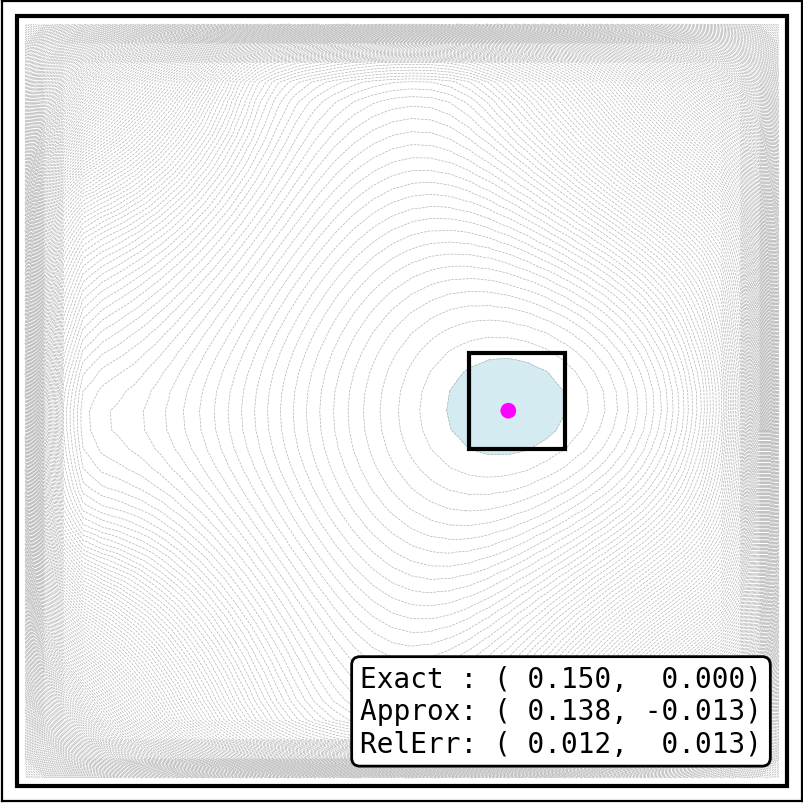}
        \caption{}
        \label{fig:onesq_img1}
    \end{subfigure}
    \hfill
    \begin{subfigure}{0.19\textwidth}
        \includegraphics[width=\textwidth]{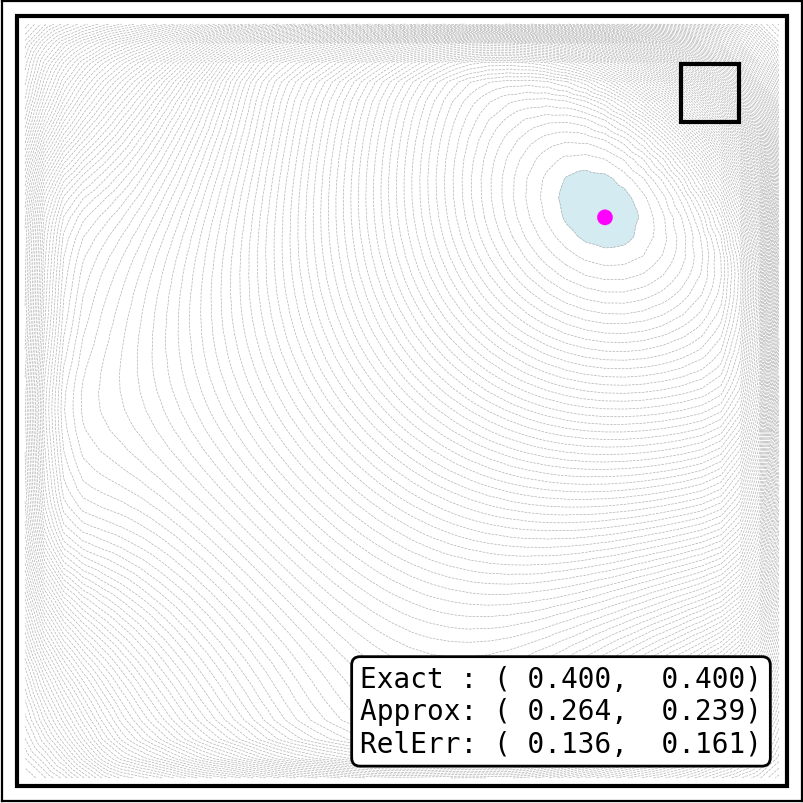}
        \caption{}
        \label{fig:onesq_img2}
    \end{subfigure}
    \hfill
    \begin{subfigure}{0.19\textwidth}
        \includegraphics[width=\textwidth]{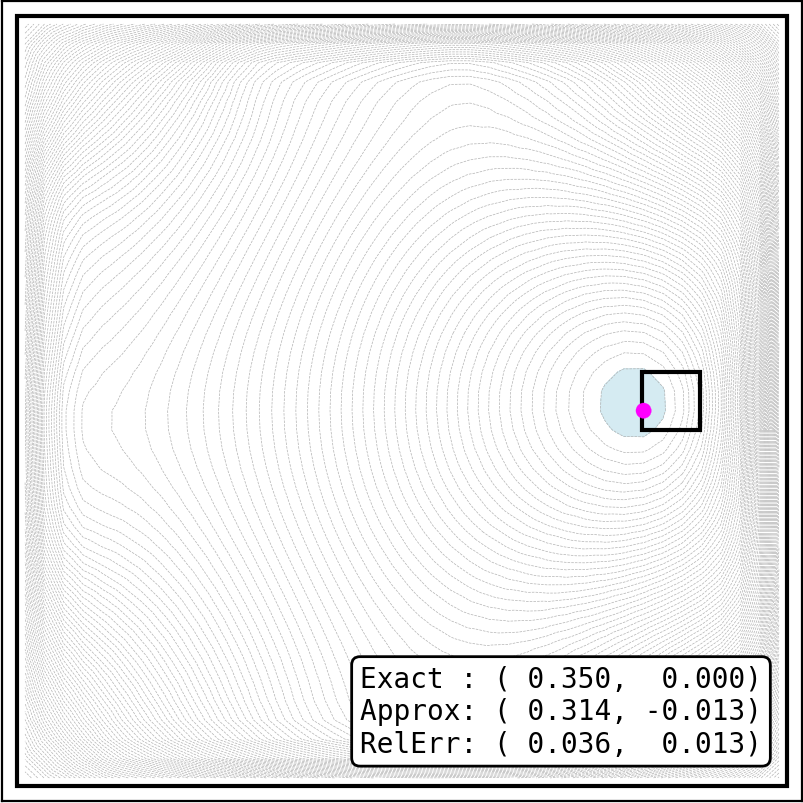}
        \caption{}
        \label{fig:onesq_img3}
    \end{subfigure}
    \hfill
    \begin{subfigure}{0.19\textwidth}
        \includegraphics[width=\textwidth]{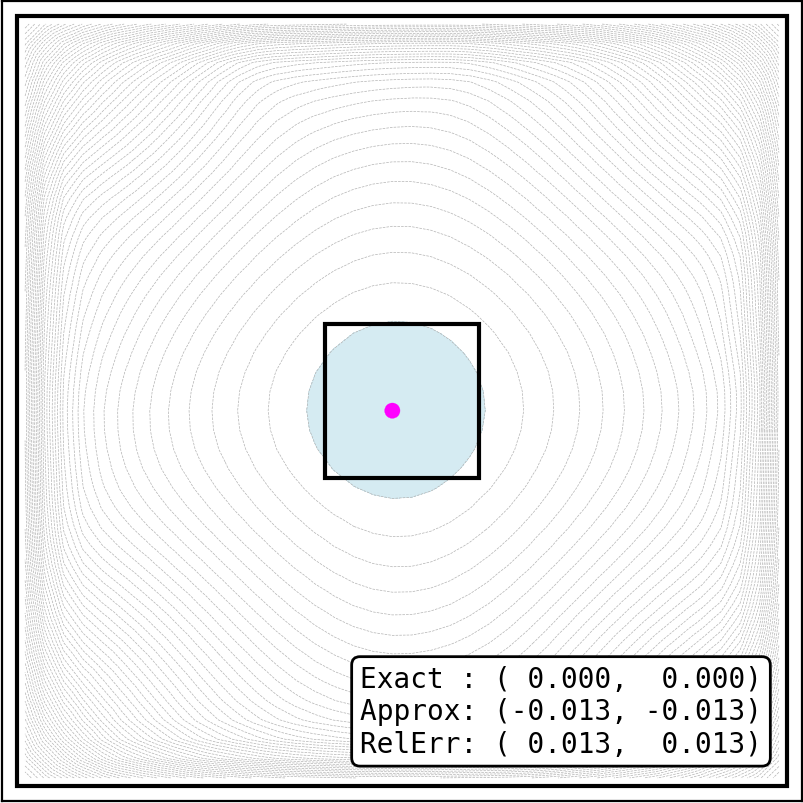}
        \caption{}
        \label{fig:onesq_img4}
    \end{subfigure}
    \hfill
    \begin{subfigure}{0.19\textwidth}
        \includegraphics[width=\textwidth]{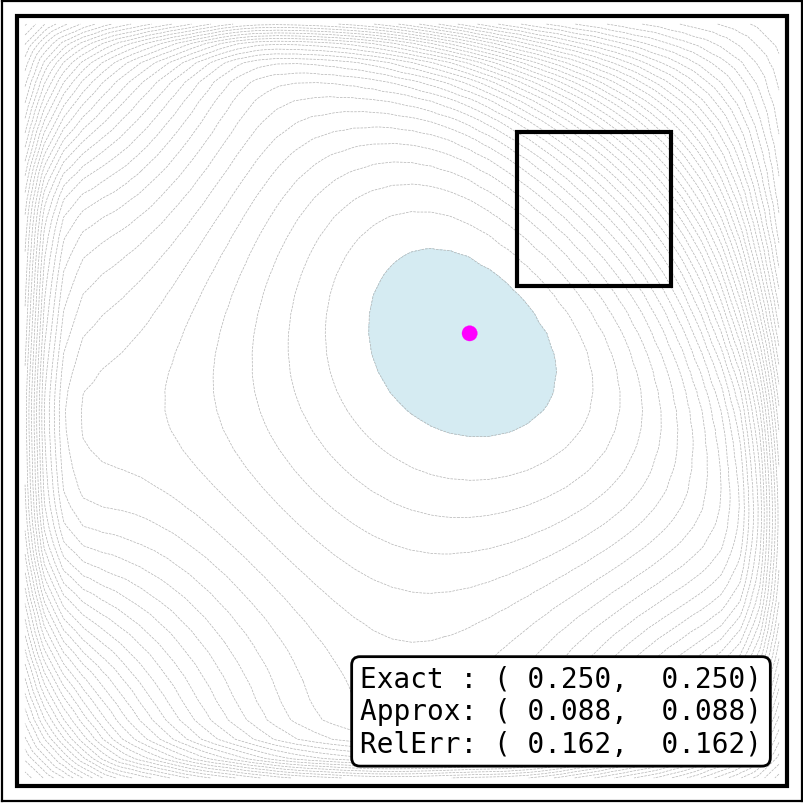}
        \caption{}
        \label{fig:onesq_img5}
    \end{subfigure}
    
    \caption{Results for one-square configurations with ($\delta, N_{\mathrm{mc}}, N_{\mathrm{scan}}) = (0.1, 100, 40)$.}
    \label{fig:onesq_shapes}
\end{figure}

\begin{figure}[htbp!]
\captionsetup{skip=0pt}
\captionsetup[subfigure]{skip=5pt}
    \centering
    
    \begin{subfigure}{0.19\textwidth}
        \includegraphics[width=\textwidth]{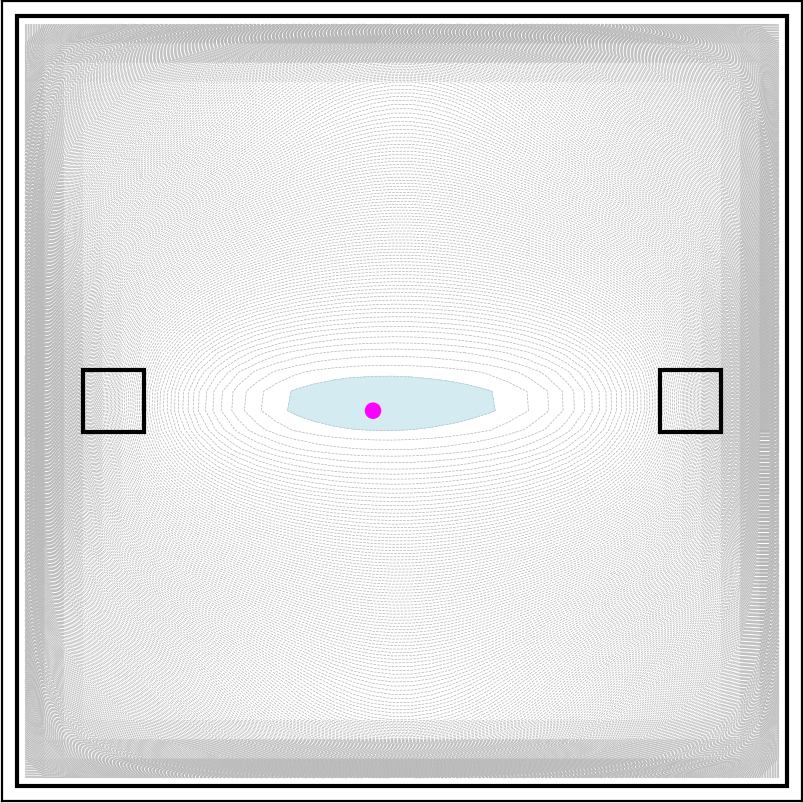}
        \caption{}
        \label{fig:twosq_img1}
    \end{subfigure}
    \hfill
    \begin{subfigure}{0.19\textwidth}
        \includegraphics[width=\textwidth]{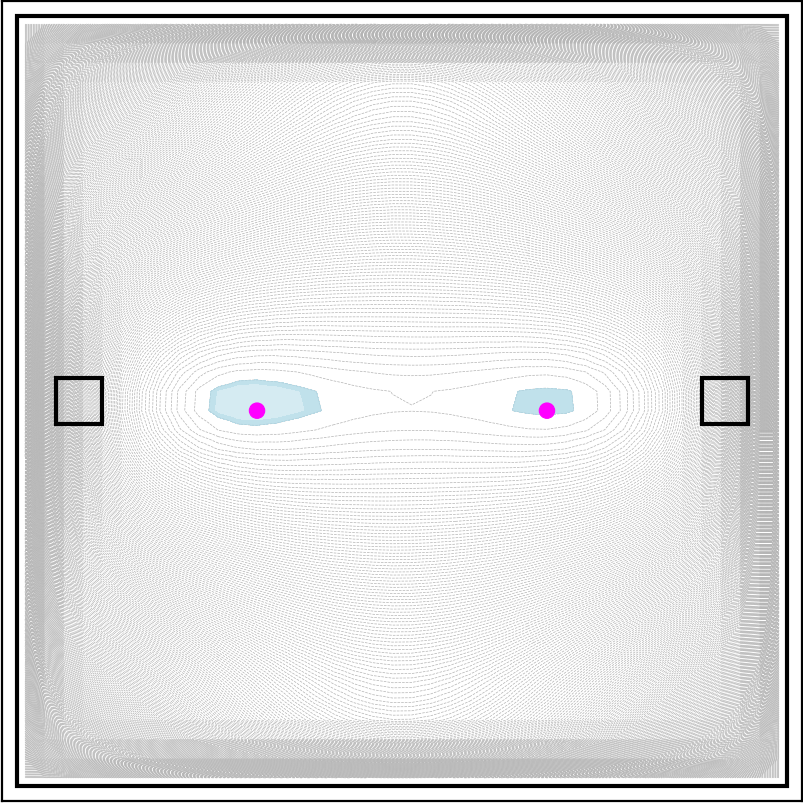}
        \caption{}
        \label{fig:twosq_img2}
    \end{subfigure}
    \hfill
    \begin{subfigure}{0.19\textwidth}
        \includegraphics[width=\textwidth]{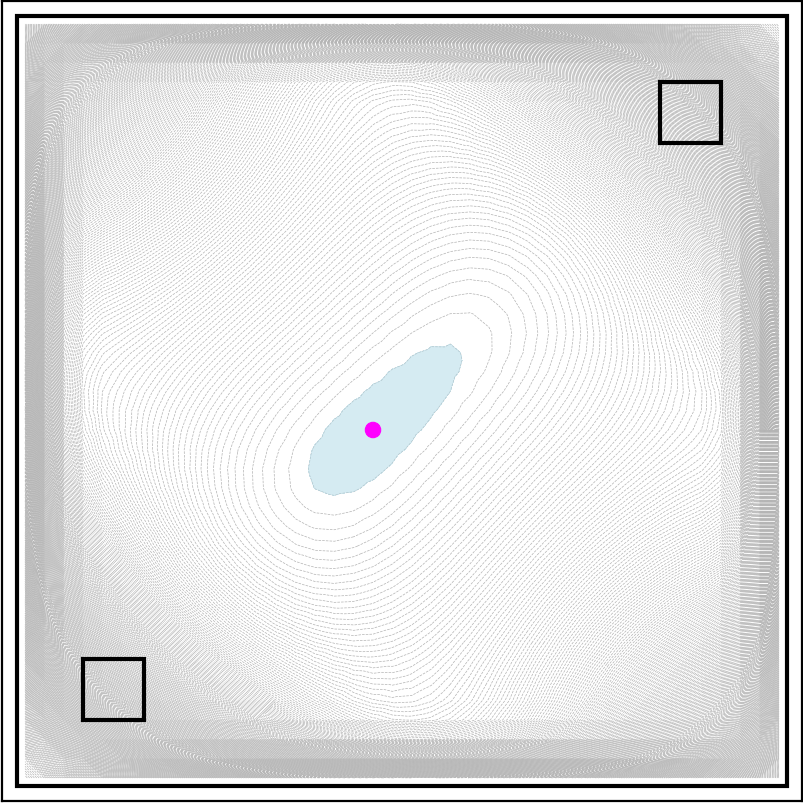}
        \caption{}
        \label{fig:twosq_img3}
    \end{subfigure}
    \hfill
    \begin{subfigure}{0.19\textwidth}
        \includegraphics[width=\textwidth]{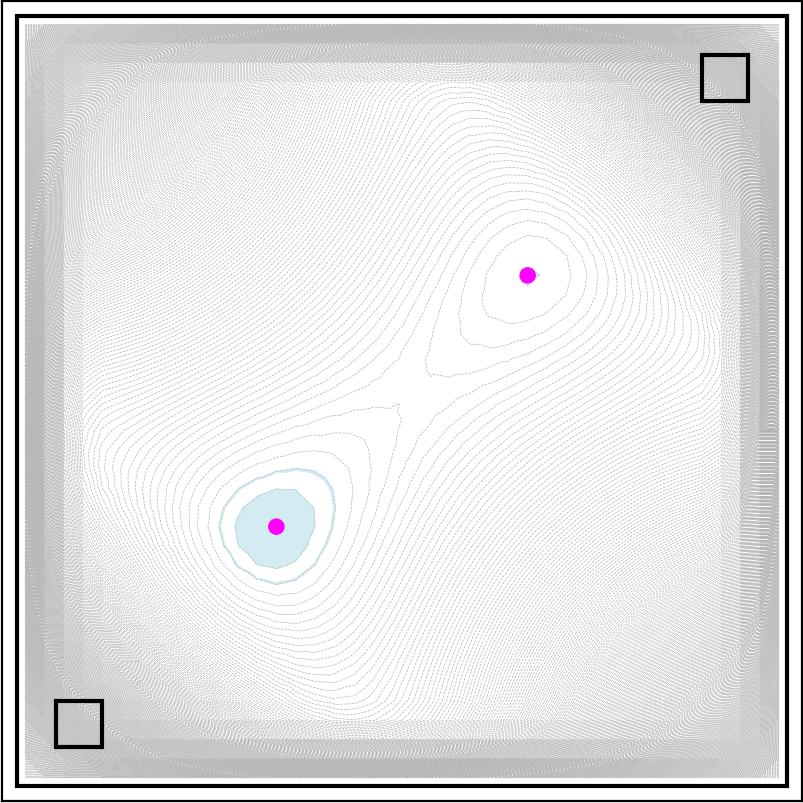}
        \caption{}
        \label{fig:twosq_img4}
    \end{subfigure}
    \hfill
    \begin{subfigure}{0.19\textwidth}
        \includegraphics[width=\textwidth]{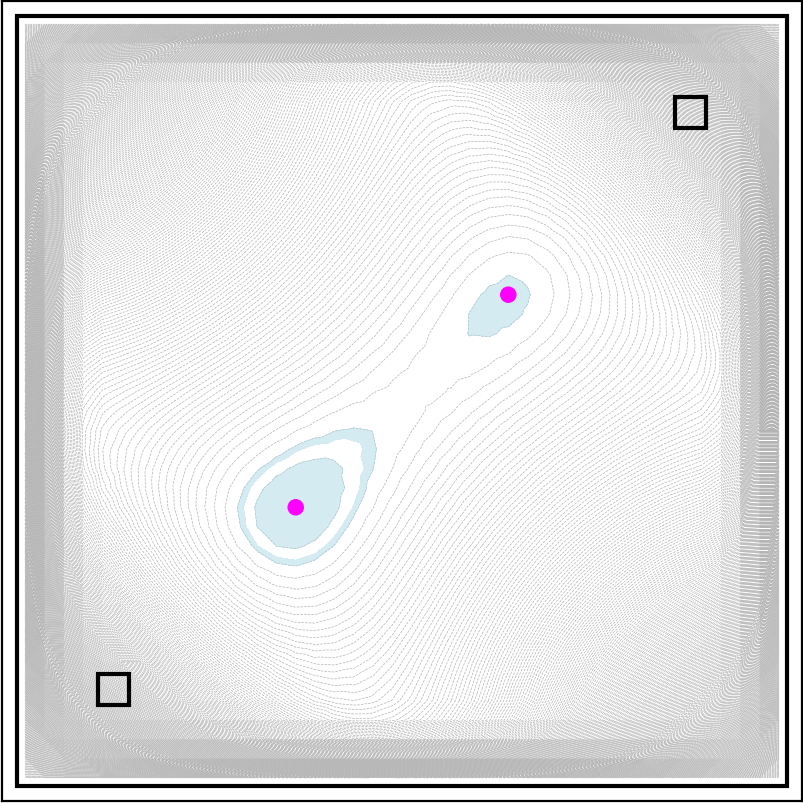}
        \caption{}
        \label{fig:twosq_img5}
    \end{subfigure}
    
    \caption{Results for two-square configurations with ($\delta, N_{\mathrm{mc}}, N_{\mathrm{scan}}) = (0.1, 100, 40)$.}
    \label{fig:twosq_shapes}
\end{figure}

\begin{figure}[htbp!]
\captionsetup{skip=0pt}
\captionsetup[subfigure]{skip=5pt}
    \centering
    
    \begin{subfigure}{0.19\textwidth}
            \includegraphics[width=\textwidth]{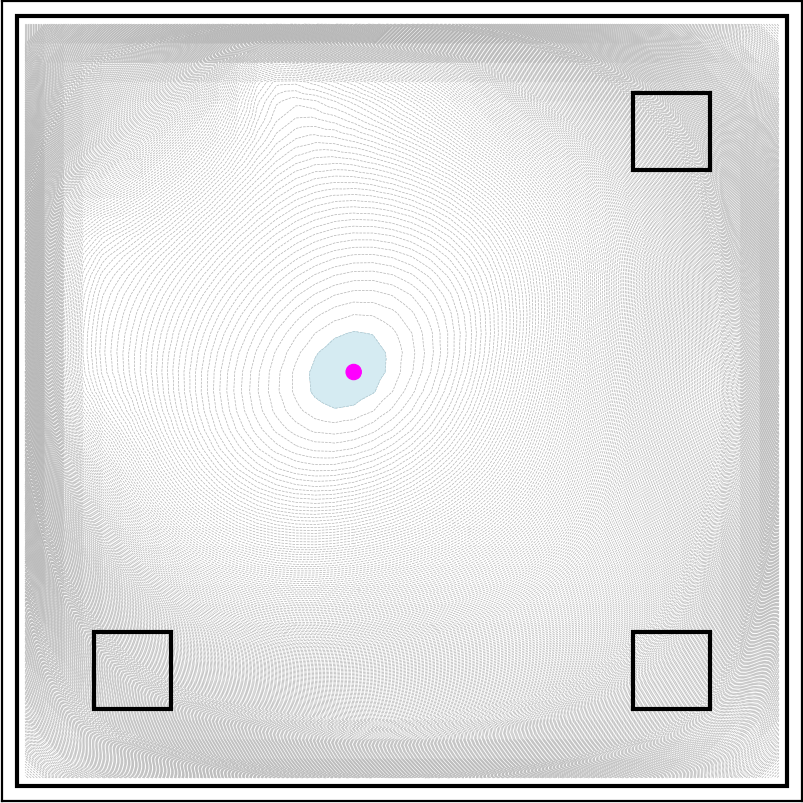}
        \caption{}
        \label{fig:threesq_img1}
    \end{subfigure}
    \hfill
    \begin{subfigure}{0.19\textwidth}
                \includegraphics[width=\textwidth]{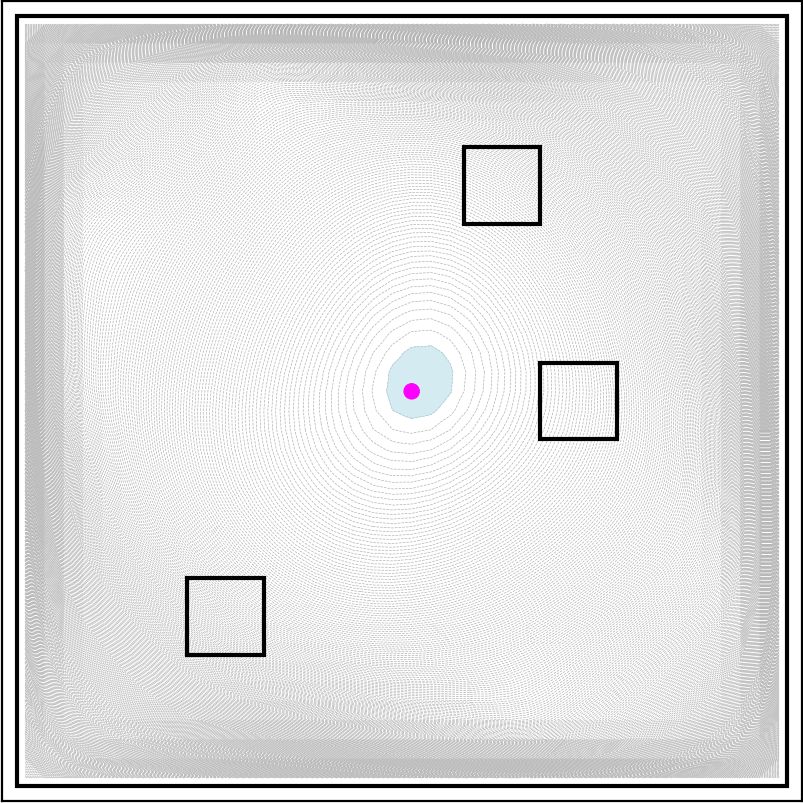}
        \caption{}
        \label{fig:threesq_img2}
    \end{subfigure}
    \hfill
    \begin{subfigure}{0.19\textwidth}
        \includegraphics[width=\textwidth]{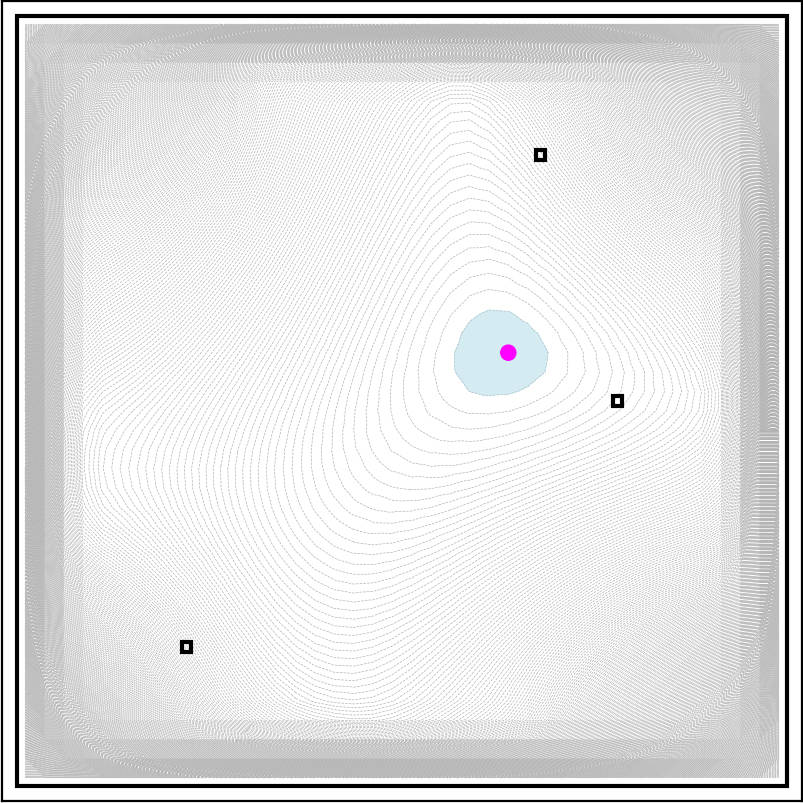}
                \caption{}
        \label{fig:threesq_img3}
    \end{subfigure}
    \hfill
    \begin{subfigure}{0.19\textwidth}
        \includegraphics[width=\textwidth]{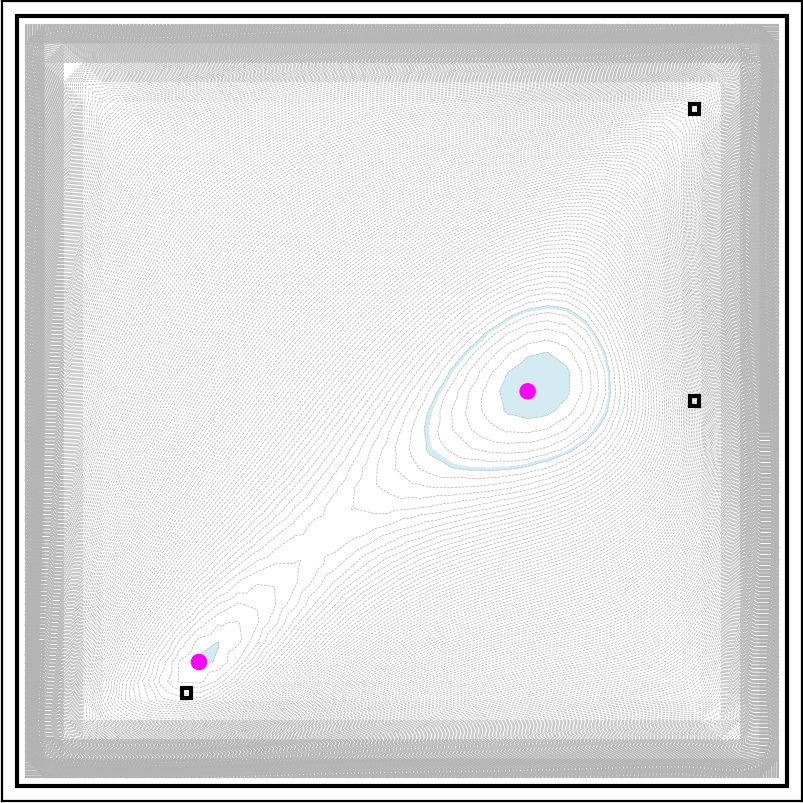}
        \caption{}
        \label{fig:threesq_img4}
    \end{subfigure}
    \hfill
    \begin{subfigure}{0.19\textwidth}
        \includegraphics[width=\textwidth]{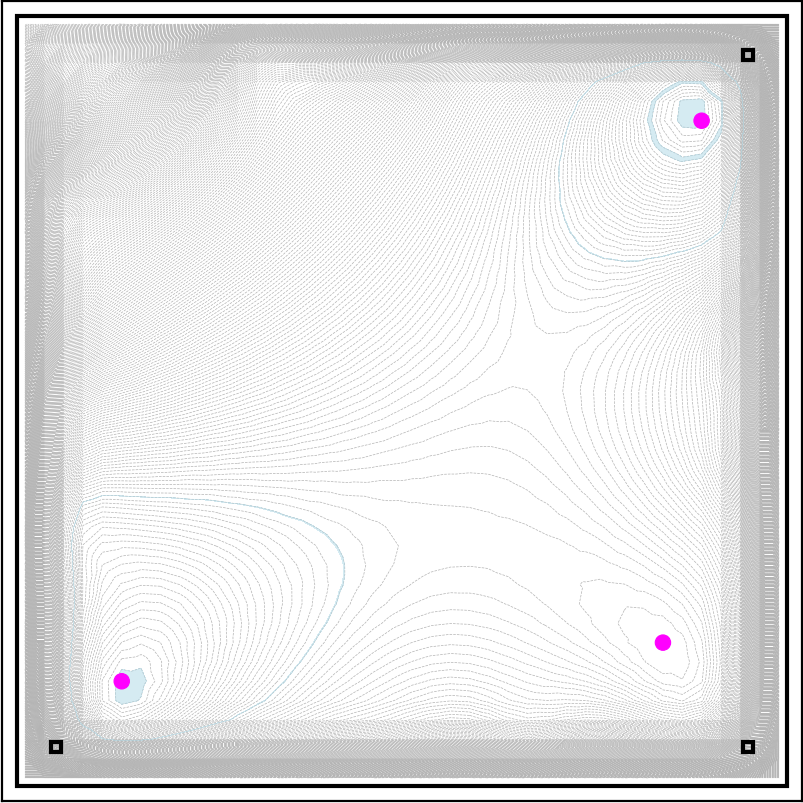}
        \caption{}
        \label{fig:threesq_img5}
    \end{subfigure}
    
    \caption{Results for three-square configurations with ($\delta, N_{\mathrm{mc}}, N_{\mathrm{scan}}) = (0.1, 100, 40)$.}
    \label{fig:threesq_shapes}
\end{figure}

\section{Reconstructions via shape optimization}
\label{sec:shape_optimization_numerics}
\subsection{Numerical algorithm for interface variation}
\label{subsec:Numerical_Algorithm_Shape}
The domain $\omega$ is updated through interface variation, following standard shape optimization procedures \cite{Doganetal2007}. 
To stabilize the boundary update and suppress oscillations, we compute a Riesz representation of the shape gradient $G$ by solving
\begin{equation}\label{eq:extension_regularization_simple}
(\nabla \VV, \nabla \vect{\varphi})_{\varOmega} + (\VV , \vect{\varphi} )_{\varOmega}
= - \inS{ G \nu , \vect{\varphi} }_{\partial\omega}, 
\quad \forall \vect{\varphi} \in H_{0}^1(\varOmega)^{d}.
\end{equation}
This produces a smoothed Sobolev gradient $\VV \in H_{0}^1(\varOmega)^{d}$, supported on $\partial \omega$ and extended into the entire domain.  
The extension allows the computational mesh to deform naturally, moving boundary nodes and internal nodes consistently.  

Let $\varOmega_h^{[k]}$ denote the current triangulated domain at the $k$-th iteration, with mesh size $h$. 
The updated domain $\varOmega_h^{[k+1]}$ is obtained by moving each mesh node along the deformation field $\VV^{[k]}$, namely
\[
\varOmega_h^{[k+1]} 
\coloneqq 
\left\{
\bx + t_h^{[k]} \VV^{[k]}(\bx)
\;\middle|\;
\bx \in \varOmega_h^{[k]}
\right\},
\]
where $t_h^{[k]} > 0$ is the step size at iteration $k$. 
It is computed using a backtracking line search, initialized at each pseudo-time step by \cite{RabagoAzegami2020}
\[
t_h^{[k]}
=
s \,
\frac{\mathcal{J}(\omega_h^{[k]})}
{\|\VV^{[k]}\|_{\HH^{d}}^{2}},
\]
with $s > 0$ a prescribed scaling factor.
The step size $t_h^{[k]}$ is further adjusted as needed to prevent mesh inversion or degeneration.  
The update is applied to all mesh nodes, both boundary and internal, ensuring a smooth deformation of the triangulated mesh.  

We note that the shape optimization reconstruction is initialized by using the topological gradient, which provides a first-step localization of the contact regions, together with the profile of the measured potentials on each boundary portion, which also indicates the likely number and approximate location of contact regions. 
This combined information improves convergence and reduces the risk of missing or merging small subregions. 
See Remark~\ref{rem:initialization_for_shape_optimization}.

In this framework, the shape optimization is effectively Lagrangian in nature, in the sense that the positions of the mesh nodes serve as the control variables for updating the domain.  

In our numerical experiments, no adaptive remeshing or explicit remeshing was performed, as straightforward mesh movement proved sufficient for the problems considered.  
We note, however, that for other configurations or more challenging geometries, performing remeshing every few iterations could potentially improve accuracy, although this was not necessary in our cases.

The iteration stops when $t_h^{[k]} < t_0$, where $t_0 > 0$ is a prescribed tolerance chosen to prevent excessive deformation of the mesh and to ensure numerical stability, or when a maximum number of iterations is reached.  
Although other stopping criteria may be more appropriate in specific applications, this simple choice already provides reasonable reconstructions, and in practice the iterations converge to a stationary shape.  
This procedure yields a stable, mesh-consistent update of the domain for first-order shape optimization.
\begin{remark}
\label{rem:initialization_for_shape_optimization}
In the present reconstruction procedure, a Lagrangian framework is adopted in which the mesh nodes are updated according to the shape gradient. As a result, the topology of the domain is preserved, and no topological changes (such as the creation or merging of inclusions) can occur.
To obtain an indication of the number of inclusions, this approach may be complemented by an auxiliary analysis. In addition to the topological gradient, the one-dimensional boundary profiles of the measured potential are examined. An inclusion may induce a localized extremum (typically a minimum) near the boundary point aligned with its interior location.
For each detected minimum, the inward normal to the boundary is considered, and the intersections of these normals inside the domain are analyzed. Clusters of intersections may indicate candidate inclusion locations: a single cluster may suggest one inclusion, while multiple distinct clusters may suggest several inclusions.
\end{remark}

\subsection{Numerical examples using shape gradient information} 

We now illustrate the performance of the proposed shape gradient method across configurations involving single and multiple contact regions under exact and noisy measurements.

For a single contact region (Figure~\ref{fig:all_examples}\subref{fig:big_img1}--\subref{fig:big_img4}), the interface is accurately reconstructed under exact measurements ($\delta=0$). 
Under noisy measurements ($\delta=0.1$), the reconstruction remains close to the true interface, exhibiting only mild boundary perturbations. 
The nonconstant boundary data $g=\abs{x}$ yields slightly improved localization compared with $g=1$.

For two symmetric contact regions of equal size (Figure~\ref{fig:all_examples}\subref{fig:two_big_exact}--\subref{fig:two_small_noisy}), both components are correctly identified and separated. 
The reconstructed interfaces closely match the exact shapes when $\delta=0$, and remain well localized under noisy measurements, with minor boundary smoothing.

For two contact regions of unequal sizes (Figure~\ref{fig:all_examples}\subref{fig:two_unequal_squares_exact_a}--\subref{fig:two_small_noisy_b}), both components are well reconstructed in the absence of noise. 
Under noisy measurements ($\delta=0.1$), the larger region is recovered accurately, whereas the smaller one exhibits greater deviations, reflecting the higher sensitivity of small geometric features to noise.

For three contact regions (Figure~\ref{fig:all_examples}\subref{fig:three_equal_img1}--\subref{fig:three_equal_img4}), the method converges to three distinct components from different initial guesses.

Overall, across all configurations, the method preserves the correct topology and separation of the contact regions, maintains smooth interfaces, and achieves stable convergence without remeshing, demonstrating robustness under moderate noise.

\subsection{Effect of the free parameter $\beta$}\label{subsec:effect_of_beta}
The numerical results in Figures~\ref{fig:beta1_shapes} and~\ref{fig:beta200_all} illustrate the interplay between the free parameter $\beta$, the initial configuration, geometric complexity, and measurement noise. 
Whereas the influence of $\beta$ in the topology optimization stage is negligible, it plays a significant role in the subsequent shape optimization step, where it affects the resulting geometric accuracy.

Under exact measurements with $\beta = 1$ (Figure~\ref{fig:beta1_shapes}), the reconstructions are unsatisfactory.
Although the approximate location of the contact regions is sometimes detected, significant discrepancies in size and boundary shape persist.
In particular, concave features are poorly resolved, indicating that a small value of $\beta$ is insufficient for accurate geometric reconstruction, even in the absence of noise ($\delta = 0$).

For noisy measurements with $\delta = 0.1$ and $\beta = 200$ (Figure~\ref{fig:beta200_all}), the overall reconstruction quality improves across the different configurations. 
In particular, the recovered shapes more closely approximate the true geometry compared to the case $\beta = 1$ (see, e.g., Figures~\ref{fig:beta200_img1} and~\ref{fig:beta200_img2} versus Figure~\ref{fig:beta1_img2}).
In several examples, concave portions of the interface are more clearly resolved, which is associated with improved reconstruction quality under noisy measurements.

The influence of the initial configuration is evident in Figures~\ref{fig:beta1_shapes} and~\ref{fig:beta200_all}.
When the initial guess is placed close to, or entirely inside, the contact region (e.g., Figures~\ref{fig:beta1_img1}, \ref{fig:beta200_img1}--\ref{fig:beta200_img3}, and~\ref{fig:beta200_img14}), the recovered shape more accurately matches its location and structure.
In contrast, when the initial shape is more displaced (see, e.g., Figures~\ref{fig:beta1_img1}--\ref{fig:beta200_img13} and~\ref{fig:beta200_img15}), larger reconstruction errors are observed.
This effect is particularly pronounced in configurations with multiple components, where the problem is inherently more complex.

Concave regions are more difficult to recover, which is expected.
In Figure~\ref{fig:beta1_shapes}, inward corners are often distorted or only partially resolved. 
In several configurations in Figure~\ref{fig:beta200_all} (e.g., Figures~\ref{fig:beta200_img1}--\ref{fig:beta200_img3} and ~\ref{fig:beta200_img13}--\ref{fig:beta200_img14}), re-entrant parts of the boundary are more clearly represented. 
However, very sharp or small-scale indentations remain difficult to capture precisely in both settings.

The reconstruction quality also depends on the size of the contact region. 
Larger regions, such as those shown in Figures~\ref{fig:beta200_img1}--\ref{fig:beta200_img3} and~\ref{fig:beta200_img6}--\ref{fig:beta200_img7}, exhibit closer agreement with the true geometry, whereas smaller inclusions (e.g., Figures~\ref{fig:beta200_img3},~\ref{fig:beta200_img4},~\ref{fig:beta200_img10}--\ref{fig:beta200_img12}) display comparatively larger shape deviations. 
Overall, the discrepancy becomes more pronounced at smaller characteristic scales.

Taken together, Figures~\ref{fig:beta1_shapes} and~\ref{fig:beta200_all} indicate a pronounced dependence of the reconstruction on the parameter $\beta$, with increased sensitivity observed under noisy measurements.

\begin{figure}[htbp!]
\centering

\begin{subfigure}{0.22\textwidth}
	\includegraphics[width=\textwidth]{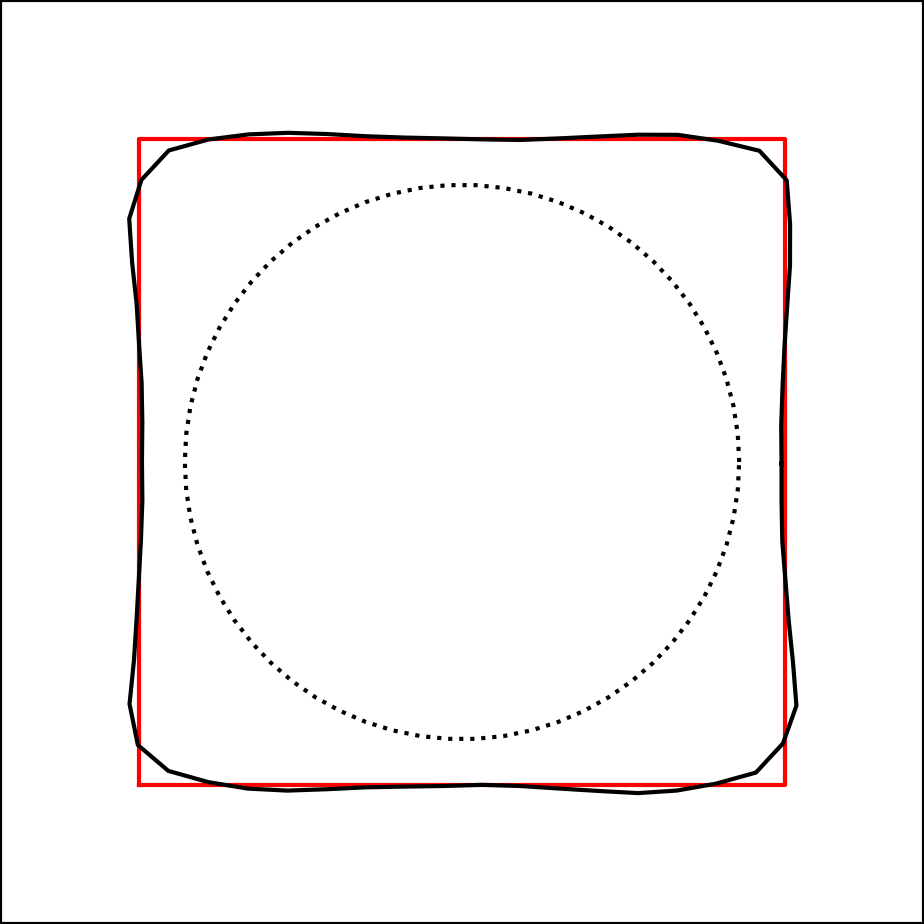}
	\caption{$\delta = 0$, \ $g=1$}
	\label{fig:big_img1}
\end{subfigure}
\hfill
\begin{subfigure}{0.22\textwidth}
	\includegraphics[width=\textwidth]{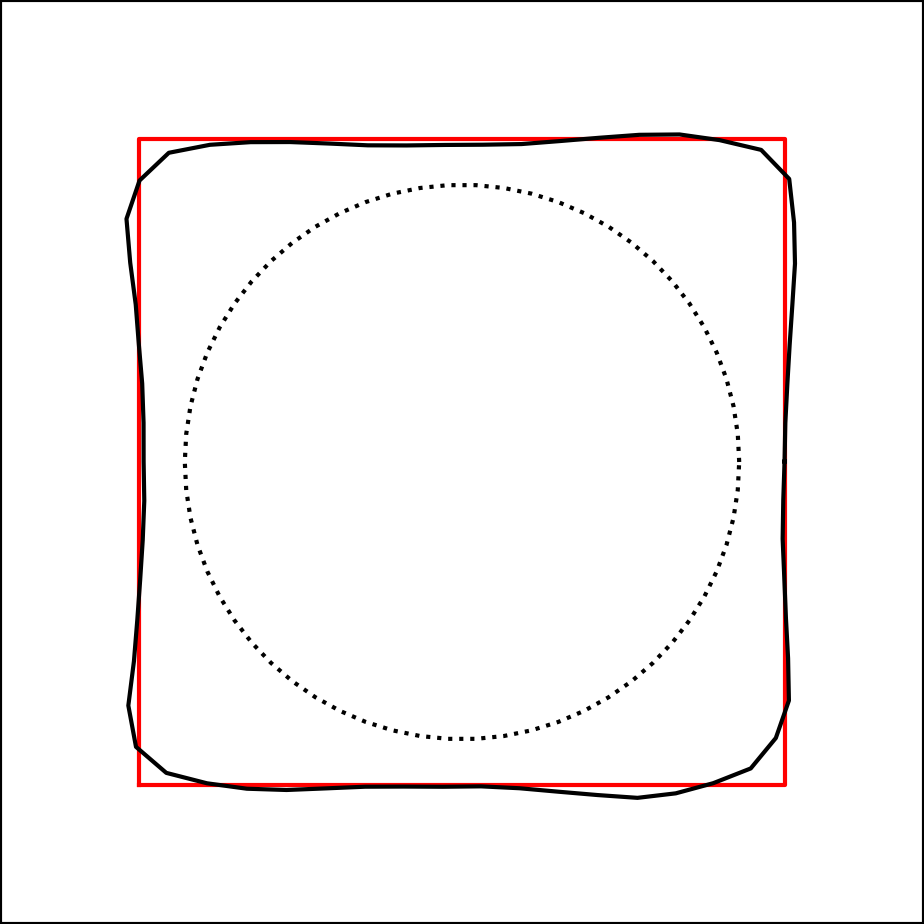}
	\caption{$\delta = 0.1$, \ $g=1$}
	\label{fig:big_img2}
\end{subfigure}
\hfill
\begin{subfigure}{0.22\textwidth}
	\includegraphics[width=\textwidth]{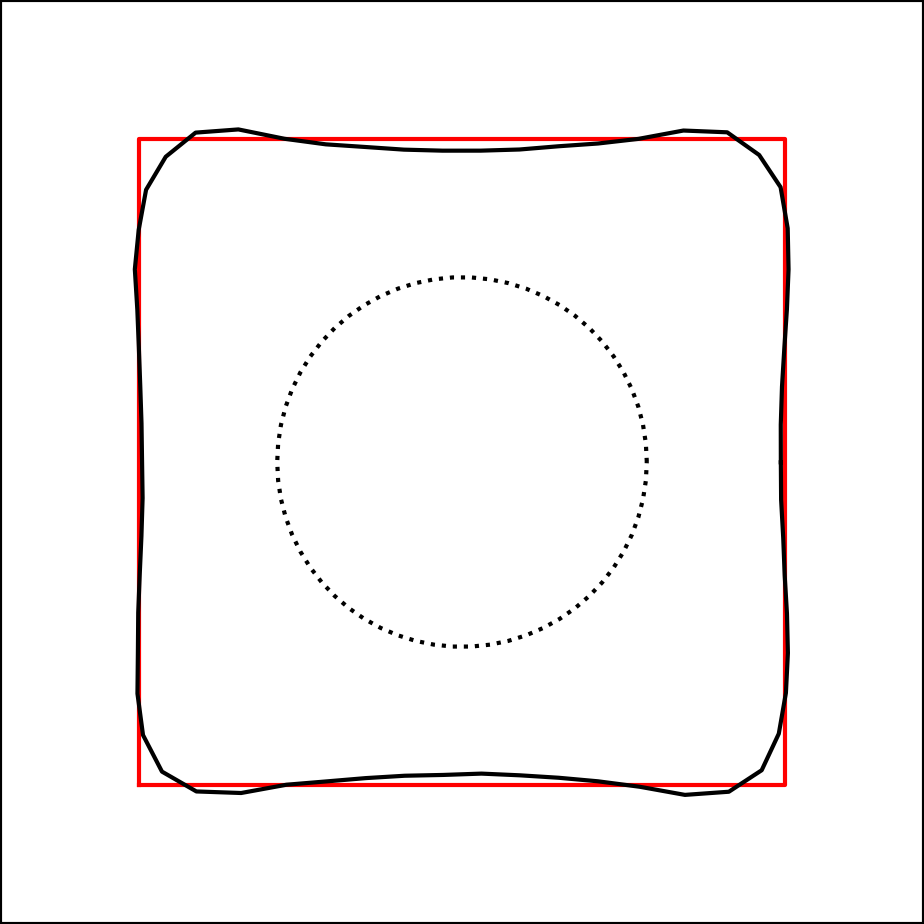}
	\caption{$\delta = 0$, \ $g=\abs{x}$}
	\label{fig:big_img3}
\end{subfigure}
\hfill
\begin{subfigure}{0.22\textwidth}
	\includegraphics[width=\textwidth]{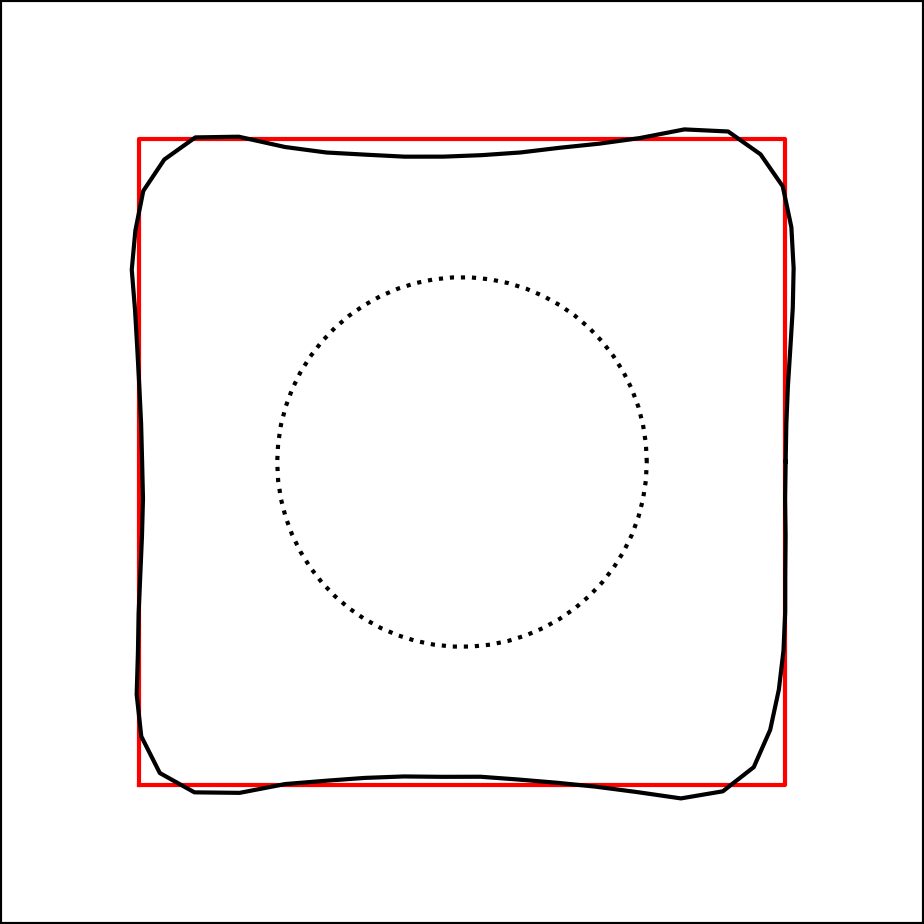}
	\caption{$\delta = 0.1$, \ $g=\abs{x}$}
	\label{fig:big_img4}
\end{subfigure}

\vspace{1em}

\begin{subfigure}{0.22\textwidth}
	\includegraphics[width=\textwidth]{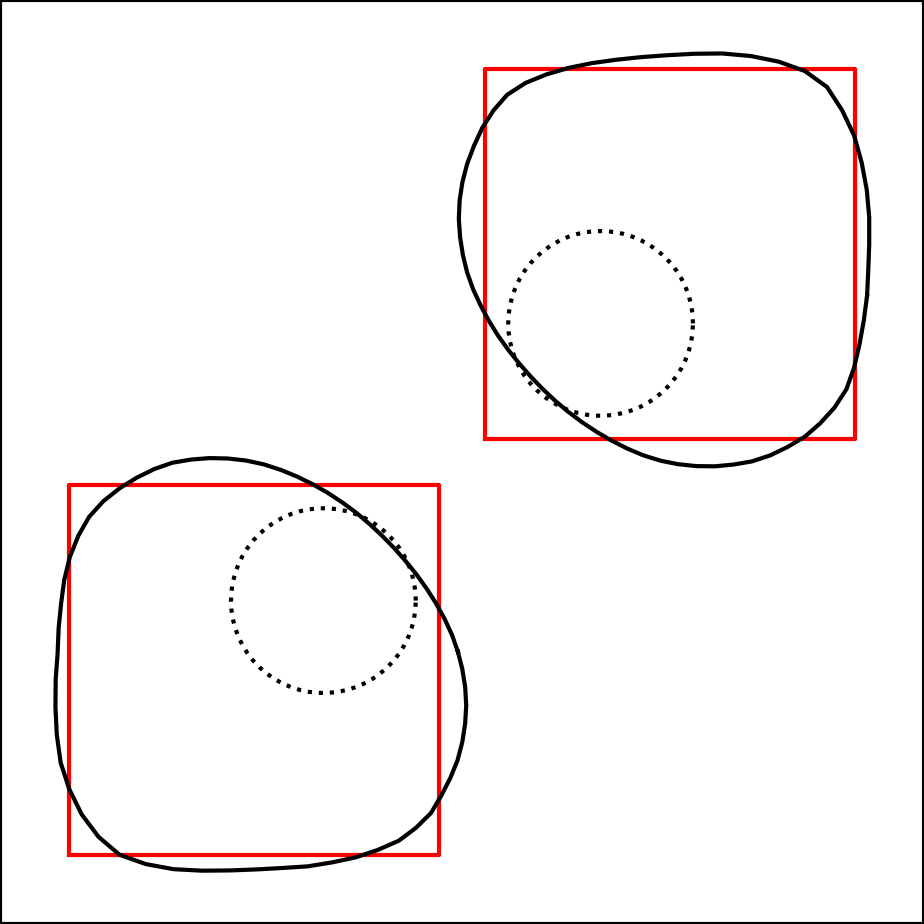}
	\caption{$\delta = 0$}
	\label{fig:two_big_exact}
\end{subfigure}
\hfill
\begin{subfigure}{0.22\textwidth}
	\includegraphics[width=\textwidth]{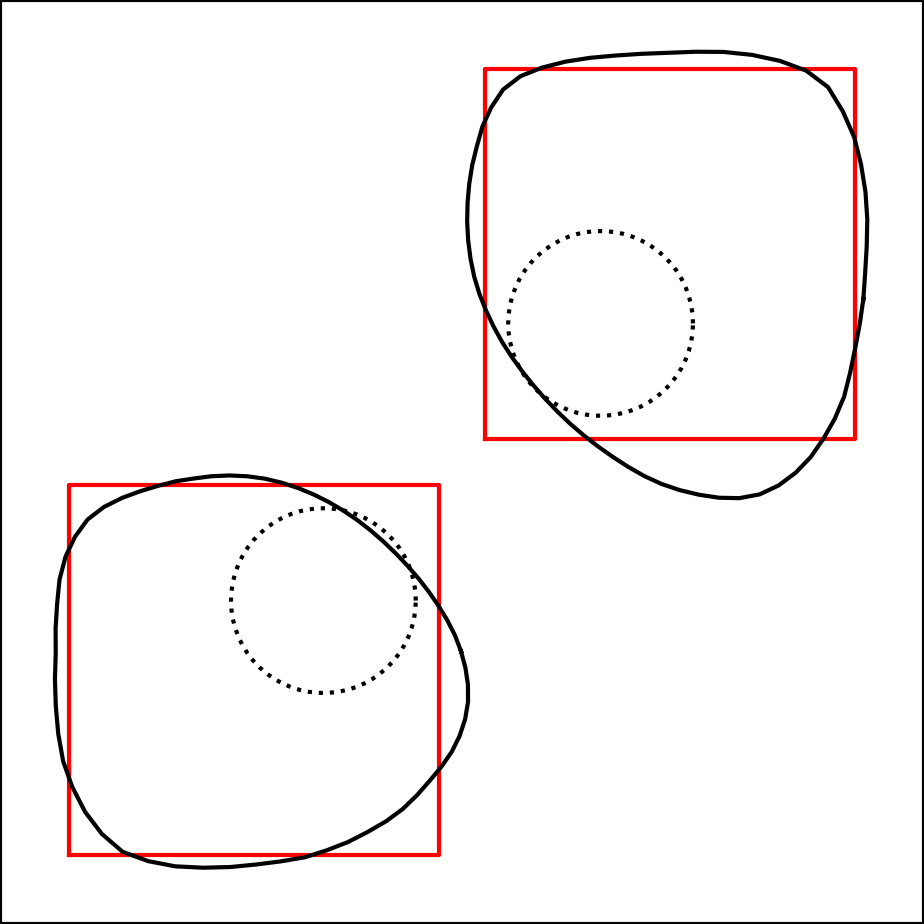}
	\caption{$\delta = 0.1$}
	\label{fig:two_big_noisy}
\end{subfigure}
\hfill
\begin{subfigure}{0.22\textwidth}
	\includegraphics[width=\textwidth]{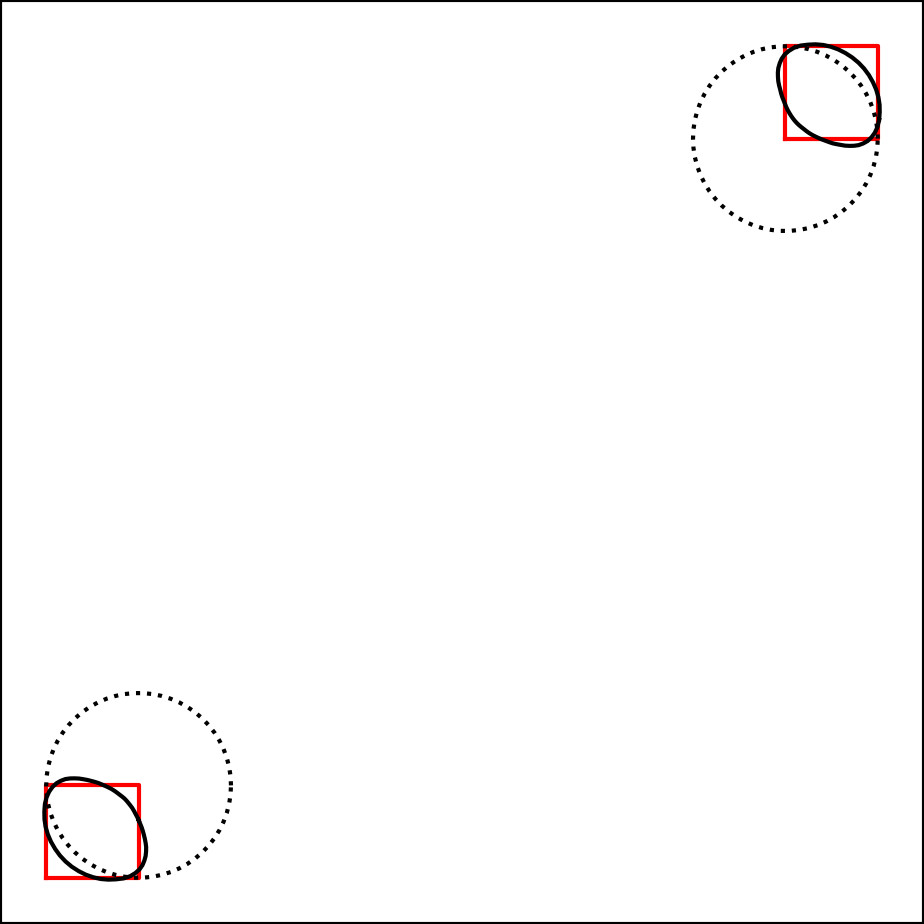}
	\caption{$\delta = 0$}
	\label{fig:two_small_exact}
\end{subfigure}
\hfill
\begin{subfigure}{0.22\textwidth}
	\includegraphics[width=\textwidth]{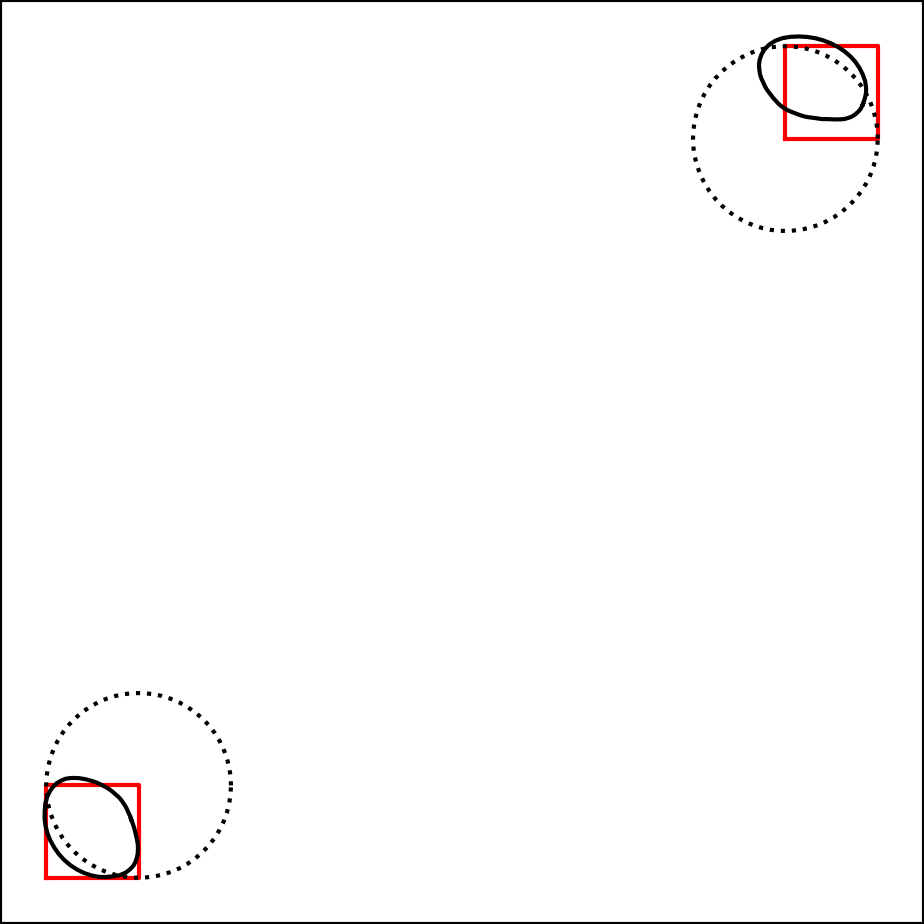}
	\caption{$\delta = 0.1$}
	\label{fig:two_small_noisy}
\end{subfigure}

\vspace{1em}

\begin{subfigure}{0.22\textwidth}
	\includegraphics[width=\textwidth]{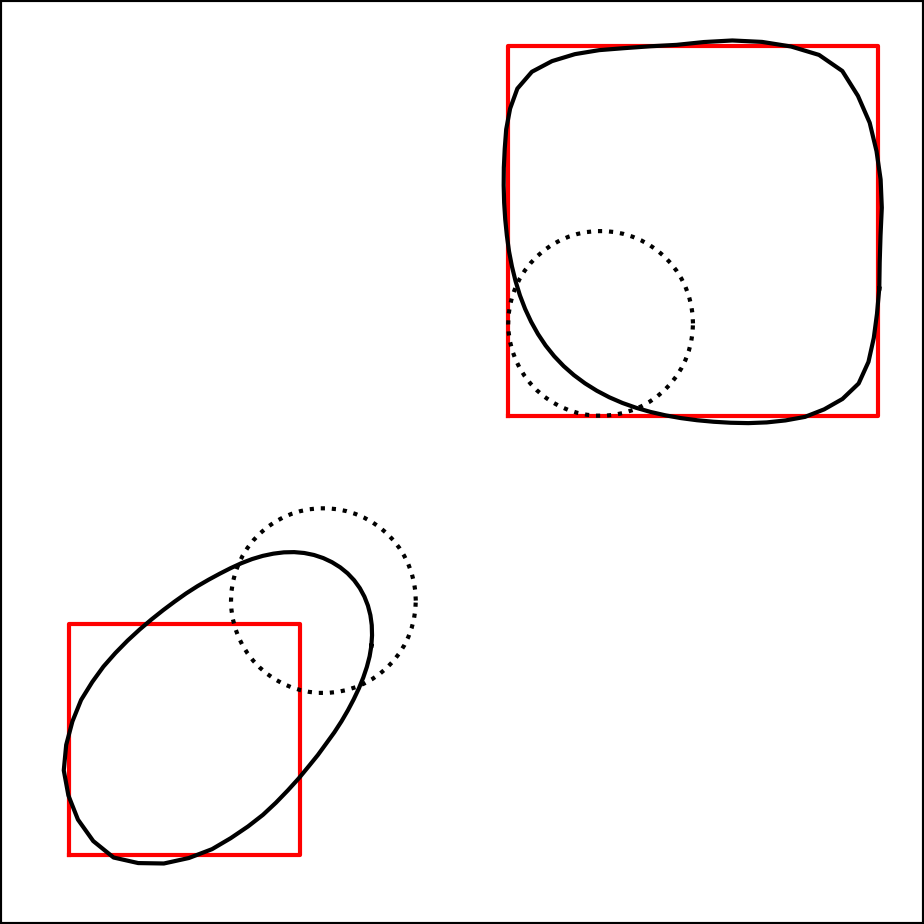}
	\caption{$\delta = 0$}
	\label{fig:two_unequal_squares_exact_a}
\end{subfigure}
\hfill
\begin{subfigure}{0.22\textwidth}
	\includegraphics[width=\textwidth]{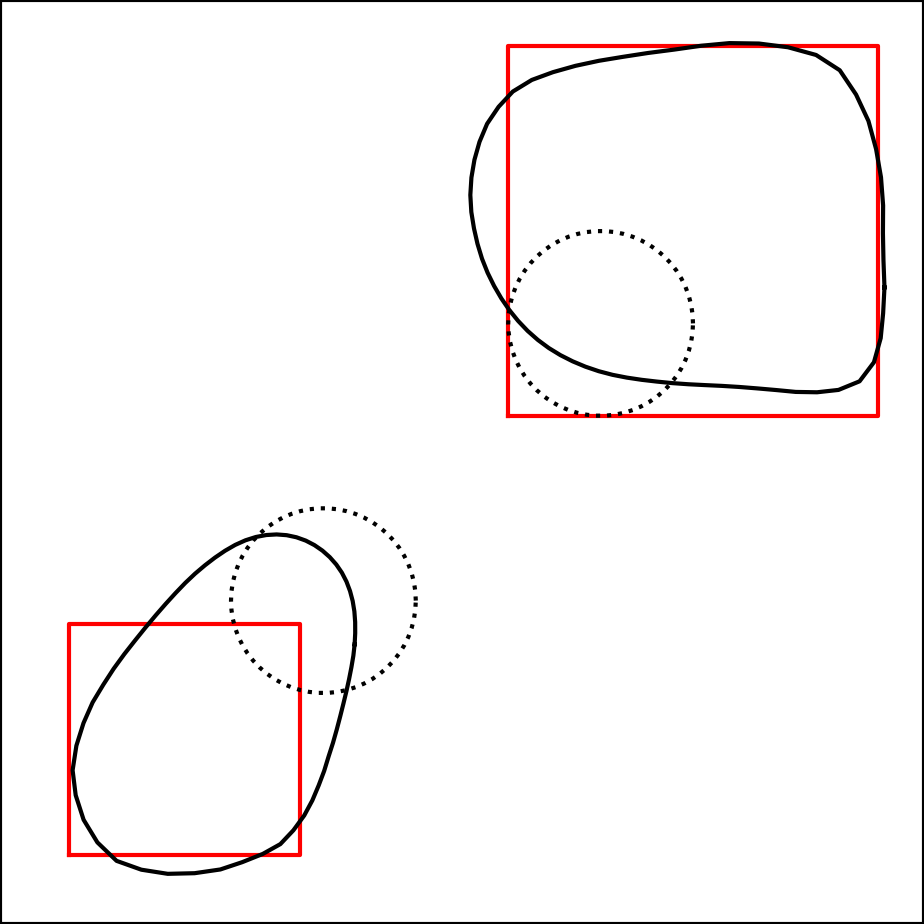}
	\caption{$\delta = 0.1$}
	\label{fig:two_unequal_squares_noisy_a}
\end{subfigure}
\hfill
\begin{subfigure}{0.22\textwidth}
	\includegraphics[width=\textwidth]{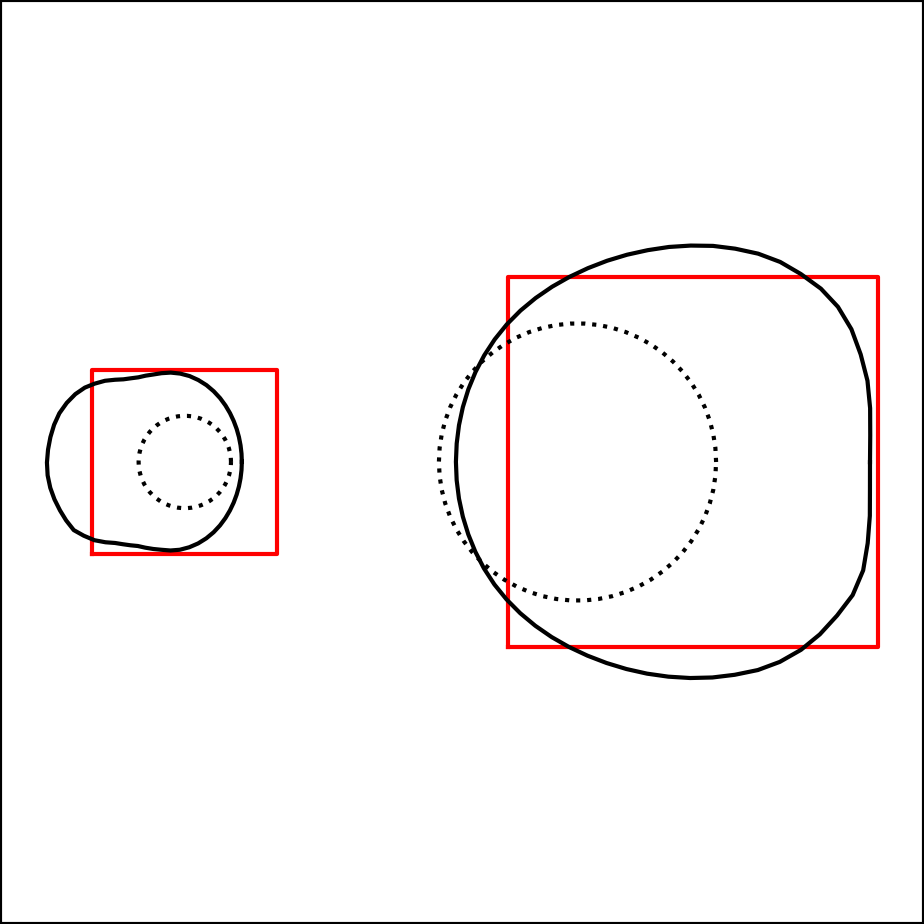}
	\caption{$\delta = 0$}
	\label{fig:two_small_exact_b}
\end{subfigure}
\hfill
\begin{subfigure}{0.22\textwidth}
	\includegraphics[width=\textwidth]{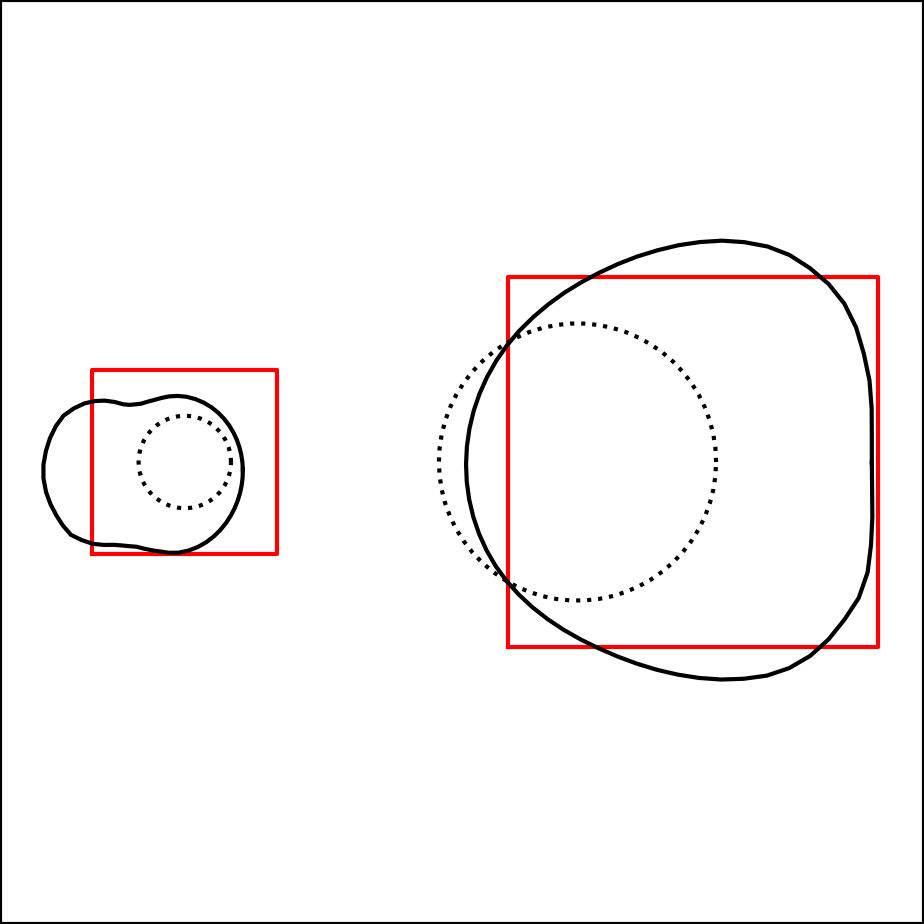}
	\caption{$\delta = 0.1$}
	\label{fig:two_small_noisy_b}
\end{subfigure}

\vspace{1em}

\begin{subfigure}{0.22\textwidth}
	\includegraphics[width=\textwidth]{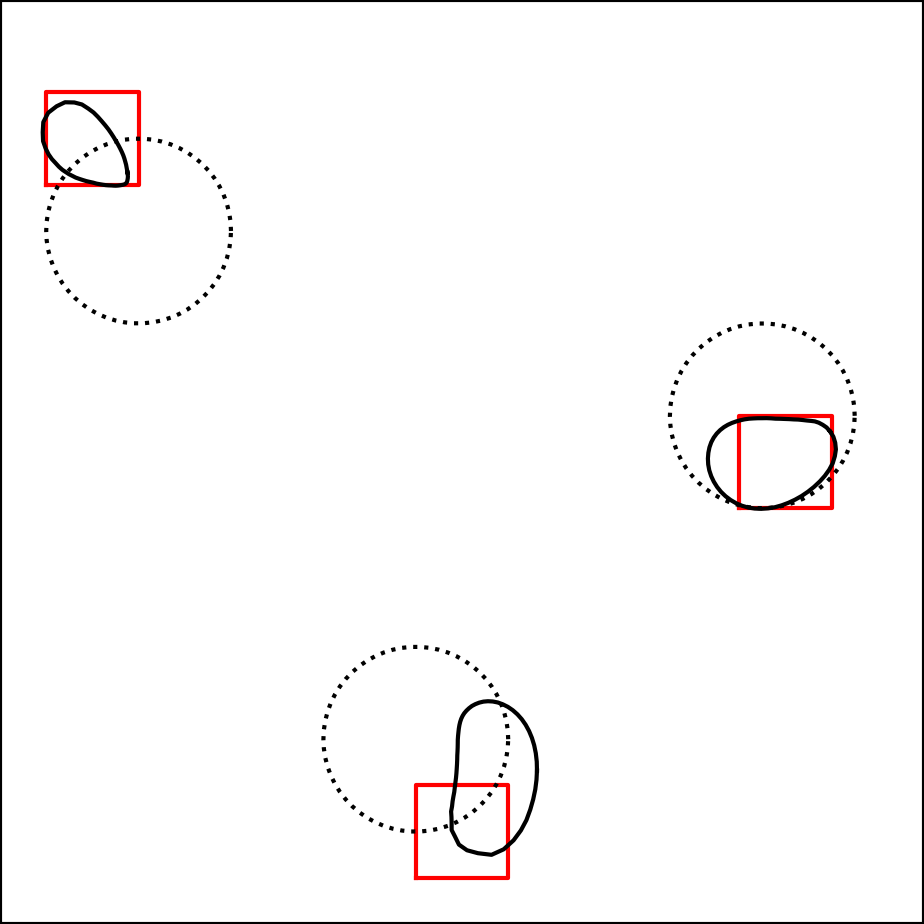}
	\caption{$\delta = 0.1$}
	\label{fig:three_equal_img1}
\end{subfigure}
\hfill
\begin{subfigure}{0.22\textwidth}
	\includegraphics[width=\textwidth]{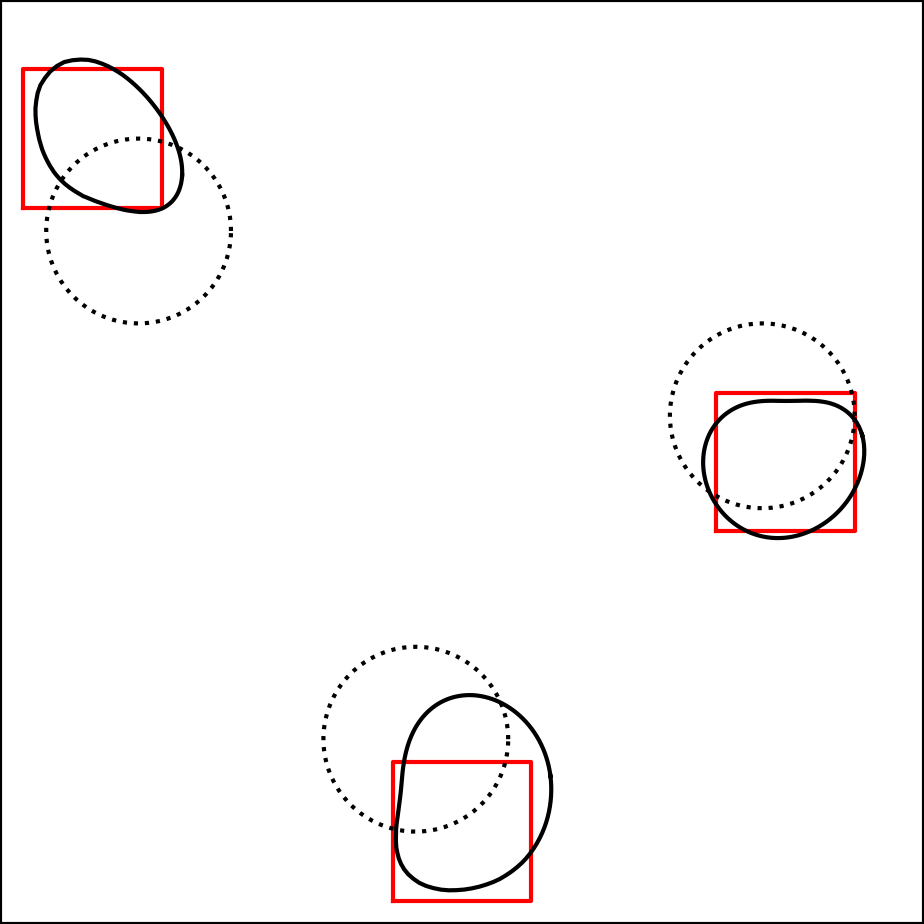}
	\caption{$\delta = 0.1$}
	\label{fig:three_equal_img2}
\end{subfigure}
\hfill
\begin{subfigure}{0.22\textwidth}
	\includegraphics[width=\textwidth]{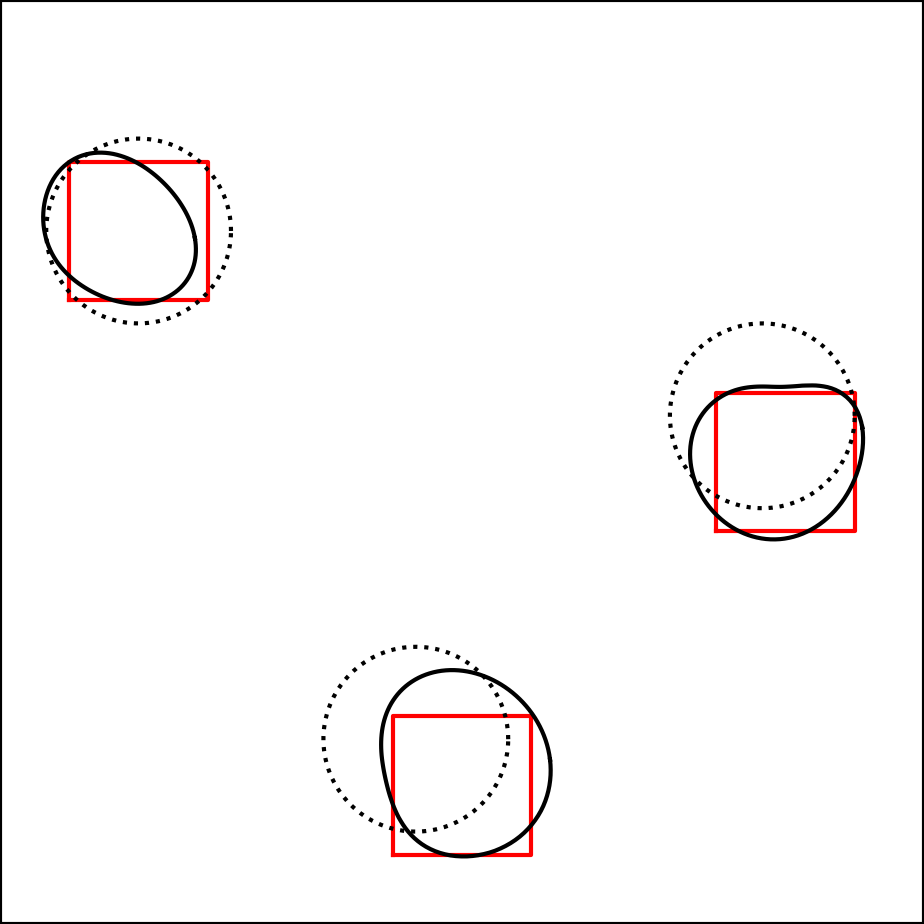}
	\caption{$\delta = 0.1$}
	\label{fig:three_equal_img3}
\end{subfigure}
\hfill
\begin{subfigure}{0.22\textwidth}
	\includegraphics[width=\textwidth]{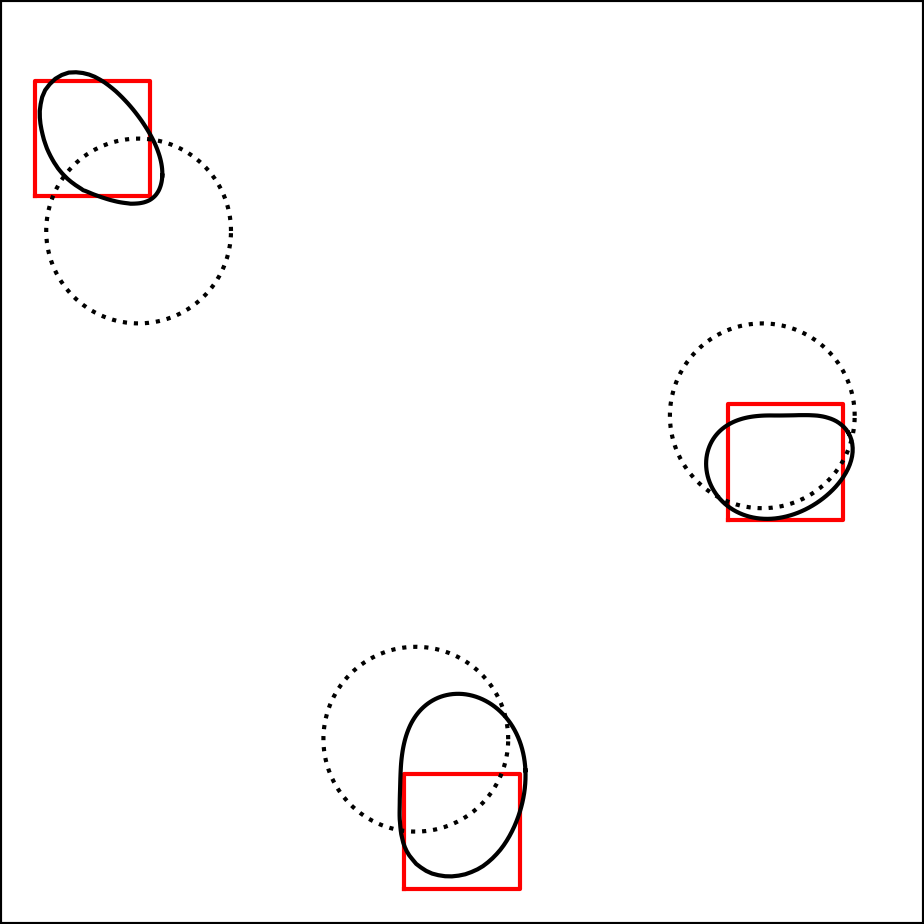}
	\caption{$\delta = 0.1$}
	\label{fig:three_equal_img4}
\end{subfigure}

\caption{Reconstruction results for single and multiple contact regions under different noise levels and boundary data.}
\label{fig:all_examples}

\end{figure}

\begin{figure}[htbp!]
\captionsetup{skip=0pt}                 
\captionsetup[subfigure]{skip=5pt}      
    \centering
    
    \begin{subfigure}{0.19\textwidth}
        \includegraphics[width=\textwidth]{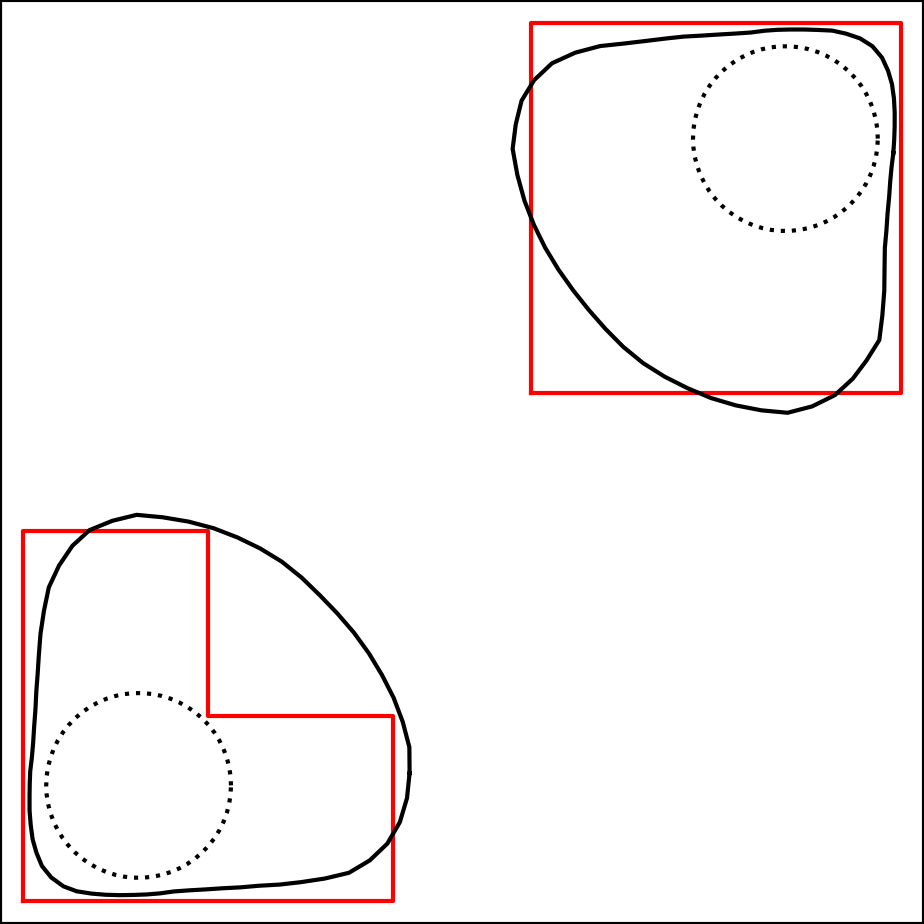}
        \caption{}
        \label{fig:beta1_img1}
    \end{subfigure}
    \hfill
    \begin{subfigure}{0.19\textwidth}
        \includegraphics[width=\textwidth]{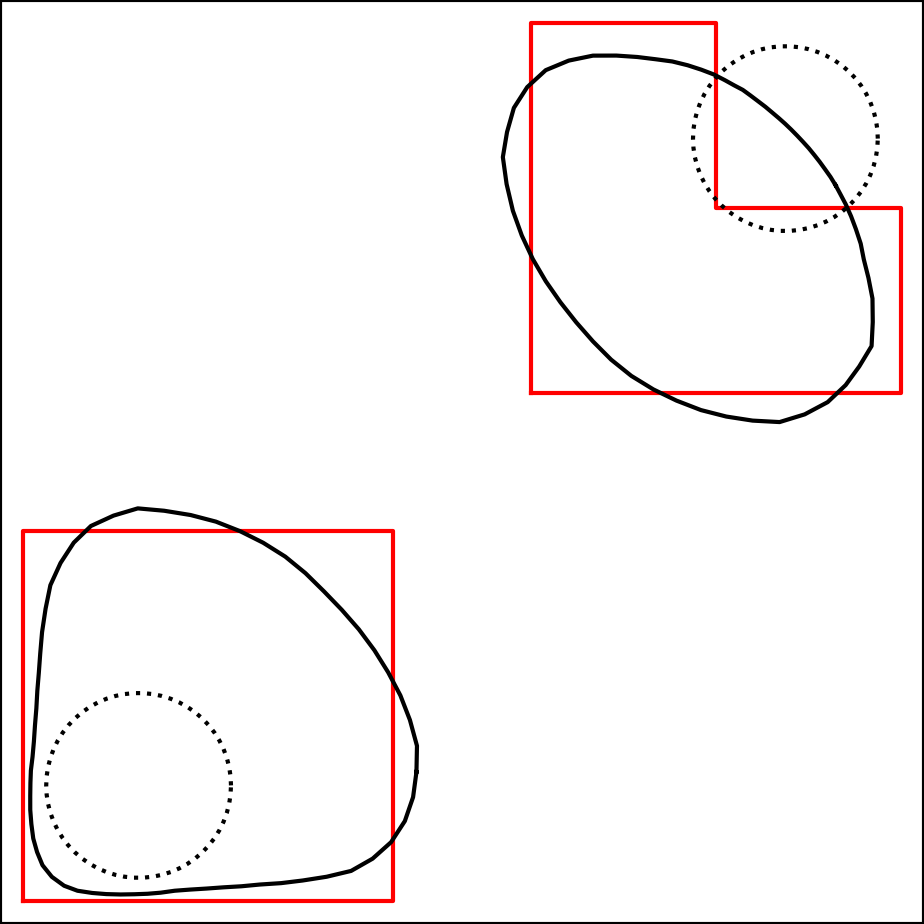}
        \caption{}
        \label{fig:beta1_img2}
    \end{subfigure}
    \hfill
    \begin{subfigure}{0.19\textwidth}
        \includegraphics[width=\textwidth]{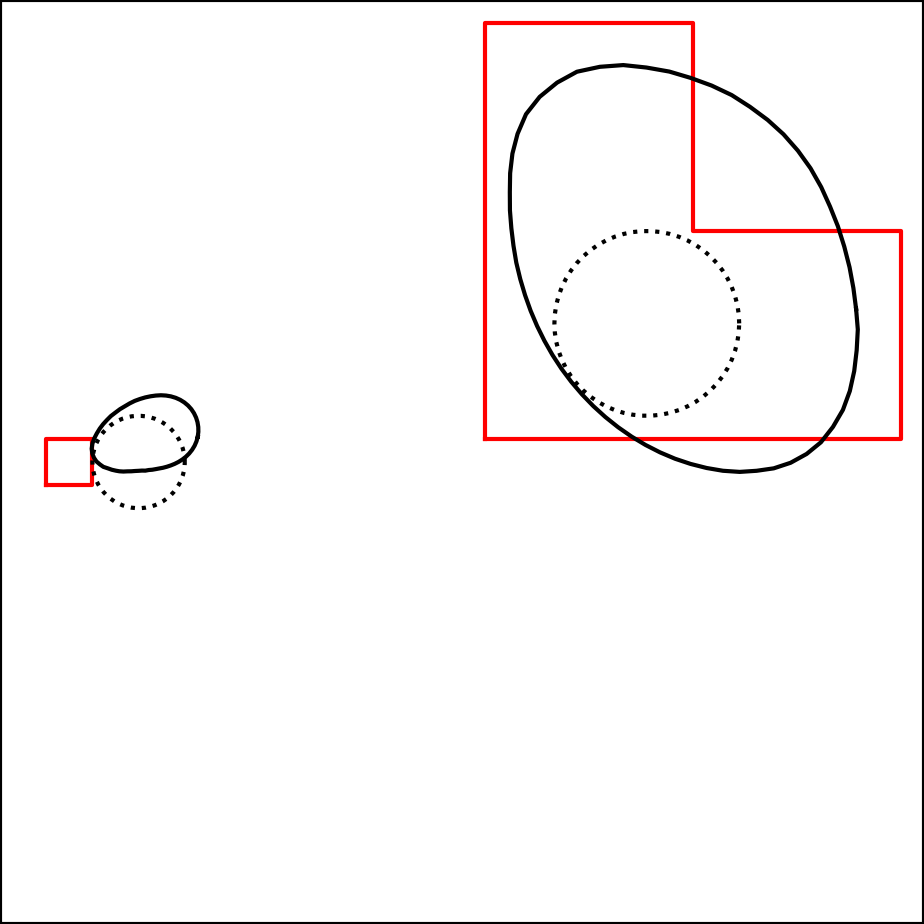}
        \caption{}
        \label{fig:beta1_img3}
    \end{subfigure}
    \hfill
    \begin{subfigure}{0.19\textwidth}
        \includegraphics[width=\textwidth]{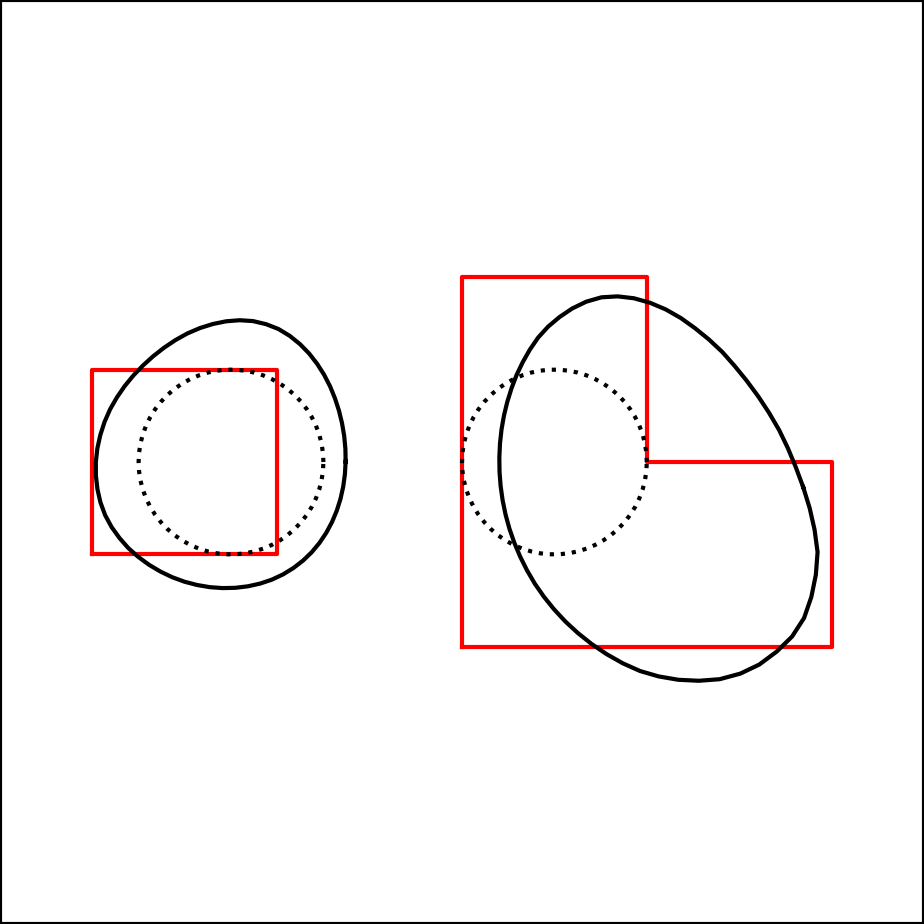}
        \caption{}
        \label{fig:beta1_img4}
    \end{subfigure}
    \hfill
    \begin{subfigure}{0.19\textwidth}
        \includegraphics[width=\textwidth]{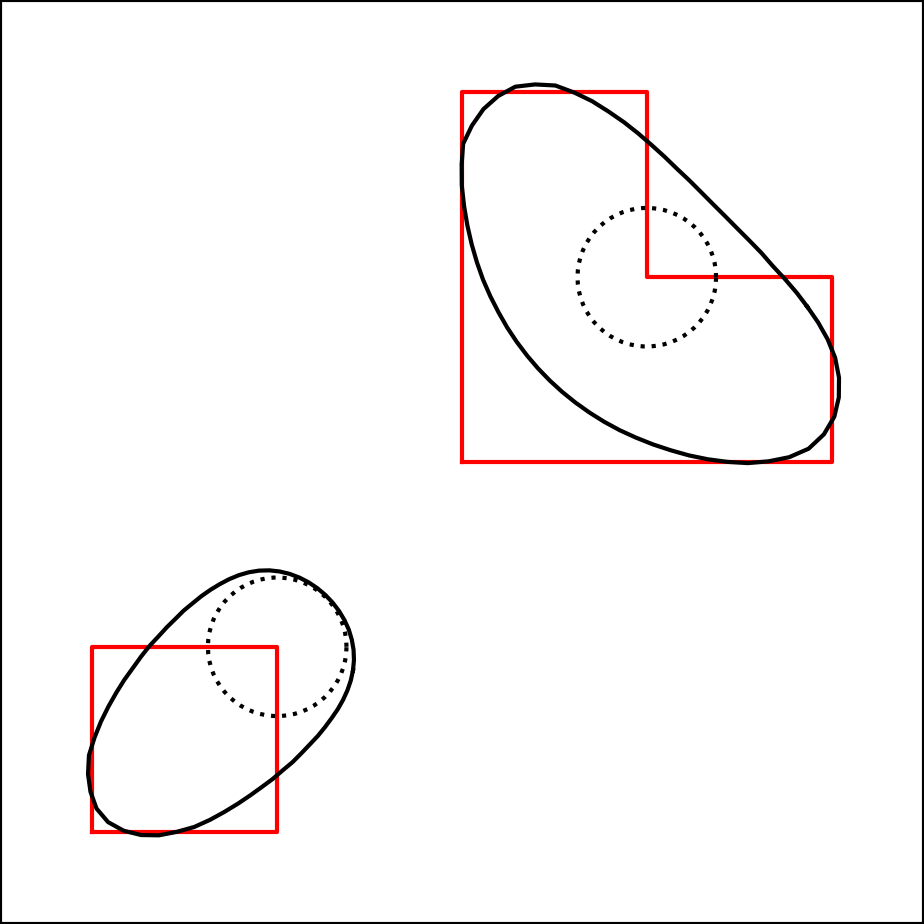}
        \caption{}
        \label{fig:beta1_img5}
    \end{subfigure}
    
    \caption{Reconstruction results under exact measurements and with $\beta = 1$.}
    \label{fig:beta1_shapes}
\end{figure}

\begin{figure}[htbp!]
\captionsetup{skip=0pt}                 
\captionsetup[subfigure]{skip=5pt}      
\centering

\begin{subfigure}{0.18\textwidth}
\includegraphics[width=\textwidth]{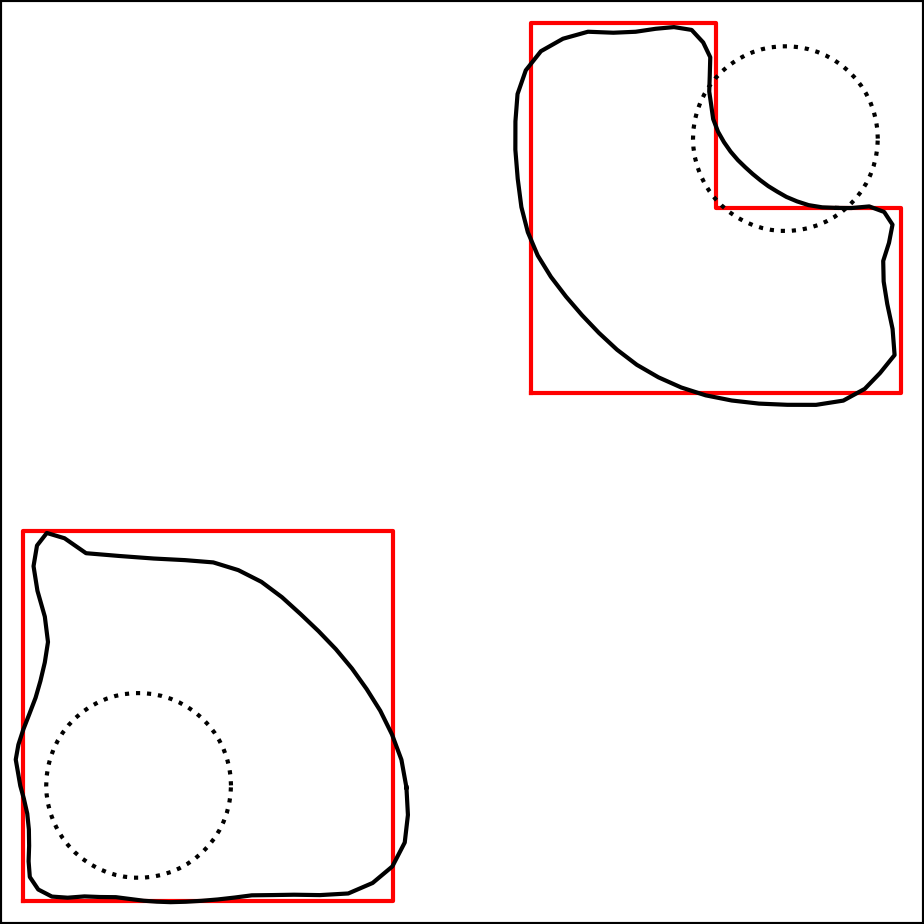}
\caption{}
\label{fig:beta200_img1}
\end{subfigure}
\hfill
\begin{subfigure}{0.18\textwidth}
\includegraphics[width=\textwidth]{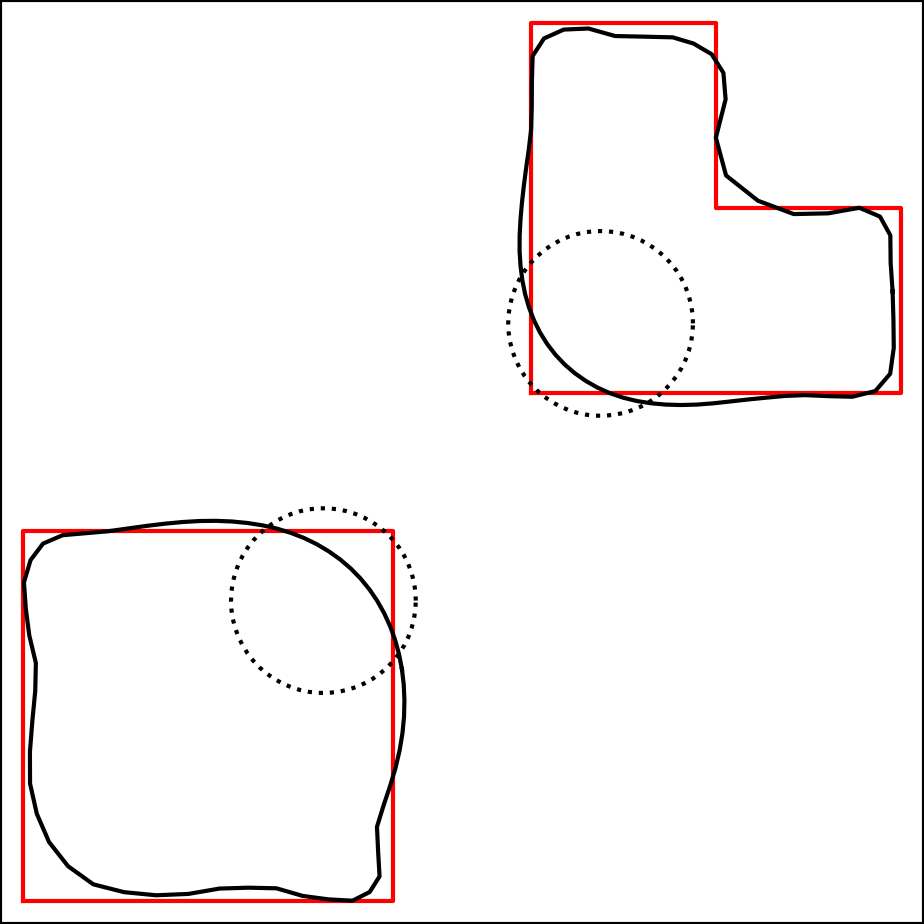}
\caption{}
\label{fig:beta200_img2}
\end{subfigure}
\hfill
\begin{subfigure}{0.18\textwidth}
\includegraphics[width=\textwidth]{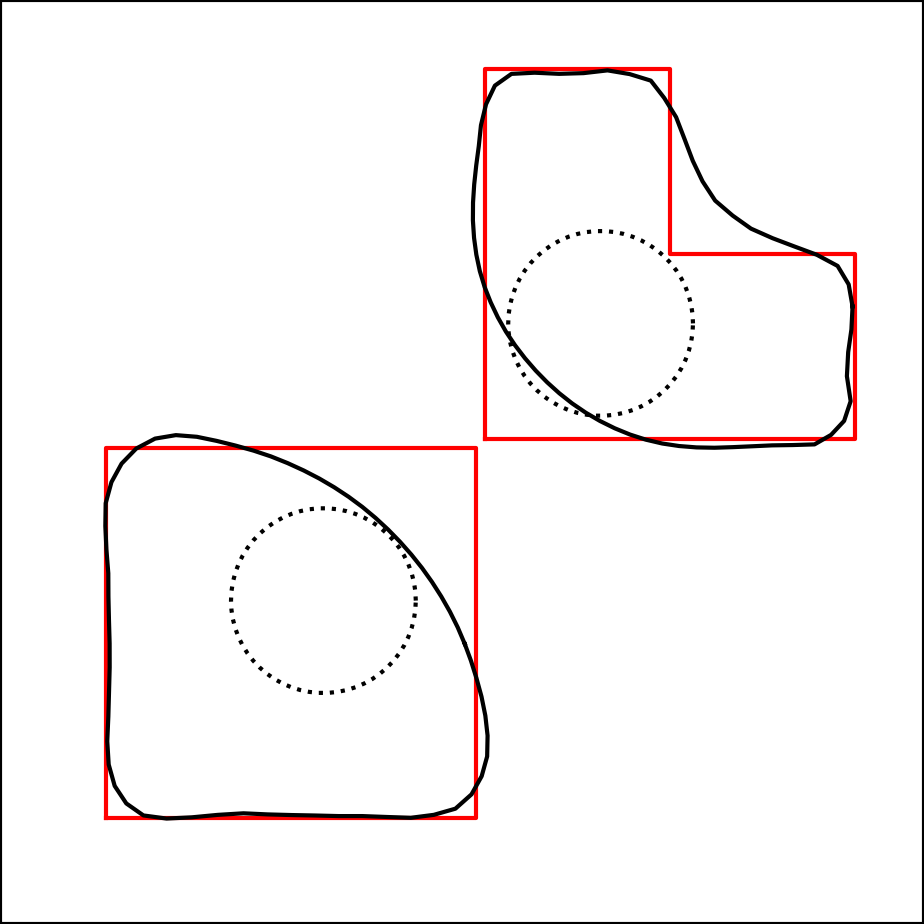}
\caption{}
\label{fig:beta200_img3}
\end{subfigure}
\hfill
\begin{subfigure}{0.18\textwidth}
\includegraphics[width=\textwidth]{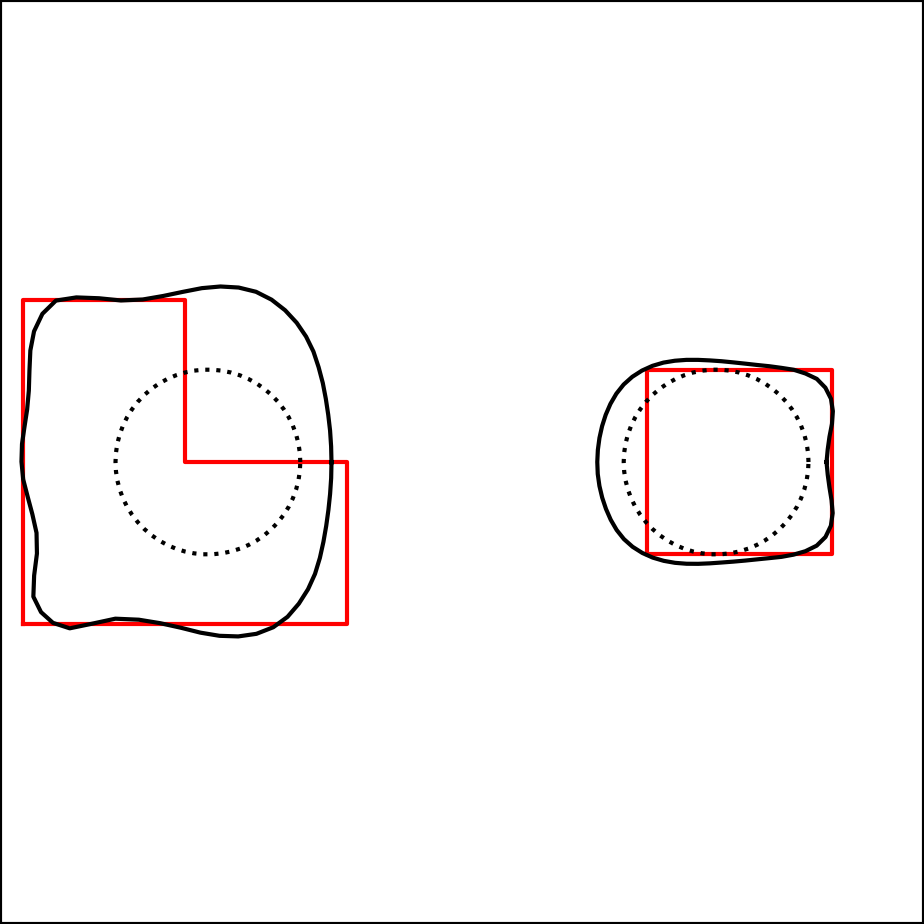}
\caption{}
\label{fig:beta200_img4}
\end{subfigure}
\hfill
\begin{subfigure}{0.18\textwidth}
\includegraphics[width=\textwidth]{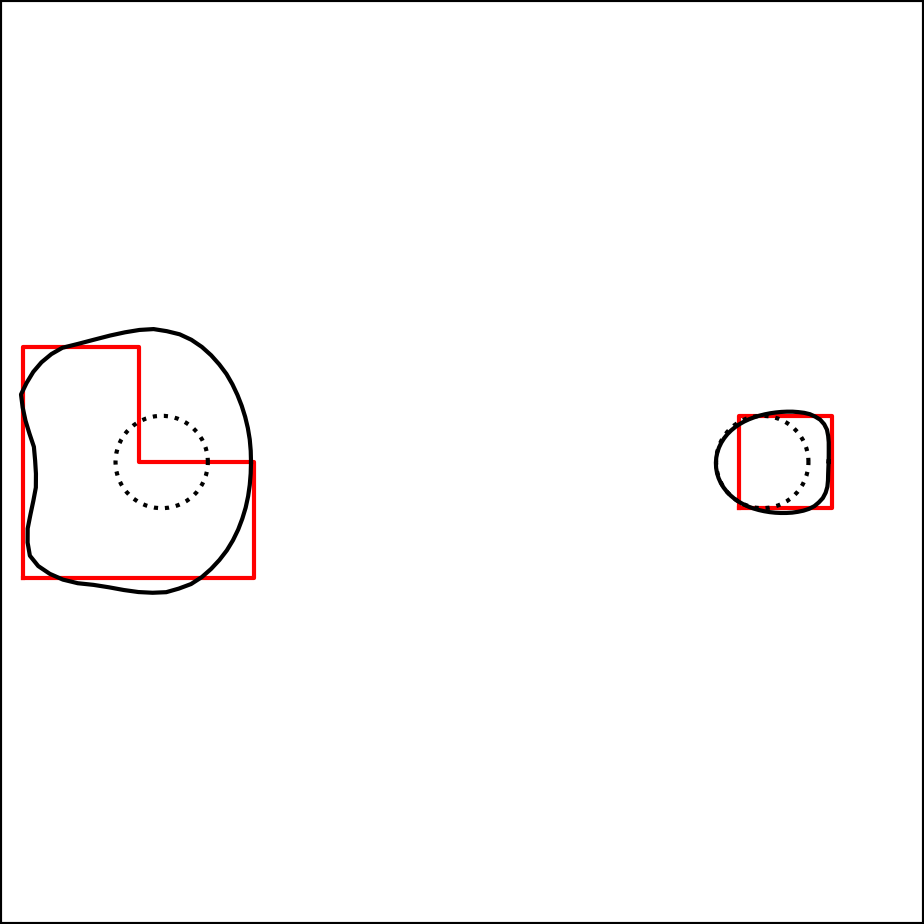}
\caption{}
\label{fig:beta200_img5}
\end{subfigure}

\vspace{0.5em}

\begin{subfigure}{0.18\textwidth}
\includegraphics[width=\textwidth]{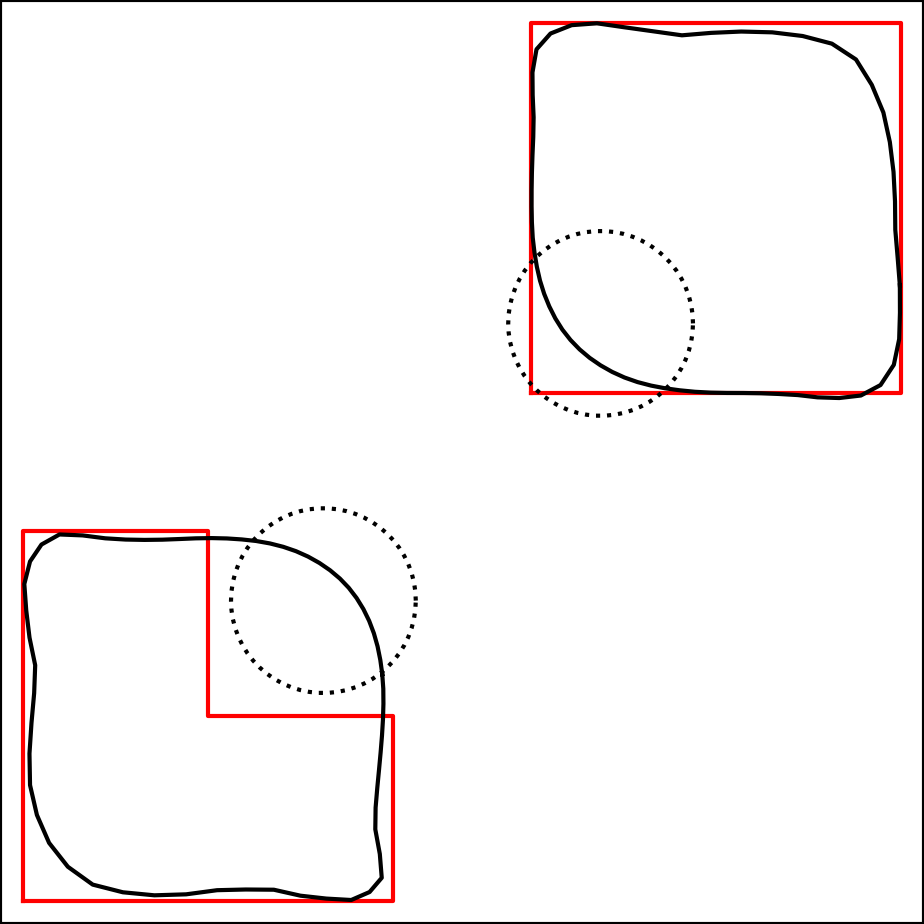}
\caption{}
\label{fig:beta200_img6}
\end{subfigure}
\hfill
\begin{subfigure}{0.18\textwidth}
\includegraphics[width=\textwidth]{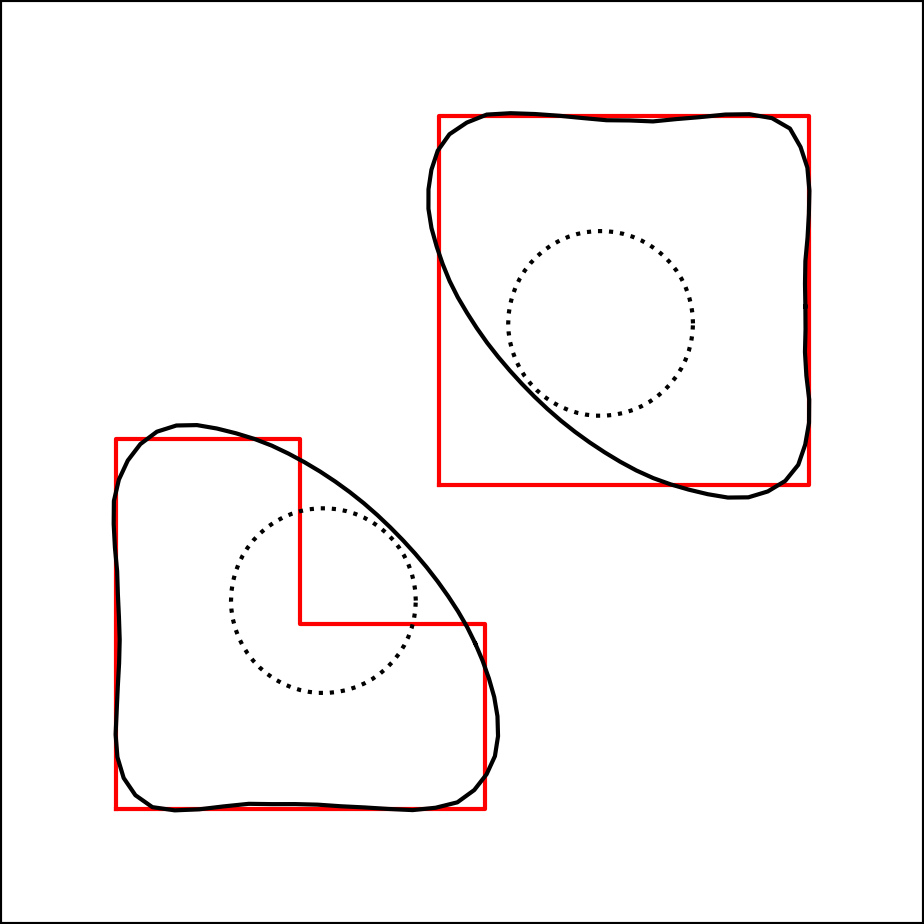}
\caption{}
\label{fig:beta200_img7}
\end{subfigure}
\hfill
\begin{subfigure}{0.18\textwidth}
\includegraphics[width=\textwidth]{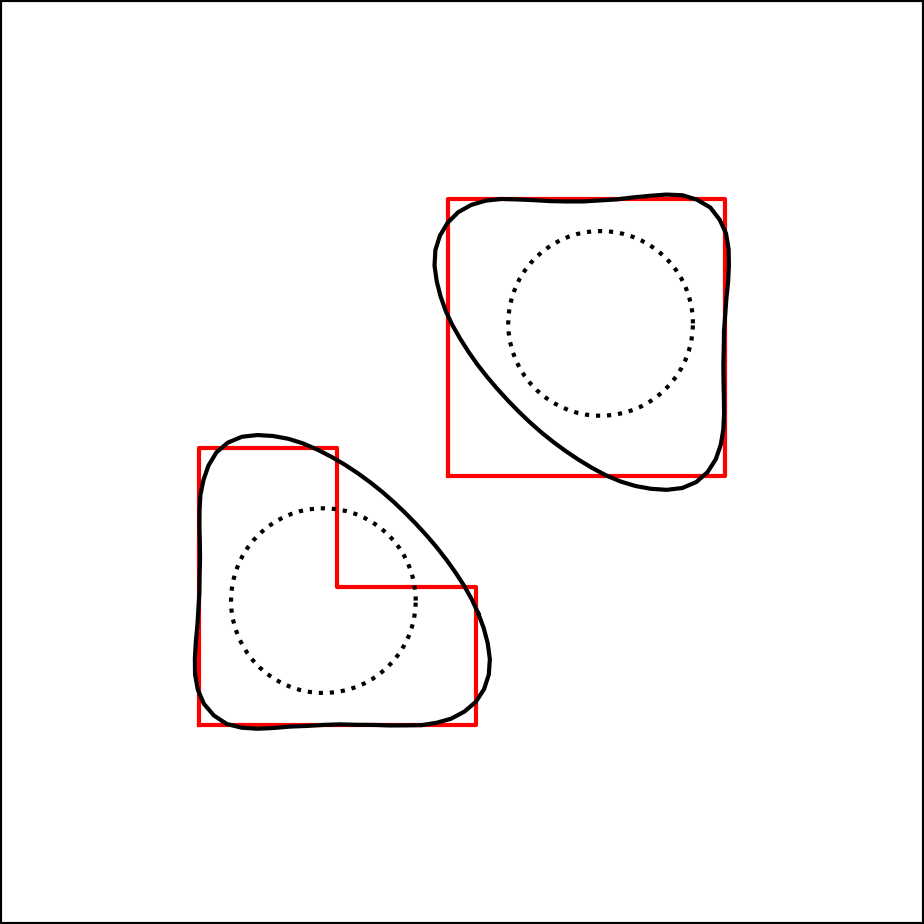}
\caption{}
\label{fig:beta200_img8}
\end{subfigure}
\hfill
\begin{subfigure}{0.18\textwidth}
\includegraphics[width=\textwidth]{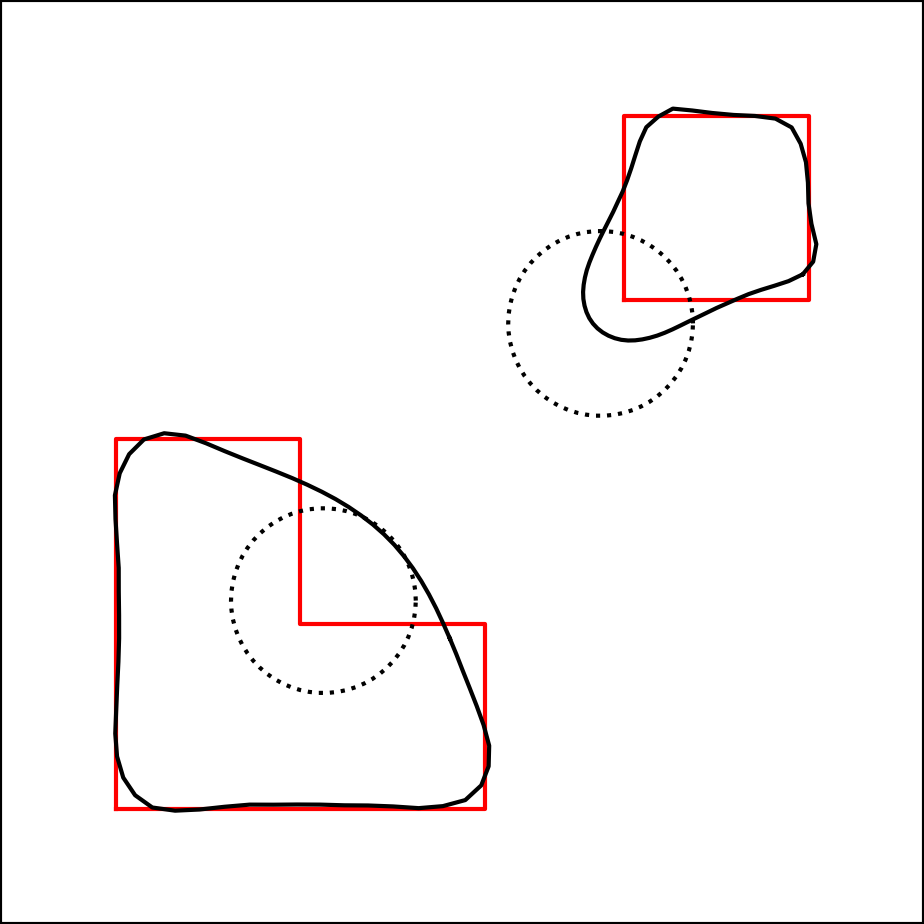}
\caption{}
\label{fig:beta200_img9}
\end{subfigure}
\hfill
\begin{subfigure}{0.18\textwidth}
\includegraphics[width=\textwidth]{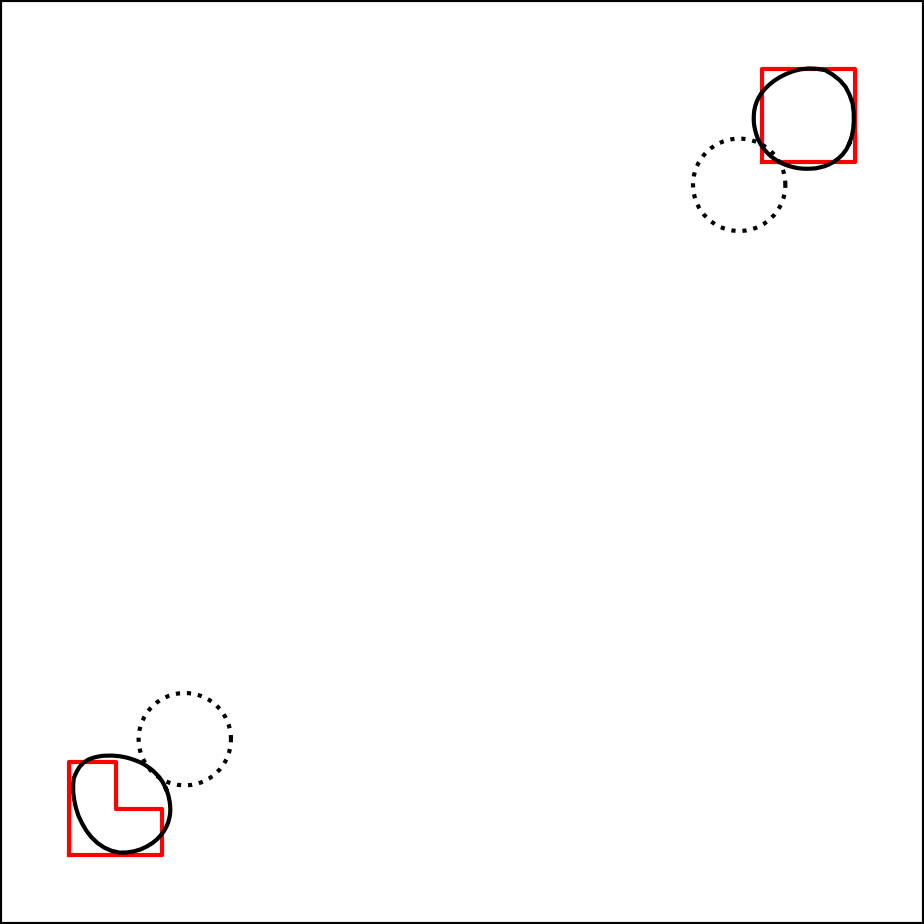}
\caption{}
\label{fig:beta200_img10}
\end{subfigure}

\vspace{0.5em}

\begin{subfigure}{0.18\textwidth}
\includegraphics[width=\textwidth]{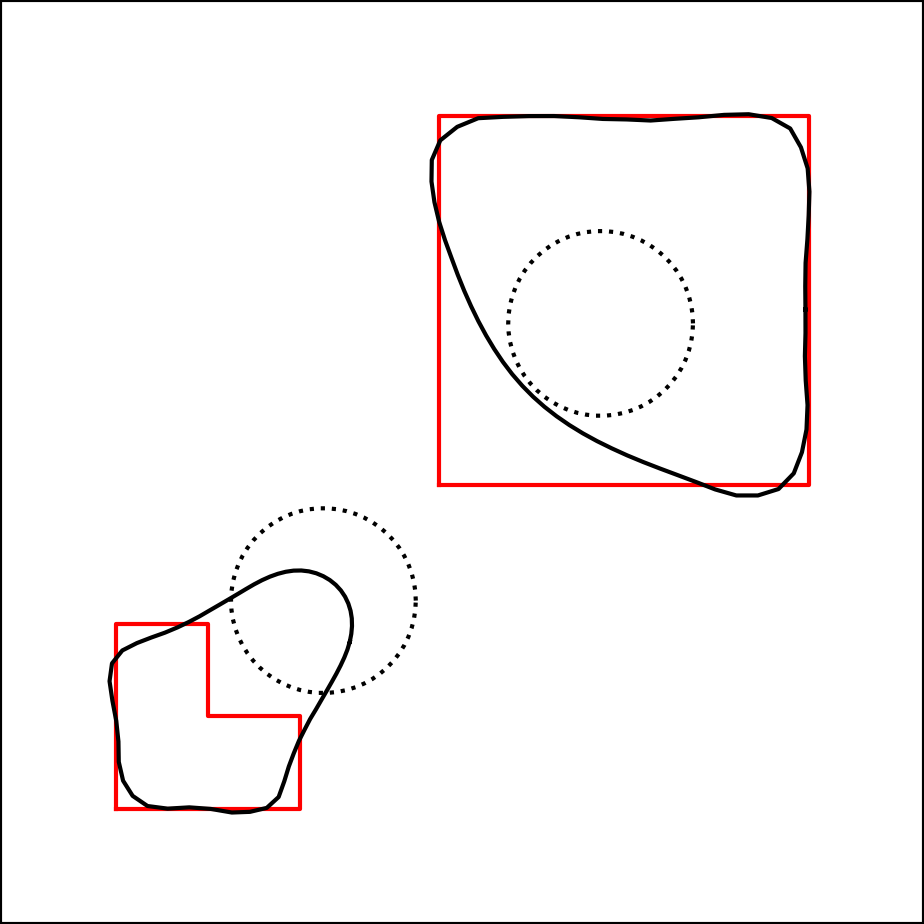}
\caption{}
\label{fig:beta200_img11}
\end{subfigure}
\hfill
\begin{subfigure}{0.18\textwidth}
\includegraphics[width=\textwidth]{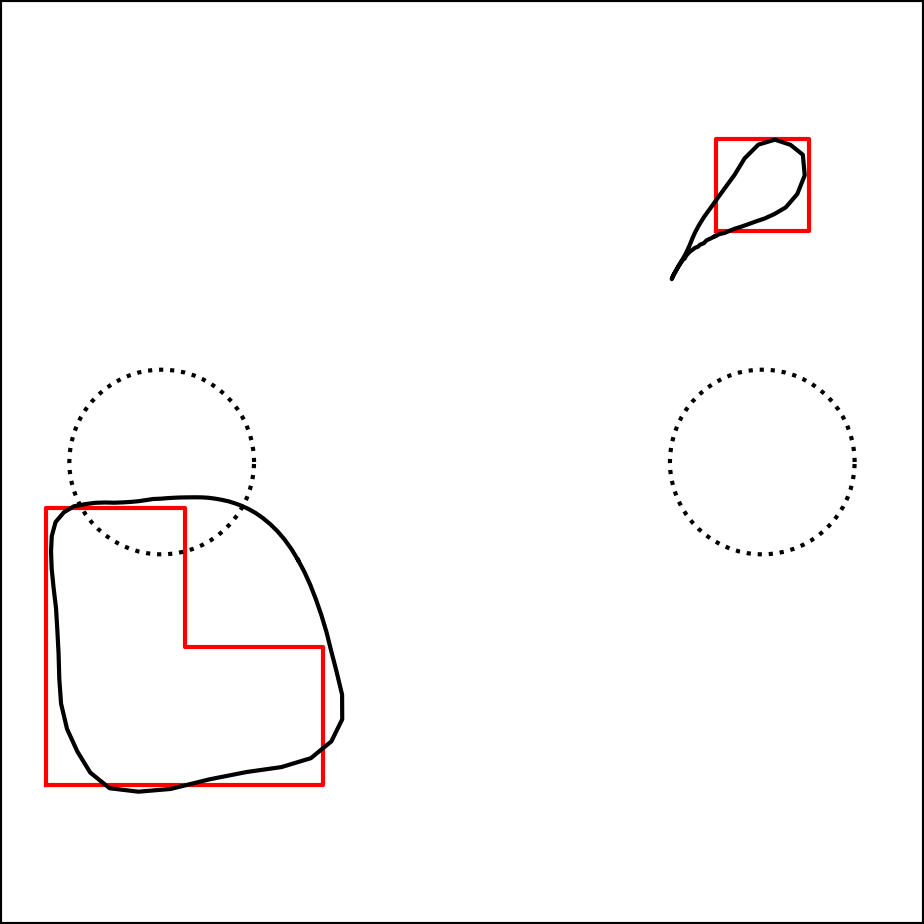}
\caption{}
\label{fig:beta200_img12}
\end{subfigure}
\hfill
\begin{subfigure}{0.18\textwidth}
\includegraphics[width=\textwidth]{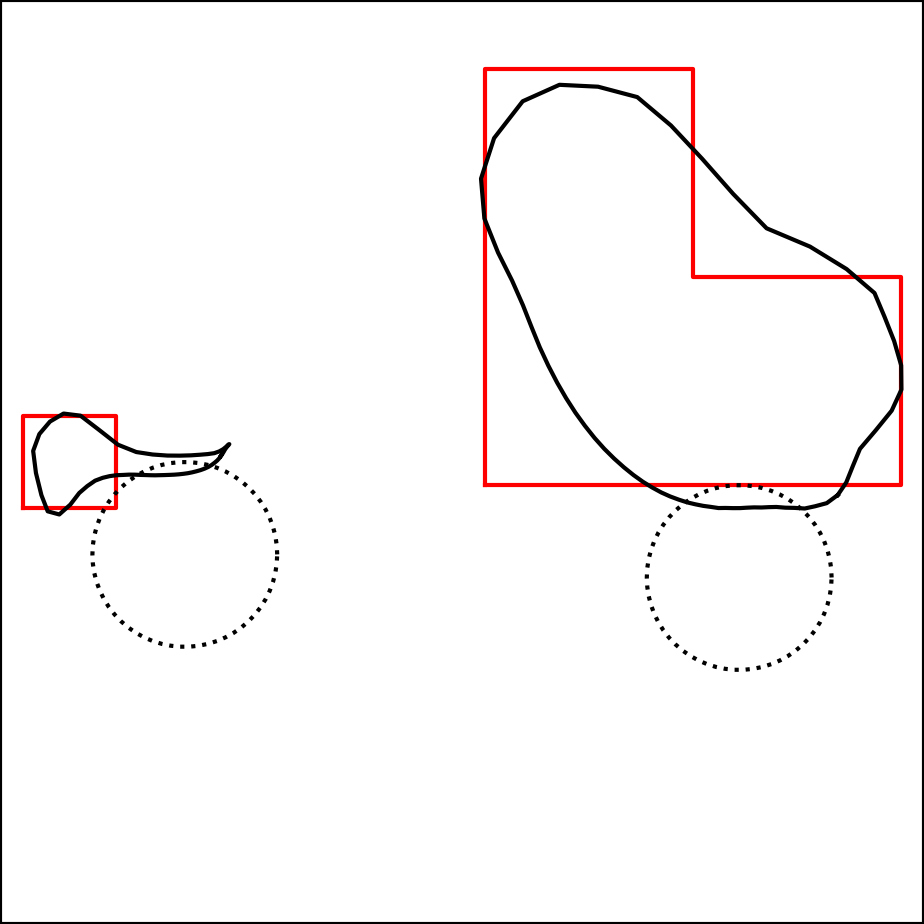}
\caption{}
\label{fig:beta200_img13}
\end{subfigure}
\hfill
\begin{subfigure}{0.18\textwidth}
\includegraphics[width=\textwidth]{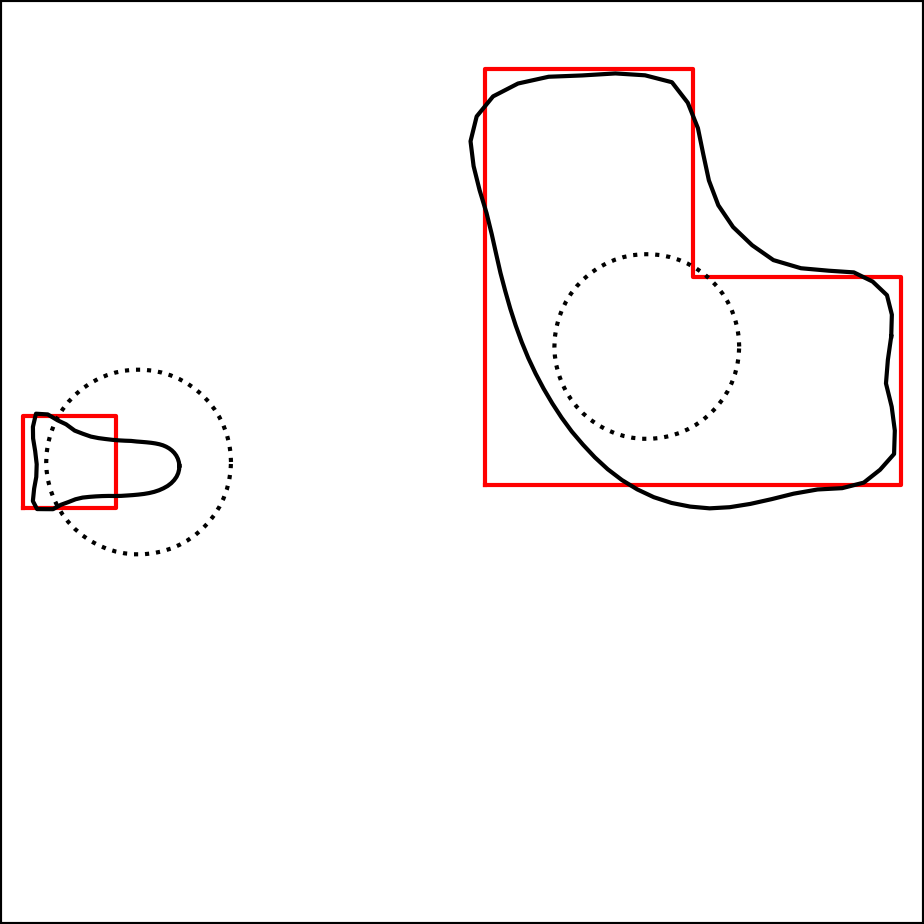}
\caption{}
\label{fig:beta200_img14}
\end{subfigure}
\hfill
\begin{subfigure}{0.18\textwidth}
\includegraphics[width=\textwidth]{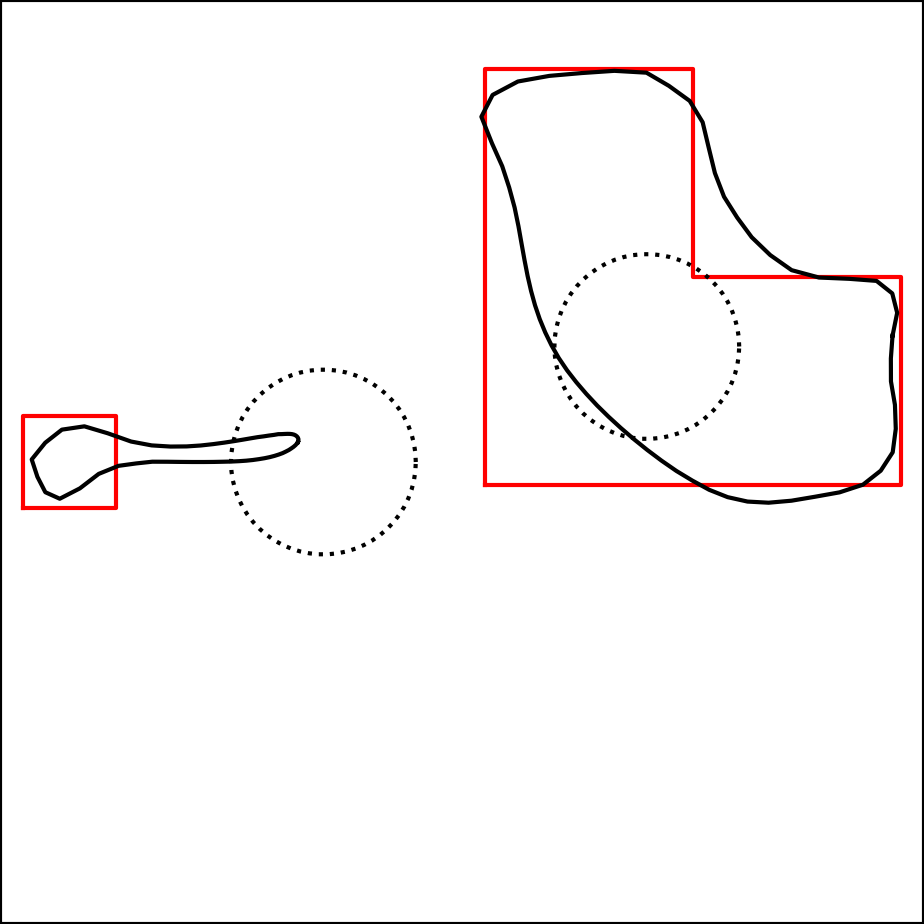}
\caption{}
\label{fig:beta200_img15}
\end{subfigure}

\caption{Reconstruction results under noisy measurements ($\delta = 0.1$) with $\beta = 200$.}
\label{fig:beta200_all}
\end{figure}
 
\section{Conclusions}
\label{sec:conclusions}

We have developed a statistically robust framework for reconstructing metal--semiconductor contact regions using topological gradients. 
The approach combines deterministic and statistical analyses: the topological gradient is shown to be stable with respect to measurement noise, and a central limit theorem for the empirical gradient enables the construction of confidence intervals and hypothesis tests for contact detection. 
Monte Carlo sampling captures both the mean behavior and variability of the gradient, providing spatially resolved confidence zones and highlighting regions of strongest contact. 
Subsequent shape optimization with Sobolev-regularized gradients refines the reconstruction, yielding smooth and accurate interfaces that respect the underlying geometry. 
The free parameter $\beta$ in the CCBM formulation further enhances the shape optimization stage, improving geometric accuracy and the resolution of complex interface features, particularly under noisy measurements. 
Numerical experiments confirm that this integrated methodology reliably identifies location, size, and topology of inclusions, while effectively distinguishing true structural features from noise-induced artifacts.

\appendix
\section{Proof of Lemma~\ref{lem:norm_estimate}}\label{appx:norm_estimate}
First, $\norm{\cdot}_{\ast}$ is a norm on $\HH$: if $\norm{\varphi}_{\ast} = 0$, then $\nabla \varphi = 0_{\mathbb{R}^{d}}$ in $\varOmega$ and $\varphi = 0$ in $\omega$. 
Hence $\varphi$ is constant in $\varOmega$ and vanishes on $\omega$, so $\varphi \equiv 0$ in $\varOmega$.  

Suppose the inequality fails. Then there exists a sequence $\{\varphi_n\} \subset \HH$ with 
\[
\norm{\varphi_n}_{1, \varOmega} = 1, \quad \norm{\varphi_n}_{\ast} \to 0 \text{ as } n \to \infty.
\]
By the compact embedding $\HH \hookrightarrow L^2(\varOmega)$, there exists a subsequence (still denoted $\varphi_n$) such that $\varphi_n \to \varphi$ in $L^2(\varOmega)$ and $\nabla \varphi_n \rightharpoonup \nabla \varphi$ weakly in $L^2(\varOmega)^{d}$. The convergence $\norm{\nabla \varphi_n}_{0, \varOmega} \to 0$ implies $\nabla \varphi = 0_{\mathbb{R}^{d}}$, so $\varphi$ is constant. Moreover, $\norm{\varphi_n}_{0, \omega} \to 0$ implies $\varphi = 0$ in $\omega$, hence $\varphi \equiv 0$ in $\varOmega$, contradicting $\norm{\varphi_n}_{1, \varOmega} = 1$.  
Therefore, there exists $C(\varOmega) > 0$ such that $\norm{\varphi}_{1, \varOmega} \le C(\varOmega) \, \norm{\varphi}_{\ast}$, which concludes the proof.

\section*{Acknowledgments}
The work of JFTR is supported by the JSPS Postdoctoral Fellowships for Research in Japan (Grant Number JP24KF0221), and partially by the JSPS Grant-in-Aid for Early-Career Scientists (Grant Number JP23K13012) and the JST CREST (Grant Number JPMJCR2014).

\bibliographystyle{siamplain}
\bibliography{main}

\end{document}